\documentclass[reqno]{amsart}

\usepackage{amssymb}
\usepackage{graphicx}
\usepackage{amscd}
\usepackage[hidelinks]{hyperref}
\usepackage{color}
\usepackage{float}
\usepackage{graphics,amsmath,amssymb}
\usepackage{amsthm}
\usepackage{amsfonts}
\usepackage{latexsym}
\usepackage{epsf}
\usepackage{xifthen}
\usepackage{mathrsfs}
\usepackage{dsfont}
\usepackage{makecell}
\usepackage[FIGTOPCAP]{subfigure}
\usepackage{amsmath}
\allowdisplaybreaks[4]
\usepackage{listings}
\usepackage{etoolbox}
\usepackage{fancyhdr}
\usepackage{pdflscape}
\usepackage[title,toc,titletoc]{appendix}
\usepackage{enumitem}
\usepackage[noadjust]{cite}
\usepackage{forest}
\usepackage{multirow}

\hypersetup{backref=true}

 \newtheoremstyle{mytheorem}
 {3pt}
 {3pt}
 {\slshape}
 {}
 {\bfseries}
 {.}
 { }
 {}

\numberwithin{equation}{section}

\theoremstyle{theorem}
\newtheorem{theorem}{Theorem}[section]
\newtheorem*{theorem*}{Theorem}

\newtheorem{corollary}[theorem]{Corollary}
\newtheorem{lemma}[theorem]{Lemma}

\newtheorem{innercustomgeneric}{\customgenericname}
\providecommand{\customgenericname}{}
\newcommand{\newcustomtheorem}[2]{%
	\newenvironment{#1}[1]
	{%
		\renewcommand\customgenericname{#2}%
		\renewcommand\theinnercustomgeneric{##1}%
		\innercustomgeneric
	}
	{\endinnercustomgeneric}
}
\newcustomtheorem{ctheorem}{Theorem}
\newcustomtheorem{clemma}{Lemma}

\theoremstyle{definition}

\newtheorem*{example*}{Example}
\newtheorem*{examples*}{Examples}

\newtheorem*{remark*}{Remark}
\newtheorem*{remarks*}{Remarks}

\newtheoremstyle{named}{}{}{\itshape}{}{\bfseries}{.}{.5em}{#1\thmnote{ #3}}
\theoremstyle{named}

\newtheorem*{namedentry}{Entry}

\newcommand{\Keywords}[1]{\ifthenelse{\isempty{#1}}{}{\smallskip \smallskip \noindent \textbf{Keywords}. #1}}
\newcommand{\MSC}[2][2020]{\ifthenelse{\isempty{#2}}{}{\smallskip \smallskip \noindent \textbf{#1MSC}. #2}}
\newcommand{\abstractnote}[1]{\ifthenelse{\isempty{#1}}{}{\smallskip \smallskip \noindent \textsuperscript{\dag}#1}}

\makeatletter
\def\specialsection{\@startsection{section}{1}%
  \z@{\linespacing\@plus\linespacing}{.5\linespacing}%
  {\normalfont}}
\def\section{\@startsection{section}{1}%
  \z@{.7\linespacing\@plus\linespacing}{.5\linespacing}%
  {\normalfont\scshape}}
\patchcmd{\@settitle}{\uppercasenonmath\@title}{\Large\boldmath}{}{}
\patchcmd{\@settitle}{\begin{center}}{\begin{flushleft}}{}{}
\patchcmd{\@settitle}{\end{center}}{\end{flushleft}}{}{}
\patchcmd{\@setauthors}{\MakeUppercase}{\normalsize}{}{}
\patchcmd{\@setauthors}{\centering}{\raggedright}{}{}
\patchcmd{\section}{\scshape}{\large\bfseries\boldmath}{}{}
\patchcmd{\subsection}{\bfseries}{\bfseries\boldmath}{}{}
\renewcommand{\@secnumfont}{\bfseries}
\patchcmd{\@startsection}{\@afterindenttrue}{\@afterindentfalse}{}{}
\patchcmd{\abstract}{\leftmargin3pc}{\leftmargin1pc}{}{}

\def\maketitle{\par
  \@topnum\z@ 
  \@setcopyright
  \thispagestyle{empty}
  \ifx\@empty\shortauthors \let\shortauthors\shorttitle
  \else \andify\shortauthors
  \fi
  \@maketitle@hook
  \begingroup
  \@maketitle
  \toks@\@xp{\shortauthors}\@temptokena\@xp{\shorttitle}%
  \toks4{\def\\{ \ignorespaces}}
  \edef\@tempa{%
    \@nx\markboth{\the\toks4
      \@nx\MakeUppercase{\the\toks@}}{\the\@temptokena}}%
  \@tempa
  \endgroup
  \c@footnote\z@
  \@cleartopmattertags
}
\makeatother

\newcommand{\sL}{\mathscr{L}}
\newcommand{\sR}{\mathscr{R}}
\newcommand{\sH}{\mathscr{H}}
\newcommand{\sM}{\mathscr{M}}
\newcommand{\sN}{\mathscr{N}}

\newcommand{\sP}{\mathscr{P}}
\newcommand{\sQ}{\mathscr{Q}}
\newcommand{\sF}{\mathscr{F}}
\newcommand{\sG}{\mathscr{G}}

\newcommand{\sA}{\mathscr{A}}
\newcommand{\sB}{\mathscr{B}}

\newcommand{\tA}{\tilde{A}}
\newcommand{\tB}{\tilde{B}}

\newcommand{\hm}{\hat{m}}
\newcommand{\hn}{\hat{n}}
\newcommand{\hw}{\hat{w}}

\newcommand{\chm}{\check{m}}
\newcommand{\chn}{\check{n}}
\newcommand{\chw}{\check{w}}

\newcommand{\cI}{\mathcal{I}}

\title[General coefficient-vanishing results]{General coefficient-vanishing results associated with theta series}

\author[S. Chern]{Shane Chern}
\address[S. Chern]{Department of Mathematics and Statistics, Dalhousie University, Halifax, Nova Scotia, B3H 4R2, Canada}
\email{chenxiaohang92@gmail.com}

\author[D. Tang]{Dazhao Tang}
\address[D. Tang]{School of Mathematic Sciences, Chongqing Normal University, Chongqing 401331, P.R. China}
\email{dazhaotang@sina.com}

\date{}

\begin{document}

\maketitle

\sloppy

\begin{abstract}

There are a number of sporadic coefficient-vanishing results associated with theta series, which suggest certain underlying patterns. By expanding theta powers as linear combinations of products of theta functions, we present two strategies that will provide a unified treatment. Our approaches rely on studying the behavior of products of two theta series under the action of the huffing operator. For this purpose, some explicit criteria are given. We may use the presented methods to not only verify experimentally discovered coefficient-vanishing results, but also to produce a series of general phenomena.


\MSC{11F27, 11B65, 11J13, 11J20.}
\end{abstract}

\section{Introduction}

Given a Laurent series $G(q)=\sum_{n}g_nq^n$, a particularly interesting problem is to determine if there exists an arithmetic progression $Mn+w$ such that the coefficients of $G(q)$ indexed by this arithmetic progression vanish, namely, $g_{Mn+w}=0$. If so, we say $G(q)$ possesses the \emph{coefficient-vanishing} property.

For $G(q)$ a quotient of several \emph{Ramanujan's theta series},
\begin{align}\label{eq:f(a,b)-summation}
	f(a,b) := \sum_{n\in \mathbb{Z}}a^{n(n+1)/2}b^{n(n-1)/2},
\end{align}
the study of its vanishing coefficients has a long history. In the
literature, such theta quotients are usually expressed in terms of the following product form of $f(a,b)$ in light of Jacobi's triple product identity \cite[p.~35, Entry 19]{Ber1991}:
\begin{align}\label{eq:f(a,b)-product}
	f(a,b) = (-a,-b,ab;ab)_\infty,
\end{align}
where the \emph{$q$-Pochhammer symbols} are defined by
\begin{align*}
	(A;q)_\infty &:= \prod_{k\ge 0}(1-Aq^k),\\
	(A_1,\ldots,A_\ell;q)_\infty &:= (A_1;q)_\infty \cdots (A_\ell;q)_\infty.
\end{align*}

The very first consideration along this line was initiated by Richmond and Szekeres \cite{RS1978}, who proved with the use of the circle method that $h(4n+2)=0$ for the \emph{Ramanujan--G\"ollnitz--Gordon continued fraction}:
\begin{align}\label{eq:RS}
	\sum_{n\ge 0} h(n)q^n:=\dfrac{(q,q^7;q^8)_\infty}{(q^3,q^5;q^8)_\infty}.
\end{align}
Their result was subsequently extended by Andrews and Bressoud \cite{AB1979} to arbitrary moduli with recourse to Ramanujan's ${}_{1}\psi_{1}$ summation formula. Through a similar analysis, a more general result was discovered by Alladi and Gordon \cite{AG1994}. Other results sharing the same nature can be found in Chan and Yesilyurt \cite{CY2006}, Mc Laughlin \cite{Mc20L15,McL21b}, Tang \cite{Tang2019b}, Chern and Tang \cite{CT2020,CT2021}, and Du and Tang \cite{DT2022}.

Another important source of coefficient-vanishing properties associated with theta series is a recent paper of Hirschhorn \cite{Hir2019}, in which it was shown that $\gamma_{1,1,5}(5n+2)=\gamma_{1,1,5}(5n+4)=0$ where
\begin{align*}
	\sum_{n\ge 0}\gamma_{1,1,5}(n)q^n:=(-q,-q^4;q^5)_\infty (q,q^9;q^{10})_\infty^3.
\end{align*}
The basic idea of Hirschhorn's approach relies on reformulating this $q$-product back to a sum-like expression by \eqref{eq:f(a,b)-summation} and then making a suitable change of variables for the sum indices. After some primary considerations due to Tang \cite{Tang2019} and Baruah and Kaur \cite{BK2020}, Mc Laughlin \cite{McL2021} moved a huge step forward by establishing new results under various moduli. For other related results, see, for instance, Mc Laughlin and Zimmer \cite{MZ2022}, and Kuar and Vandna \cite{VK2022,VK2022b}.

To tackle the coefficient-vanishing puzzle associated with theta series, various approaches were applied, including the circle method \cite{RS1978}, $q$-hypergeometric transformations \cite{AB1979,AG1994}, computer algebra \cite{XZ2020}, and explicit series dissections \cite{DX2021,Hir2002,Lin2013,Tang2021,TX2020}. However, a particularly powerful method is the one that was first utilized by Andrews and Bressoud in \cite{AB1979}, relying on subtle substitutions of sum indices for the summation form of theta series. This idea will be the main ingredient in our work.

Moreover, in a series of papers of Tang \cite{Tang2022a,Tang2022b,Tang2023,Tang2022c}, it was discovered that the previous coefficient-vanishing results on theta series are only the tip of the iceberg. More precisely, Tang considered five families of coefficient functions:
\begin{align}
	\sum_{n}\alpha_{i,j,r,\ell,m}(n)q^n &:= \frac{(q^{i},q^{r-i};q^{r})_\infty^\ell}{(q^{j},q^{r-j};q^{r})_\infty^m},\label{eq:def-alpha}\\
	\sum_{n}\beta_{i,j,r,\ell,m}(n)q^n &:= \frac{(q^{i},q^{r-i};q^{r})_\infty^\ell}{(-q^{j},-q^{r-j};q^{r})_\infty^m},\label{eq:def-beta}\\
	\sum_{n}\gamma_{i,j,r,\ell,m}(n)q^n &:= (-q^{i},-q^{r-i};q^{r})_\infty^\ell (q^{j},q^{2r-j};q^{2r})_\infty^m,\label{eq:def-gamma}\\
	\sum_{n}\delta_{i,j,r,\ell,m}(n)q^n &:= (q^{i},q^{r-i};q^{r})_\infty^\ell (-q^{j},-q^{2r-j};q^{2r})_\infty^m,\label{eq:def-delta}\\
	\sum_{n}\epsilon_{i,j,r,\ell,m}(n)q^n &:= (q^{i},q^{r-i};q^{r})_\infty^\ell (q^{j},q^{2r-j};q^{2r})_\infty^m.\label{eq:def-epsilon}
\end{align}
It was observed by him that the coefficient-vanishing property still holds for a vast number of choices of $i$, $j$ and $r$ even if the powers $\ell$ and $m$ vary. In particular, dozens of such relations were established in an explicit way, while a list of general conjectures was further proposed. These observations stimulate our investigation from a broader setting.

Define the \emph{huffing operator} \cite[Eq.~(19.4.7)]{Hir2019} for $G(q)=\sum_{n}g_nq^n$ a Laurent series and $M$ a positive integer by
\begin{align}\label{eq:U-operator}
	\mathbf{H}_M\big(G(q)\big):=\sum_{n} g_{Mn}q^{Mn}.
\end{align}
Then $g_{Mn+w}=0$ is equivalent to
\begin{align}
	\mathbf{H}_M\big(q^{-w}G(q)\big) = 0.
\end{align}

We first state our coefficient-vanishing results with the above $\mathbf{H}$-operator notation.

\begin{theorem}\label{th:Type-I}
	Let $\mu\ge 1$, $\ell\ge 0$ and $m\ge 0$ be integers.
	\begin{enumerate}[label=\textup{(\roman*)}, widest=iii, itemindent=*, leftmargin=*]
		\item If $M=2\ell+6m+3$ and $\sigma=-(2\ell+4m+2)k$, then for $\kappa\in\{0,1\}$, $\lambda\in\{0,1\}$ and any $k$ such that $\gcd(k,M)=1$, 
		\begin{align}\label{eq:I-1-e-o}
			\mathbf{H}_{M}\Bigg(q^{\sigma}\cdot f\big((-1)^{\kappa}q^{k},(-1)^{\kappa}q^{\mu M-k}\big)^{2\ell} \bigg(\frac{f\big({-q^{2k}},-q^{\mu M-2k}\big)}{f\big((-1)^{\lambda}q^{k},(-1)^{\lambda}q^{\mu M-k}\big)}\bigg)^{2m+1}\Bigg)=0.
		\end{align}
		
		\item If $M=8\ell+6m+3$ and $\sigma=-(6\ell+4m+2)k$, then for $\kappa\in\{0,1\}$, $\lambda\in\{0,1\}$ and any $k$ such that $\gcd(k,M)=1$, 
		\begin{align}\label{eq:I-2-e-o}
			\mathbf{H}_{M}\Bigg(q^{\sigma}\cdot f\big((-1)^{\kappa}q^{2k},(-1)^{\kappa}q^{\mu M-2k}\big)^{2\ell} \bigg(\frac{f\big({-q^{2k}},-q^{\mu M-2k}\big)}{f\big((-1)^{\lambda}q^{k},(-1)^{\lambda}q^{\mu M-k}\big)}\bigg)^{2m+1}\Bigg)=0.
		\end{align}
		
		\item If $M=4\ell+6m+5$ and $\sigma=-(2\ell+2m+2)k$, then for $\kappa\in\{1\}$, $\lambda\in\{0,1\}$ and any $k$ such that $\gcd(k,M)=1$, 
		\begin{align}\label{eq:I-2-o-e}
			\mathbf{H}_{M}\Bigg(q^{\sigma}\cdot f\big((-1)^{\kappa}q^{2k},(-1)^{\kappa}q^{\mu M-2k}\big)^{2\ell+1} \bigg(\frac{f\big({-q^{2k}},-q^{\mu M-2k}\big)}{f\big((-1)^{\lambda}q^{k},(-1)^{\lambda}q^{\mu M-k}\big)}\bigg)^{4m+2}\Bigg)=0.
		\end{align}
	\end{enumerate}
\end{theorem}

\begin{theorem}\label{th:Type-I.2}
	Let $\mu\ge 1$, $\ell\ge 0$ and $m\ge 0$ be integers.
	\begin{enumerate}[label=\textup{(\roman*)}, widest=ii, itemindent=*, leftmargin=*]
		\item
		If $M=4\ell+2m+3$ and $\sigma=-(3\ell+m+2)k$, then for $(\kappa,\lambda)\in\{(0,1),(1,0)\}$ and any $k$ such that $\gcd(k,M)=1$, 
		\begin{align}\label{eq:I.2}
			\mathbf{H}_{M}\Big(q^{\sigma}\cdot f\big((-1)^{\kappa}q^{k},(-1)^{\kappa}q^{\mu M-k}\big)^{2\ell+1} f\big((-1)^{\lambda}q^{\mu M+k},(-1)^{\lambda}q^{\mu M-k}\big)^{2m+1}\Big)=0.
		\end{align}
		
		\item
		If $M=2\ell+4m+3$ and $\sigma=-(2\ell+2m+2)k$, then for $(\kappa,\lambda)\in\{(0,1),(1,1)\}$ and any $k$ such that $\gcd(k,M)=1$, 
		\begin{align}\label{eq:I.2-2}
			\mathbf{H}_{M}\Big(q^{\sigma}\cdot f\big((-1)^{\kappa}q^{k},(-1)^{\kappa}q^{\mu M-k}\big)^{2\ell+1} f\big((-1)^{\lambda}q^{\mu M+2k},(-1)^{\lambda}q^{\mu M-2k}\big)^{2m+1}\Big)=0.
		\end{align}
	\end{enumerate}
\end{theorem}

\begin{theorem}\label{th:Type-I.3}
	Let $\mu\ge 1$, $\ell\ge 0$ and $m\ge 0$ be integers.
	\begin{enumerate}[label=\textup{(\roman*)}, widest=iii, itemindent=*, leftmargin=*]
		\item
		If $M=4\ell+2m+3$ and $\sigma=-(\ell+m+1)k$, then for $(\kappa,\lambda)\in\{(0,1),(1,0)\}$ and any $k$ such that $\gcd(k,M)=1$, 
		\begin{align}\label{eq:I.3-1}
			\mathbf{H}_{M}\Big(q^{\sigma}\cdot f\big((-1)^{\kappa}q^{k},(-1)^{\kappa}q^{\mu M-k}\big)^{2\ell+1} f\big((-1)^{\lambda}q^{k},(-1)^{\lambda}q^{2\mu M-k}\big)^{2m+1}\Big)=0.
		\end{align}
		
		\item
		If $M=2\ell+16m+9$ and $\sigma=-(2\ell+12m+7)k$, then for $(\kappa,\lambda)\in\{(0,1),(1,1)\}$ and any $k$ such that $\gcd(k,M)=1$, 
		\begin{align}\label{eq:I.3-2}
			\mathbf{H}_{M}\Big(q^{\sigma}\cdot f\big((-1)^{\kappa}q^{k},(-1)^{\kappa}q^{\mu M-k}\big)^{2\ell+1} f\big((-1)^{\lambda}q^{4k},(-1)^{\lambda}q^{2\mu M-4k}\big)^{2m+1}\Big)=0.
		\end{align}
		
		\item
		If $M=16\ell+2m+9$ and $\sigma=-(10\ell+2m+6)k$, then for $(\kappa,\lambda)\in\{(1,0),(1,1)\}$ and any $k$ such that $\gcd(k,M)=1$, 
		\begin{align}\label{eq:I.3-3}
			\mathbf{H}_{M}\Big(q^{\sigma}\cdot f\big((-1)^{\kappa}q^{2k},(-1)^{\kappa}q^{\mu M-2k}\big)^{2\ell+1} f\big((-1)^{\lambda}q^{k},(-1)^{\lambda}q^{2\mu M-k}\big)^{2m+1}\Big)=0.
		\end{align}
	\end{enumerate}
\end{theorem}

\begin{theorem}\label{th:Type-II}
	Let $\mu\ge 1$, $\ell\ge 0$ and $m\ge 0$ be integers.
	\begin{enumerate}[label=\textup{(\roman*)}, widest=iii, itemindent=*, leftmargin=*]
		\item
		Assume that $\gcd(2\ell+1,2m+1)=1$. If $M=2\ell+4m+3$ and $\sigma=2(2m+1)^2 k$, then for $(\kappa,\lambda)\in\{(0,1),(1,0)\}$ and any $k$ such that $\gcd(k,M)=1$, 
		\begin{align}\label{eq:II-1}
			\mathbf{H}_{M}\Big(q^{\sigma}&\cdot f\big((-1)^{\kappa}q^{(2m+1)k},(-1)^{\kappa}q^{\mu M-(2m+1)k}\big)^{2\ell+1}\notag\\
			&\times f\big((-1)^{\lambda}q^{(2\ell+1)k},(-1)^{\lambda}q^{2\mu M-(2\ell+1)k}\big)^{2m+1}\Big)=0.
		\end{align}
		
		\item
		Assume that $\gcd(2\ell+1,2m+2)=1$. If $M=2\ell+4m+5$ and $\sigma=2(2m+2)^2 k$, then for $(\kappa,\lambda)\in\{(1,0),(1,1)\}$ and any $k$ such that $\gcd(k,M)=1$,  
		\begin{align}\label{eq:II-2}
			\mathbf{H}_{M}\Big(q^{\sigma}&\cdot f\big((-1)^{\kappa}q^{(2m+2)k},(-1)^{\kappa}q^{\mu M-(2m+2)k}\big)^{2\ell+1}\notag\\
			&\times f\big((-1)^{\lambda}q^{(2\ell+1)k},(-1)^{\lambda}q^{2\mu M-(2\ell+1)k}\big)^{2m+2}\Big)=0.
		\end{align}
		
		\item
		Assume that $\gcd(2\ell+2,2m+1)=1$. If $M=4\ell+2m+5$ and $\sigma=3(2\ell+2)^2 k$, then for $(\kappa,\lambda)\in\{(0,1),(1,1)\}$ and any $k$ such that $\gcd(k,M)=1$, 
		\begin{align}\label{eq:II-3}
			\mathbf{H}_{M}\Big(q^{\sigma}&\cdot f\big((-1)^{\kappa}q^{(2m+1)k},(-1)^{\kappa}q^{\mu M-(2m+1)k}\big)^{2\ell+2}\notag\\
			&\times f\big((-1)^{\lambda}q^{(4\ell+4)k},(-1)^{\lambda}q^{2\mu M-(4\ell+4)k}\big)^{2m+1}\Big)=0.
		\end{align}
	\end{enumerate}
\end{theorem}

In the next corollary, we will translate the previous $\mathbf{H}$-operator relations into explicit vanishing expressions involving the five coefficient functions $\alpha$, $\beta$, $\gamma$, $\delta$ and $\epsilon$ defined in \eqref{eq:def-alpha}--\eqref{eq:def-epsilon} together with two new families:
\begin{align}
	\sum_{n}\phi_{i,j,r,\ell,m}(n)q^n &:= (-q^{i},-q^{r-i};q^{r})_\infty^\ell (q^{j},q^{r-j};q^{r})_\infty^m,\label{eq:def-phi}\\
	\sum_{n}\psi_{i,j,r,\ell,m}(n)q^n &:= (q^{i},q^{r-i};q^{r})_\infty^\ell (q^{j},q^{r-j};q^{r})_\infty^m.\label{eq:def-psi}
\end{align}

\begin{corollary}\label{coro:explicit-vanishing}
	We have the following results of the form
	\begin{align*}
		\chi_{ak,hM\mu+bk,M\mu,\ell,m}(Mn+sk)=0,
	\end{align*}
	where $\chi$ is among the seven coefficient functions, $\ell\ge 0$, $m\ge 0$, $\mu\ge 1$, $M\ge 1$, $h\ge 0$ and $s$ are integers, $a$ and $b$ are positive integers, and $k$ is such that $\gcd(k,M)=1$. Also, if the same relation holds true for several choices of $\chi$, we write for simplicity $\{\chi_1,\chi_2,\ldots\}$ with the corresponding subscript.
	\begin{align*}
		\{\phi,\psi\}_{k,2k,(2\ell+8m+5)\mu,2\ell+1,2m+1}\big((2\ell+8m+5)n+(2\ell+6m+4)k\big)&=0,\\
		\{\alpha,\beta\}_{2k,k,(8\ell+6m+3)\mu,2\ell+2m+1,2m+1}\big((8\ell+6m+3)n+(6\ell+4m+2)k\big)&=0,\\
		\{\alpha,\beta\}_{2k,k,(4\ell+6m+5)\mu,2\ell+4m+3,4m+2}\big((4\ell+6m+5)n+(2\ell+2m+2)k\big)&=0,\\
		\{\gamma,\delta\}_{k,(4\ell+2m+3)\mu+k,(4\ell+2m+3)\mu,2\ell+1,2m+1}\big((4\ell+2m+3)n+(3\ell+m+2)k\big)&=0,\\
		\{\gamma,\epsilon\}_{k,(2\ell+4m+3)\mu+2k,(2\ell+4m+3)\mu,2\ell+1,2m+1}\big((2\ell+4m+3)n+(2\ell+2m+2)k\big)&=0,\\
		\{\gamma,\delta\}_{k,k,(4\ell+2m+3)\mu,2\ell+1,2m+1}\big((4\ell+2m+3)n+(\ell+m+1)k\big)&=0,\\
		\{\gamma,\epsilon\}_{k,4k,(2\ell+16m+9)\mu,2\ell+1,2m+1}\big((2\ell+16m+9)n+(2\ell+12m+7)k\big)&=0,\\
		\{\delta,\epsilon\}_{2k,k,(16\ell+2m+9)\mu,2\ell+1,2m+1}\big((16\ell+2m+9)n+(10\ell+2m+6)k\big)&=0.
	\end{align*}
	Further, if $\gcd(2\ell+1,2m+1)=1$,
	\begin{align*}
		\{\gamma,\delta\}_{(2m+1)k,(2\ell+1)k,(2\ell+4m+3)\mu,2\ell+1,2m+1}\big((2\ell+4m+3)n-2(2m+1)^2 k\big)=0;
	\end{align*}
	if $\gcd(2\ell+1,2m+2)=1$,
	\begin{align*}
		\{\delta,\epsilon\}_{(2m+2)k,(2\ell+1)k,(2\ell+4m+5)\mu,2\ell+1,2m+2}\big((2\ell+4m+5)n-2(2m+2)^2 k\big)=0;
	\end{align*}
	if $\gcd(2\ell+2,2m+1)=1$,
	\begin{align*}
		\{\gamma,\epsilon\}_{(2m+1)k,(4\ell+4)k,(4\ell+2m+5)\mu,2\ell+2,2m+1}\big((4\ell+2m+5)n-3(2\ell+2)^2 k\big)=0.
	\end{align*}
\end{corollary}

\section{Outline}\label{sec:outline}

We outline the basic idea of our approach in this section. Let us begin with Entry 29 of Chapter XVI in Ramanujan's \textit{Notebooks} \cite[p.~199]{Ram1957}:

\begin{namedentry}[29]
	If $ab=cd$, then
	\begin{align}\label{eq:Rama-29-1}
		f(a,b)f(c,d)+f(-a,-b)f(-c,-d)=2f(ac,bd)f(ad,bc),
	\end{align}
	and
	\begin{align}\label{eq:Rama-29-2}
		f(a,b)f(c,d)-f(-a,-b)f(-c,-d)=2af(b/c,ac^2d)f(b/d,acd^2).
	\end{align}
		%
\end{namedentry}

From this entry, it is clear that products of two theta functions can be reformulated as linear combinations of other theta products. In fact, results of this nature can be traced at least back to Schr\"oter's 1854 dissertation \cite{Sch1854}, which is recorded as Lemma \ref{le:Sch} in Section \ref{sec:pairing}. Setting $a=c=q^k$ and $b=d=q^{-k+M}$, and then summing \eqref{eq:Rama-29-1} and \eqref{eq:Rama-29-2}, we have
\begin{align*}
	f\big(q^k,q^{-k+M}\big)^2=f\big(q^{2k},q^{-2k+2M}\big)f\big(q^M,q^M\big)+q^k f\big(q^{2k+M},q^{-2k+M}\big)f\big(1,q^{2M}\big).
\end{align*}
Recall that the $\mathbf{H}$-operator has a property that for any given series $F$ in $q^M$,
\begin{align*}
	\mathbf{H}_M\big(G(q)\cdot F(q^M)\big)= F(q^M)\cdot \mathbf{H}_M\big(G(q)\big).
\end{align*}
Based on this fact, we observe that on the right-hand side of the above reformulation of $f\big(q^k,q^{-k+M}\big)^2$, there is only one effective factor in each summand under the action of $\mathbf{H}_M$, namely, $f\big(q^{2k},q^{-2k+2M}\big)$ and $q^k f\big(q^{2k+M},q^{-2k+M}\big)$.

More generally, we know from a recent result of Mc Laughlin \cite{McL2019} that the above treatment can be embedded into a broader setting. In particular, for any theta power or any product of two theta powers, we may expand it as a linear combination of the form $\sum \sA \sF$, where $\sA$ is a theta series, usually times a power of $-1$ and a power of $q$, and $\sF$ is a series in $q^M$ for a certain $M$. Therefore, under the action of $\mathbf{H}_M$, only $\sA$ is effective in each summand. These results are recorded in Theorems \ref{th:theta-product} and \ref{th:theta-product-2}.

A surprising fact of the aforementioned linear expansion of theta powers is its hidden symmetry. To be precise, if we write the theta power
$$f\big((-1)^{\kappa}q^{k+A'},(-1)^{\kappa}q^{-k+(A-A')}\big)^m=\sum \sA \sF$$
as above, then for each summand $\sA \sF$ with only one exception, there exists a companion summand $\sA' \sF'$ such that $\sF=\sF'$. In other words, we may pair the summands and write the theta power as
\begin{align*}
	f\big((-1)^{\kappa}q^{k+A'},(-1)^{\kappa}q^{-k+(A-A')}\big)^m= \sA_0 \sF_0 +\sum\big(\sA_{I}+\sA_{I\!I}\big)\sF.
\end{align*}
Meanwhile, the quintuple product identity (see \cite[p.~99, Eq.~(10.1.4)]{Hir2017} or \cite[p.~119, Eq.~(1.9)]{Coo2006}) tells us that
\begin{align*}
	\frac{f\big({-q^{2k}},-q^{\mu-2k}\big)f\big({-q^{\mu}},-q^{2\mu}\big)}{f\big((-1)^{\kappa}q^k,(-1)^{\kappa}q^{\mu-k}\big)}&=f\big((-1)^{\kappa}q^{3k+\mu},(-1)^{\kappa}q^{-3k+2\mu}\big)\\
	&\quad+(-1)^{\kappa+1}q^k f\big((-1)^{\kappa}q^{3k+2\mu},(-1)^{\kappa}q^{-3k+\mu}\big).
\end{align*}
With the linear expansion formula for products of two theta powers, we find that a similar symmetry also holds true for powers of the theta quotient on the left-hand side of the above, and particularly, the exceptional unpaired summand vanishes. Namely,
\begin{align*}
	\left(\frac{f\big({-q^{2k}},-q^{\mu-2k}\big)}{f\big((-1)^{\kappa}q^k,(-1)^{\kappa}q^{\mu-k}\big)}\right)^m= \sum\big(\sA_{I}+\sA_{I\!I}\big)\sF.
\end{align*}
These results are recorded in Corollaries \ref{coro:power-pairing}, \ref{coro:power-pairing-A'-A'} and \ref{coro:f/f-expansion}.

With the above preparation in mind, it is easily seen that the theta products or quotients in Theorems \ref{th:Type-I}--\ref{th:Type-II} can be reformulated as
\begin{align*}
	\left(\sum \sA\sF\right)\cdot \left(\sum \sB\sG\right),
\end{align*}
or with our pairing process,
\begin{align*}
	\left(\sA_0 \sF_0 +\sum\big(\sA_{I}+\sA_{I\!I}\big)\sF\right)\cdot \left(\sB_0 \sG_0 +\sum\big(\sB_{I}+\sB_{I\!I}\big)\sG\right),
\end{align*}
where $\sF$ and $\sG$ (including $\sF_0$ and $\sG_0$) are series in $q^M$ for a certain $M$. We remark that $\sA_0$ or $\sB_0$ may vanish.

Our next observation comes from the known coefficient-vanishing results in the literature. Briefly speaking, if a coefficient-vanishing phenomenon appears, we either encounter a series of cancelations according to the pairing process:
\begin{align*}
	\mathbf{H}_M\big(\sA_0\sB_{0}\big) &=0,\\
	\mathbf{H}_M\big(\sA_0\sB_{I}\big) &=\pm \mathbf{H}_M\big(\sA_0\sB_{I\!I}\big),\\
	\mathbf{H}_M\big(\sB_{0}\sA_{I}\big) &=\pm \mathbf{H}_M\big(\sB_{0}\sA_{I\!I}\big),\\
	\mathbf{H}_M\big(\sA_{I}\sB_{I}\big) &=\pm \mathbf{H}_M\big(\sA_{I\!I}\sB_{I\!I}\big),\\
	\mathbf{H}_M\big(\sA_{I}\sB_{I\!I}\big) &=\pm \mathbf{H}_M\big(\sA_{I\!I}\sB_{I}\big),
\end{align*}
or we uniformly have
\begin{align*}
	\mathbf{H}_M\big(\sA\sB\big) =0.
\end{align*}
The above situations lead us to two different strategies for proving a given coefficient-vanishing problem. In particular, Theorems \ref{th:Type-I}--\ref{th:Type-I.3} fall into the former circumstance, and Theorem \ref{th:Type-II} requires the latter strategy. Detailed discussions and proofs are presented in Section \ref{sec:proof}.

What remains is a unified verification of the above relations. Since our main focus revolves around the behavior of
\begin{align*}
	\sH(q) &:= q^{w}f\big((-1)^{\kappa}q^{u+A'M},(-1)^{\kappa}q^{-u+(A-A')M}\big)\\
	&\ \quad\times f\big((-1)^{\lambda}q^{v+B'M},(-1)^{\lambda}q^{-v+(B-B')M}\big)
\end{align*}
under the action of $\mathbf{H}_M$, an effective way is to use the summation form of theta series:
\begin{align*}
	\sH(q)=\sum_{m,n\in\mathbb{Z}}(-1)^{\kappa m+\lambda n}q^{AM\binom{m}{2}+A'Mm+BM\binom{n}{2}+B'Mn+um+vn+w}.
\end{align*}
It turns out that what plays a central role under the action of $\mathbf{H}_M$ is the factor $q^{um+vn+w}$. This requires us to determine the solution set $\{(m,n):m,n\in\mathbb{Z}\}$ of the linear congruence
\begin{align*}
	um+vn+w \equiv 0 \pmod{M}.
\end{align*}
For this purpose, we begin with the homogeneous case:
\begin{align*}
	um+vn \equiv 0 \pmod{M}.
\end{align*}
Our target is to represent its solution set as a sublattice of $\mathbb{Z}^2$. Then for the inhomogeneous case, we only need to make a shift on the previously obtained sublattice. See Lemma \ref{le:w=0} and Theorem \ref{th:w} for details.

Finally, with the knowledge of the structure of the solution set of $um+vn+w \equiv 0 \pmod{M}$, we make substitutions for the summation indices $m$ and $n$, and obtain an explicit double summation formula for $\mathbf{H}_M\big(\sH(q)\big)$:
\begin{align*}
	\mathbf{H}_M\big(\sH(q)\big)=\mathbf{H}_M\left(\sum_{m,n\in\mathbb{Z}}(\star\cdots \star)\right) = \sum_{s,t\in\mathbb{Z}}(\star\cdots \star),
\end{align*}
in two new indices $s$ and $t$ over $\mathbb{Z}$. This reformulation allows us to determine for which $\sH(q)$, the relation $\mathbf{H}_M\big(\sH(q)\big)=0$ holds true. Also of interest are two companions of $\sH(q)$:
\begin{align*}
	\hat{\sH}(q)& :=q^{\hw}f\big((-1)^{\kappa}q^{u+A'M},(-1)^{\kappa}q^{-u+(A-A')M}\big)\notag\\
	&\ \quad\times f\big((-1)^{\lambda}q^{v+(B-B')M},(-1)^{\lambda}q^{-v+B'M}\big),\\
	\check{\sH}(q)&: =q^{\chw}f\big((-1)^{\kappa}q^{u+(A-A')M},(-1)^{\kappa}q^{-u+A'M}\big)\notag\\
	&\ \quad\times f\big((-1)^{\lambda}q^{v+(B-B')M},(-1)^{\lambda}q^{-v+B'M}\big).
\end{align*}
We provide some effective criteria to check whether $\mathbf{H}_M\big(\sH(q)\big)$ equals $\mathbf{H}_M\big(\hat{\sH}(q)\big)$ or $\mathbf{H}_M\big(\check{\sH}(q)\big)$. The related results are presented in Corollary \ref{coro:UMH=0-J} and Theorems \ref{th:cancel-1} and \ref{th:cancel-2}.

\section{Expanding and pairing powers of theta series}\label{sec:pairing}

\subsection{Schr\"oter and Mc Laughlin}

As we have pointed out in the previous section, reformulations of products of Ramanujan's theta series, such as \eqref{eq:Rama-29-1} and \eqref{eq:Rama-29-2}, have been widely studied. Among these results, a particularly interesting identity comes from Schr\"oter's 1854 dissertation \cite{Sch1854}; see also \cite[p.~111]{BB1987}.

\begin{lemma}[Schr\"oter]\label{le:Sch}
	We have
	\begin{align}
		f\big(q^{A}x,q^{A}/x\big)f\big(q^{B}y,q^{B}/y\big)&= \sum_{n=0}^{A+B-1}q^{An^2}x^n f\big(q^{A+B+2An}x/y,q^{A+B-2An}y/x\big)\notag\\
		&\quad\times f\big(q^{AB(A+B+2n)}(x^B y^A),q^{AB(A+B-2n)}/(x^B y^A)\big).
	\end{align}
	In particular,
	\begin{align}\label{eq:Sch}
		f\big(qz,q/z\big)^2&=\sum_{n=0}^1 q^{n^2}z^n f\big(q^{2+2n},q^{2-2n}\big) f\big(q^{2+2n}z^2,q^{2-2n}/z^2\big).
	\end{align}
\end{lemma}

In a recent paper of Mc Laughlin \cite{McL2019}, Schr\"oter's identity was generalized to products of an arbitrary number of theta series. In particular, the following formula is stated in \cite[Eq.~(4.4)]{McL2019}.

\begin{lemma}[Mc Laughlin]
	For any $m\ge 3$,
	\begin{align}\label{eq:McL}
		f\big(qz,q/z\big)^m&=\sum_{n_1=0}^1\sum_{n_2=0}^2\cdots \sum_{n_{m-1}=0}^{m-1} z^{n_{m-1}}
		q^{n_1^2+(n_2-n_1)^2+\cdots+(n_{m-1}-n_{m-2})^2}\notag\\
		&\quad\times f\big(q^{2+2n_1},q^{2-2n_1}\big)f\big(q^{m+2n_{m-1}}z^m,q^{m-2n_{m-1}}/z^m\big)\notag\\
		&\quad\times\prod_{i=2}^{m-1}f\big(q^{i(i+1)+2(i+1)n_{i-1}-2in_i},q^{i(i+1)-2(i+1)n_{i-1}+2in_i}\big).
	\end{align}
\end{lemma}

In fact, the above identity is a particular case of \cite[Corollary 4.2]{McL2019}, which may also be specialized as follows.

\begin{lemma}
	For any $m_1,m_2\ge 1$ with $m=m_1+m_2$,
	\begin{align}\label{eq:McL-2}
		&f\big(qxy,q/(xy)\big)^{m_1}f\big(qx/y,qy/x\big)^{m_2}\notag\\
		&=\sum_{n_1=0}^1\sum_{n_2=0}^2\cdots \sum_{n_{m-1}=0}^{m-1} x^{n_{m-1}}y^{2n_{m_1}-n_{m-1}}
		q^{n_1^2+(n_2-n_1)^2+\cdots+(n_{m-1}-n_{m-2})^2}\notag\\
		&\quad\times f\big(q^{2+2n_1}y^2,q^{2-2n_1}y^{-2}\big) f\big(q^{m+2n_{m-1}}x^my^{m_1-m_2},q^{m-2n_{m-1}}x^{-m}y^{m_2-m_1}\big)\notag\\
		&\quad\times\prod_{i=2}^{m_1}f\big(q^{i(i+1)+2(i+1)n_{i-1}-2in_i}y^{-2},q^{i(i+1)-2(i+1)n_{i-1}+2in_i}y^2\big)\notag\\
		&\quad\times\prod_{i=m_1+1}^{m-1}f\big(q^{i(i+1)+2(i+1)n_{i-1}-2in_i}y^{2m_1},q^{i(i+1)-2(i+1)n_{i-1}+2in_i}y^{-2m_1}\big).
	\end{align}
\end{lemma}

\begin{proof}
	This lemma follows upon setting $z=x$, $a_1=\cdots=a_{m_1}=y$ and $a_{m_1+1}=\cdots =a_{m_1+m_2}=y^{-1}$ in \cite[Corollary 4.2]{McL2019}.
\end{proof}

\subsection{One theta power}

We will make use of Schr\"oter's formula and Mc Laughlin's generalization to expand an arbitrary theta power as a summation of theta products. Before stating our result, an auxiliary lemma is required.

\begin{lemma}\label{le:sigma}
	For each positive integer $i$, define
	\begin{align*}
		\tau_i(n):=\begin{cases}
			0 & \text{if $n=0$},\\
			i+1-n & \text{if $n\ne 0$},
		\end{cases}
	\end{align*}
	for $0\le n\le i$. Let $m$ be a fixed positive integer. Then
	\begin{align}\label{eq:tau-def}
		\tau(n_1,n_2,\ldots,n_m):=\big(\tau_1(n_1),\tau_2(n_2),\ldots,\tau_m(n_m)\big)
	\end{align}
	is a bijection on the Cartesian product $\cI_m:=\{0,1\}\times \{0,1,2\}\times\cdots\times\{0,1,\ldots,m\}$. Further, for $(n_1,n_2,\ldots,n_m)\in \cI_m$, define
	\begin{align}
		\sigma(n_1,n_2,\ldots,n_m;q)&:=q^{\frac{n_1^2+(n_{2}-n_1)^2+\cdots +(n_m-n_{m-1})^2-n_m}{2}} f(q^{1+n_1},q^{1-n_1})\notag\\
		&\ \quad\times \prod_{i=2}^m f\big(q^{\frac{i(i+1)}{2}+(i+1)n_{i-1}-i n_i},q^{\frac{i(i+1)}{2}-(i+1)n_{i-1}+i n_i}\big).
	\end{align}
	Then
	\begin{align}\label{eq:sigma-relation}
		\sigma(n_1,n_2,\ldots,n_m;q) = \sigma\big(\tau_1(n_1),\tau_2(n_2),\ldots,\tau_m(n_m);q\big).
	\end{align}
\end{lemma}

\begin{proof}
	The first part of this lemma is obvious. For the second part, we apply induction on $m$. Notice that the $m=1$ case is trivial. This is because $\tau_1(n)=n$ for $n\in\{0,1\}$. Let us assume that \eqref{eq:sigma-relation} holds true for $m-1$ with $m\ge 2$. That is,
	\begin{align}\label{eq:sigma-induction-hyp}
		\sigma(n_1,n_2,\ldots,n_{m-1};q) = \sigma\big(\tau_1(n_1),\tau_2(n_2),\ldots,\tau_{m-1}(n_{m-1});q\big).
	\end{align} 
	
	We start by noting that
	\begin{align}\label{eq:sigma-g}
		&\sigma(n_1,n_2,\ldots,n_{m};q)\notag\\
		&\quad=\sigma(n_1,n_2,\ldots,n_{m-1};q)q^{\frac{(n_{m}-n_{m-1})^2-(n_{m}-n_{m-1})}{2}}\notag\\
		&\quad\quad \times f\big(q^{\frac{m(m+1)}{2}+(m+1)n_{m-1}-m n_m},q^{\frac{m(m+1)}{2}-(m+1)n_{m-1}+m n_m}\big)\notag\\
		&\quad =: \sigma(n_1,n_2,\ldots,n_{m-1};q) \cdot g(n_{m-1},n_m).
	\end{align}
	Now we claim that
	\begin{align}\label{eq:g-relation}
		g\big(\tau_{m-1}(n_{m-1}),\tau_m(n_{m})\big) = g(n_{m-1},n_m).
	\end{align}
	To prove this relation, we consider the following four cases:
	\begin{itemize}[leftmargin=*,align=left,itemsep=5pt]
		\renewcommand{\labelitemi}{\scriptsize$\blacktriangleright$}
		
		\item If $n_{m-1}=0$ and $n_{m}=0$, then
		\begin{align*}
			g\big(\tau_{m-1}(n_{m-1}),\tau_m(n_{m})\big) = g(0,0) = g(n_{m-1},n_m).
		\end{align*}
		
		\item If $n_{m-1}\ne 0$ and $n_{m}=0$, then
		\begin{align*}
			&g\big(\tau_{m-1}(n_{m-1}),\tau_m(n_{m})\big) = g\big(m-n_{m-1},0\big)\\
			& = q^{\frac{(-m+n_{m-1})^2-(-m+n_{m-1})}{2}}\\
			&\quad \times f\big(q^{\frac{m(m+1)}{2}+(m+1)(m-n_{m-1})},q^{\frac{m(m+1)}{2}-(m+1)(m-n_{m-1})}\big)\\
			& = q^{\frac{(-m+n_{m-1})^2-(-m+n_{m-1})}{2}}\\
			&\quad \times \sum_{s\in\mathbb{Z}}\big(q^{\frac{m(m+1)}{2}+(m+1)(m-n_{m-1})}\big)^{\frac{s(s+1)}{2}}\big(q^{\frac{m(m+1)}{2}-(m+1)(m-n_{m-1})}\big)^{\frac{s(s-1)}{2}}\\
			\text{\tiny($s\mapsto -1-s$)}& = q^{\frac{(-n_{m-1})^2-(-n_{m-1})}{2}}\\
			&\quad \times \sum_{s\in\mathbb{Z}}\big(q^{\frac{m(m+1)}{2}+(m+1)n_{m-1}}\big)^{\frac{s(s+1)}{2}}\big(q^{\frac{m(m+1)}{2}-(m+1)n_{m-1}}\big)^{\frac{s(s-1)}{2}}\\
			& = q^{\frac{(-n_{m-1})^2-(-n_{m-1})}{2}} f\big(q^{\frac{m(m+1)}{2}+(m+1)n_{m-1}},q^{\frac{m(m+1)}{2}-(m+1)n_{m-1}}\big)\\
			& = g(n_{m-1},0) = g(n_{m-1},n_m).
		\end{align*}
		
		\item If $n_{m-1}= 0$ and $n_{m}\ne 0$, then
		\begin{align*}
			&g\big(\tau_{m-1}(n_{m-1}),\tau_m(n_{m})\big) = g\big(0,m+1-n_m\big)\\
			& = q^{\frac{(m+1-n_m)^2-(m+1-n_m)}{2}}\\
			&\quad \times f\big(q^{\frac{m(m+1)}{2}-m(m+1-n_m)},q^{\frac{m(m+1)}{2}+m(m+1-n_m)}\big)\\
			& = q^{\frac{(m+1-n_m)^2-(m+1-n_m)}{2}}\\
			&\quad \times \sum_{s\in\mathbb{Z}}\big(q^{\frac{m(m+1)}{2}-m(m+1-n_m)}\big)^{\frac{s(s+1)}{2}}\big(q^{\frac{m(m+1)}{2}+m(m+1-n_m)}\big)^{\frac{s(s-1)}{2}}\\
			\text{\tiny($s\mapsto 1-s$)}& = q^{\frac{n_m^2-n_m}{2}} \sum_{s\in\mathbb{Z}}\big(q^{\frac{m(m+1)}{2}-mn_{m}}\big)^{\frac{s(s+1)}{2}}\big(q^{\frac{m(m+1)}{2}+mn_{m}}\big)^{\frac{s(s-1)}{2}}\\
			& = q^{\frac{n_m^2-n_m}{2}} f\big(q^{\frac{m(m+1)}{2}-mn_{m}},q^{\frac{m(m+1)}{2}+mn_{m}}\big)\\
			& = g(0,n_m) = g(n_{m-1},n_m).
		\end{align*}
		
		\item If $n_{m-1}\ne 0$ and $n_{m}\ne 0$, then
		\begin{align*}
			&g\big(\tau_{m-1}(n_{m-1}),\tau_m(n_{m})\big) = g\big(m-n_{m-1},m+1-n_m\big)\\
			& = q^{\frac{(1+n_{m-1}-n_m)^2-(1+n_{m-1}-n_m)}{2}}\\
			&\quad \times f\big(q^{\frac{m(m+1)}{2}+(m+1)(m-n_{m-1})-m(m+1-n_m)},q^{\frac{m(m+1)}{2}-(m+1)(m-n_{m-1})+m(m+1-n_m)}\big)\\
			& = q^{\frac{(1+n_{m-1}-n_m)^2-(1+n_{m-1}-n_m)}{2}}\\
			&\quad \times \sum_{s\in\mathbb{Z}}\big(q^{\frac{m(m+1)}{2}+(m+1)(m-n_{m-1})-m(m+1-n_m)})^{\frac{s(s+1)}{2}}\\
			&\quad\times\big(q^{\frac{m(m+1)}{2}-(m+1)(m-n_{m-1})+m(m+1-n_m)}\big)^{\frac{s(s-1)}{2}}\\
			& = q^{\frac{(n_{m}-n_{m-1})^2-(n_{m}-n_{m-1})}{2}}\\
			&\quad \times \sum_{s\in\mathbb{Z}}\big(q^{\frac{m(m+1)}{2}+(m+1)n_{m-1}-m n_m}\big)^{\frac{s(s+1)}{2}}\big(q^{\frac{m(m+1)}{2}-(m+1)n_{m-1}+m n_m}\big)^{\frac{s(s-1)}{2}}\\
			& = q^{\frac{(n_{m}-n_{m-1})^2-(n_{m}-n_{m-1})}{2}}\\
			&\quad \times f\big(q^{\frac{m(m+1)}{2}+(m+1)n_{m-1}-m n_m},q^{\frac{m(m+1)}{2}-(m+1)n_{m-1}+m n_m}\big)\\
			& = g(n_{m-1},n_m).
		\end{align*}
	\end{itemize}
	
	Finally, we conclude that
	\begin{align*}
		&\sigma\big(\tau_1(n_1),\tau_2(n_2),\ldots,\tau_m(n_m);q\big)\\
		\text{\tiny (by \eqref{eq:sigma-g})}&=\sigma\big(\tau_1(n_1),\tau_2(n_2),\ldots,\tau_{m-1}(n_{m-1});q\big)\cdot g\big(\tau_{m-1}(n_{m-1}),\tau_m(n_{m})\big)\\
		\text{\tiny (by \eqref{eq:g-relation})}&=\sigma\big(\tau_1(n_1),\tau_2(n_2),\ldots,\tau_{m-1}(n_{m-1});q\big)\cdot g(n_{m-1},n_m)\\
		\text{\tiny (by \eqref{eq:sigma-induction-hyp})}&= \sigma(n_1,n_2,\ldots,n_{m-1};q) \cdot g(n_{m-1},n_m)\\
		\text{\tiny (by \eqref{eq:sigma-g})}&= \sigma(n_1,n_2,\ldots,n_{m};q).
	\end{align*}
	This gives our desired result.
\end{proof}

\begin{theorem}\label{th:theta-product}
	Let $A$, $A'$ and $k$ be integers and let $\kappa\in\{0,1\}$. For any $m\ge 1$, there exist series $\sM_{s}^{(m)}(q^A)$ $(0\le s\le m-1)$ in $q^A$, depending only on $s$ and $m$, such that
	\begin{align}
		&f\big((-1)^{\kappa}q^{k+A'},(-1)^{\kappa}q^{-k+(A-A')}\big)^m\notag\\ &=\sum_{s=0}^{m-1}(-1)^{\kappa s}q^{ks}f\big((-1)^{\kappa m}q^{km+As+A'm},(-1)^{\kappa m}q^{-km-As+(A-A')m}\big)\notag\\
		&\quad\times q^{A's}\sM_{s}^{(m)}(q^A).\label{eq:f^m+}
	\end{align}
	Further, for $1\leq s\le m-1$,
	\begin{align}\label{relat-1}
		\sM_{s}^{(m)}(q^A) =\sM_{m-s}^{(m)}(q^A).
	\end{align}
\end{theorem}

\begin{proof}
	The case of $m=1$ is trivial. In particular,
	\begin{align}
		\sM_{0}^{(1)}(q^A)=1.
	\end{align}
	For $m=2$, we replace $q$ by $q^{\frac{A}{2}}$ and choose $z=(-1)^{\kappa}q^{k+(A'-\frac{A}{2})}$ in \eqref{eq:Sch}. Then
	\begin{align*}
		&f\big((-1)^{\kappa}q^{k+A'},(-1)^{\kappa}q^{-k+(A-A')}\big)^2\\
		&=\sum_{s=0}^1 (-1)^{\kappa s}q^{ks}f\big(q^{2k+As+2A'},q^{-2k-As+2(A-A')}\big)\\
		&\quad\times q^{A's}q^{\frac{s^2-s}{2}A}f\big(q^{(1+s)A},q^{(1-s)A}\big).
	\end{align*}
	Thus, for $0\le s\le 1$,
	\begin{align}
		\sM_{s}^{(2)}(q^A) = q^{\frac{s^2-s}{2}A}f\big(q^{(1+s)A},q^{(1-s)A}\big).
	\end{align}
	For any $m\ge 3$, we replace $q$ by $q^{\frac{A}{2}}$ and choose $z=(-1)^{\kappa}q^{k+(A'-\frac{A}{2})}$ in \eqref{eq:McL}. Then
	\begin{align*}
		&f\big((-1)^{\kappa}q^{k+A'},(-1)^{\kappa}q^{-k+(A-A')}\big)^m\\
		&= \sum_{s=0}^{m-1}(-1)^{\kappa s}q^{ks}f\big((-1)^{\kappa m}q^{km+As+A'm},(-1)^{\kappa m}q^{-km-As+(A-A')m}\big)\\
		&\quad\times q^{A's}\sum_{n_1=0}^1\cdots \sum_{n_{m-2}=0}^{m-2}q^{\frac{n_1^2+(n_{2}-n_1)^2+\cdots +(n_{m-2}-n_{m-3})^2+(s-n_{m-2})^2-s}{2}A}\\
		&\quad\times f(q^{(1+n_1)A},q^{(1-n_1)A})\\
		&\quad\times \prod_{i=2}^{m-2} f\big(q^{(\frac{i(i+1)}{2}+(i+1)n_{i-1}-i n_i)A},q^{(\frac{i(i+1)}{2}-(i+1)n_{i-1}+i n_i)A}\big)\\
		&\quad\quad\quad\quad\times f\big(q^{(\frac{m(m-1)}{2}+m n_{m-2}-(m-1)s)A},q^{(\frac{m(m-1)}{2}-m n_{m-2}+(m-1)s)A}\big).
	\end{align*}
	Recalling the notation in Lemma \ref{le:sigma}, we find that for $0\le s\le m-1$,
	\begin{align}\label{eq:M-m>=3}
		\sM_{s}^{(m)}(q^A)&=\sum_{(n_1,\ldots,n_{m-2})\in\cI_{m-2}}\sigma\big(n_1,\ldots,n_{m-2},s;q^{A}\big).
	\end{align}
	Therefore, \eqref{eq:f^m+} is established.
	
	For \eqref{relat-1}, we notice that the $m=1$ and $2$ cases are trivial. Assuming that $m\ge 3$, we deduce from the above that for $1\le s\le m-1$,
	\begin{align*}
		\sM_{m-s}^{(m)}(q^A)&=\sum_{(n_1,\ldots,n_{m-2})\in\cI_{m-2}}\sigma\big(n_1,\ldots,n_{m-2},m-s;q^{A}\big)\\
		\text{\tiny ($\tau$ in \eqref{eq:tau-def} is bijective)}&=\sum_{(n_1,\ldots,n_{m-2})\in\cI_{m-2}}\sigma\big(\tau_1(n_1),\ldots,\tau_{m-2}(n_{m-2}),\tau_{m-1}(s);q^{A}\big)\\
		\text{\tiny (by \eqref{eq:sigma-relation})}&= \sum_{(n_1,\ldots,n_{m-2})\in\cI_{m-2}}\sigma\big(n_1,\ldots,n_{m-2},s;q^{A}\big)\\
		\text{\tiny (by \eqref{eq:M-m>=3})}&= \sM_{s}^{(m)}(q^A),
	\end{align*}
	thereby concluding our proof.
\end{proof}

Finally, \eqref{relat-1} allows us to pair terms in \eqref{eq:f^m+}.

\begin{corollary}\label{coro:power-pairing}
	Let $\sM_{s}^{(m)}$ be as in Theorem \ref{th:theta-product}. Then
	\begin{align}
		f\big((-1)^{\kappa}q^{k+A'},(-1)^{\kappa}q^{-k+(A-A')}\big)^m=\sum_{s=0}^{m-1}\sP_s \sM_{s}^{(m)}(q^A),\label{eq:f^m+-new}
	\end{align}
	where
	\begin{align}
		\sP_0 = f\big((-1)^{\kappa m}q^{km+A'm},(-1)^{\kappa m}q^{-km+(A-A')m}\big),
	\end{align}
	and for $1\le s\le m-1$,
	\begin{align}
		\sP_s&=\frac{1}{2}\Big\{(-1)^{\kappa s}q^{(k+A')s}\notag\\
		&\quad\quad\quad\times f\big((-1)^{\kappa m}q^{km+As+A'm},(-1)^{\kappa m}q^{-km-As+(A-A')m}\big)\notag\\
		&\quad\quad+(-1)^{\kappa (m-s)}q^{(k+A')(m-s)}\notag\\
		&\quad\quad\quad\times f\big((-1)^{\kappa m}q^{km-As+(A+A')m},(-1)^{\kappa m}q^{-km+As-A'm}\big)\Big\}.
	\end{align}
\end{corollary}

\begin{proof}
	Recalling \eqref{relat-1}, we deduce from \eqref{eq:f^m+} that
	\begin{align*}
		&f\big((-1)^{\kappa}q^{k+A'},(-1)^{\kappa}q^{-k+(A-A')}\big)^m\notag\\
		&=f\big((-1)^{\kappa m}q^{km+A'm},(-1)^{\kappa m}q^{-km+(A-A')m}\big)\sM_{0}^{(m)}(q^A)\notag\\
		&\quad+\frac{1}{2}\sum_{s=1}^{m-1}(-1)^{\kappa s}q^{ks}f\big((-1)^{\kappa m}q^{km+As+A'm},(-1)^{\kappa m}q^{-km-As+(A-A')m}\big)\notag\\
		&\quad\quad\times q^{A's}\sM_{s}^{(m)}(q^A)\\
		&\quad+\frac{1}{2}\sum_{s=1}^{m-1}(-1)^{\kappa s}q^{ks}f\big((-1)^{\kappa m}q^{km+As+A'm},(-1)^{\kappa m}q^{-km-As+(A-A')m}\big)\notag\\
		&\quad\quad\times q^{A's}\sM_{m-s}^{(m)}(q^A).
	\end{align*}
	For the last summation in the above, we then change the indices by $s\mapsto m-s$. The desired result therefore follows.
\end{proof}

\begin{corollary}\label{coro:power-pairing-A'-A'}
	Let $\sM_{s}^{(m)}$ be as in Theorem \ref{th:theta-product}. Let $A=2A'$. Then
	\begin{align}
		f\big((-1)^{\kappa}q^{k+A'},(-1)^{\kappa}q^{-k+A'}\big)^m=\sum_{s=0}^{m-1}\sP_s \sM_{s}^{(m)}(q^A),\label{eq:f^m+-new-A'-A'}
	\end{align}
	where
	\begin{align}
		\sP_0 = f\big((-1)^{\kappa m}q^{km+A'm},(-1)^{\kappa m}q^{-km+A'm}\big),
	\end{align}
	and for $1\le s\le m-1$,
	\begin{align}
		\sP_s&=\frac{1}{4}\Big\{(-1)^{\kappa s}q^{(k+A')s}\notag\\
		&\quad\quad\quad\times f\big((-1)^{\kappa m}q^{km+A'(m+2s)},(-1)^{\kappa m}q^{-km+A'(m-2s)}\big)\notag\\
		&\quad\quad+(-1)^{\kappa s}q^{(-k+A')s}\notag\\
		&\quad\quad\quad\times f\big((-1)^{\kappa m}q^{km+A'(m-2s)},(-1)^{\kappa m}q^{-km+A'(m+2s)}\big)\Big\}\notag\\
		&\quad+\frac{1}{4}\Big\{(-1)^{\kappa (m-s)}q^{(k+A')(m-s)}\notag\\
		&\quad\quad\quad\times f\big((-1)^{\kappa m}q^{km+A'(3m-2s)},(-1)^{\kappa m}q^{-km+A'(-m+2s)}\big)\notag\\
		&\quad\quad+(-1)^{\kappa (m-s)}q^{(-k+A')(m-s)}\notag\\
		&\quad\quad\quad\times f\big((-1)^{\kappa m}q^{km+A'(-m+2s)},(-1)^{\kappa m}q^{-km+A'(3m-2s)}\big)\Big\}.
	\end{align}
\end{corollary}

\begin{proof}
	Recalling Corollary \ref{coro:power-pairing}, we find that
	\begin{align*}
		&2f\big((-1)^{\kappa}q^{k+A'},(-1)^{\kappa}q^{-k+A'}\big)^m\notag\\
		&=f\big((-1)^{\kappa}q^{k+A'},(-1)^{\kappa}q^{-k+A'}\big)^m+f\big((-1)^{\kappa}q^{-k+A'},(-1)^{\kappa}q^{k+A'}\big)^m\\
		&=f\big((-1)^{\kappa m}q^{km+A'm},(-1)^{\kappa m}q^{-km+A'm}\big)\sM_{0}^{(m)}(q^A)\\
		&\quad+\frac{1}{2}\Big\{(-1)^{\kappa s}q^{(k+A')s}\notag\\
		&\quad\quad\quad\times f\big((-1)^{\kappa m}q^{km+2A's+A'm},(-1)^{\kappa m}q^{-km-2A's+A'm}\big)\notag\\
		&\quad\quad+(-1)^{\kappa (m-s)}q^{(k+A')(m-s)}\notag\\
		&\quad\quad\quad\times f\big((-1)^{\kappa m}q^{km-2A's+3A'm},(-1)^{\kappa m}q^{-km+2A's-A'm}\big)\Big\}\sM_{s}^{(m)}(q^A)\\
		&\quad+f\big((-1)^{\kappa m}q^{-km+A'm},(-1)^{\kappa m}q^{km+A'm}\big)\sM_{0}^{(m)}(q^A)\\
		&\quad+\frac{1}{2}\Big\{(-1)^{\kappa s}q^{(-k+A')s}\notag\\
		&\quad\quad\quad\times f\big((-1)^{\kappa m}q^{-km+2A's+A'm},(-1)^{\kappa m}q^{km-2A's+A'm}\big)\notag\\
		&\quad\quad+(-1)^{\kappa (m-s)}q^{(-k+A')(m-s)}\notag\\
		&\quad\quad\quad\times f\big((-1)^{\kappa m}q^{-km-2A's+3A'm},(-1)^{\kappa m}q^{km+2A's-A'm}\big)\Big\}\sM_{s}^{(m)}(q^A)\\
		&=2f\big((-1)^{\kappa m}q^{km+A'm},(-1)^{\kappa m}q^{-km+A'm}\big)\sM_{0}^{(m)}(q^A)\\
		&\quad+\frac{1}{2}\Big\{(-1)^{\kappa s}q^{(k+A')s}\notag\\
		&\quad\quad\quad\times f\big((-1)^{\kappa m}q^{km+A'(m+2s)},(-1)^{\kappa m}q^{-km+A'(m-2s)}\big)\notag\\
		&\quad\quad+(-1)^{\kappa s}q^{(-k+A')s}\notag\\
		&\quad\quad\quad\times f\big((-1)^{\kappa m}q^{km+A'(m-2s)},(-1)^{\kappa m}q^{-km+A'(m+2s)}\big)\Big\}\sM_{s}^{(m)}(q^A)\\
		&\quad+\frac{1}{2}\Big\{(-1)^{\kappa (m-s)}q^{(k+A')(m-s)}\notag\\
		&\quad\quad\quad\times f\big((-1)^{\kappa m}q^{km+A'(3m-2s)},(-1)^{\kappa m}q^{-km+A'(-m+2s)}\big)\notag\\
		&\quad\quad+(-1)^{\kappa (m-s)}q^{(-k+A')(m-s)}\notag\\
		&\quad\quad\quad\times f\big((-1)^{\kappa m}q^{km+A'(-m+2s)},(-1)^{\kappa m}q^{-km+A'(3m-2s)}\big)\Big\}\sM_{s}^{(m)}(q^A),
	\end{align*}
	which is exactly what we need.
\end{proof}

%

\subsection{Two theta powers}

Our next concern is about reformulating products of two theta powers.

\begin{theorem}\label{th:theta-product-2}
	Let $A$, $A'$ and $k$ be integers and let $\kappa\in\{0,1\}$. For any $m_1,m_2\ge 1$ with $m=m_1+m_2$, there exist series $\sN_{s}^{(m_1,m_2)}(q^{A'},q^{A})$ $(0\le s\le m-1)$, depending only on $s$, $m_1$ and $m_2$, such that
	\begin{align}\label{eq:theta-product-2}
		&f\big((-1)^{\kappa}q^{k+A'},(-1)^{\kappa}q^{-k+(A-A')}\big)^{m_1}f\big((-1)^{\kappa}q^{k+(A-A')},(-1)^{\kappa}q^{-k+A'}\big)^{m_2}\notag\\ &=\sum_{s=0}^{m-1}(-1)^{\kappa s}q^{ks}\notag\\
		&\quad\times f\big((-1)^{\kappa m}q^{km+As+Am_2+A'(m_1-m_2)},(-1)^{\kappa m}q^{-km-As+Am_1-A'(m_1-m_2)}\big)\notag\\
		&\quad\times \sN_{s}^{(m_1,m_2)}(q^{A'},q^A).
	\end{align}
\end{theorem}

\begin{proof}
	We replace $q$ by $q^{\frac{A}{2}}$ and choose $x=(-1)^{\kappa}q^{k}$ and $y=q^{A'-\frac{A}{2}}$ in \eqref{eq:McL-2}. Then the index $n_{m-1}$ in \eqref{eq:McL-2} is renamed by $s$ so we have the summation with index $s$ in \eqref{eq:theta-product-2}. The terms with $\sN_{s}^{(m_1,m_2)}(q^{A'},q^A)$ removed in the summand of \eqref{eq:theta-product-2} are given by the following factor
	of the summand in \eqref{eq:McL-2}:
	\begin{align*}
		x^{s}f\big(q^{m+2s}x^my^{m_1-m_2},q^{m-2s}x^{-m}y^{m_2-m_1}\big).
	\end{align*}
	We will not present the explicit expression of $\sN_{s}^{(m_1,m_2)}(q^{A'},q^A)$ but it can be calculated by summing the remaining factors of the summand in \eqref{eq:McL-2} over indices $n_1,\ldots,n_{m-2}$.
\end{proof}

\begin{corollary}\label{coro:f/f-expansion}
	Let $\sM_{s}^{(m)}$ and $\sN_{s}^{(m_1,m_2)}$ be as in Theorems \ref{th:theta-product} and \ref{th:theta-product-2}. Then
	\begin{align}
		\left(\frac{f\big({-q^{2k}},-q^{\mu-2k}\big)}
		{f\big((-1)^{\kappa}q^k,(-1)^{\kappa}q^{\mu-k}\big)}\right)^m&=\frac{1}{f\big({-q^{\mu}},-q^{2\mu}\big)^m}\cdot 
		\Bigg\{\sum_{s=0}^{m-1}\sQ_s^{(0)}\sM_{s}^{(m)}(q^{3\mu})\notag\\
		&\quad+\sum_{m_1=1}^{m-1}\sum_{s=0}^{m-1}\sQ_{s}^{(m_1)}
		\sN_{s}^{(m_1,m-m_1)}(q^{\mu},q^{3\mu})\Bigg\},
	\end{align}
	where
	\begin{align}
		\sQ_{0}^{(0)}&=f\big((-1)^{\kappa m}q^{3km+\mu m},(-1)^{\kappa m}q^{-3km+2\mu m}\big)\notag\\
		&\quad +(-1)^{(\kappa+1)m}q^{km} f\big((-1)^{\kappa m}q^{3km+2\mu m},(-1)^{\kappa m}q^{-3km+\mu m}\big),
	\end{align}
	and for $1\le s\le m-1$,
	\begin{align}
		\sQ_{s}^{(0)}&=\frac{1}{2}\Big\{(-1)^{\kappa s}q^{(3k+\mu)s}\notag\\
		&\quad\quad\quad\times f\big((-1)^{\kappa m}q^{3km+\mu(m+3 s)},(-1)^{\kappa m}q^{-3km+\mu(2m-3 s)}\big)\notag\\
		&\quad\quad+(-1)^{(\kappa+1)m+\kappa s}q^{km+(-3k+\mu)s}\notag\\
		&\quad\quad\quad\times f\big((-1)^{\kappa m}q^{3km+\mu(2m-3 s)},(-1)^{\kappa m}q^{-3km+\mu(m+3 s)}\big)\Big\}\notag\\
		&\quad+\frac{1}{2}\Big\{(-1)^{(\kappa+1)m+\kappa (m-s)}q^{km+(-3k+\mu)(m-s)}\notag\\
		&\quad\quad\quad\times f\big((-1)^{\kappa m}q^{3km+\mu(-m+3 s)},(-1)^{\kappa m}q^{-3km+\mu(4m-3 s)}\big)\notag\\
		&\quad\quad+(-1)^{\kappa (m-s)}q^{(3k+\mu)(m-s)}\notag\\
		&\quad\quad\quad\times f\big((-1)^{\kappa m}q^{3km+\mu(4m-3 s)},(-1)^{\kappa m}q^{-3km+\mu(-m+3 s)}\big)\Big\},
	\end{align}
	and for $1\le m_1\le m-1$ and $0\le s\le m-1$,
	\begin{align}
		\sQ_{s}^{(m_1)}&=\frac{1}{2}\binom{m}{m_1}\Big\{(-1)^{(\kappa+1)(m-m_1)+\kappa s}q^{k(m-m_1)+3ks}\notag\\
		&\quad\quad\quad\times f\big((-1)^{\kappa m}q^{3km+\mu(2m-m_1+3 s)},(-1)^{\kappa m}q^{-3km+\mu(m+m_1-3s)}\big)\notag\\
		&\quad\quad+(-1)^{(\kappa+1)m_1+\kappa s}q^{km_1-3ks}\notag\\
		&\quad\quad\quad\times f\big((-1)^{\kappa m}q^{3km+\mu(m+m_1-3s)},(-1)^{\kappa m}q^{-3km+\mu(2m-m_1+3 s)}\big)\Big\}.
	\end{align}
\end{corollary}

\begin{proof}
	We know from the quintuple product identity that
	\begin{align*}
		\frac{f\big({-q^{2k}},-q^{\mu-2k}\big)f\big({-q^{\mu}},-q^{2\mu}\big)}{f\big((-1)^{\kappa}q^k,(-1)^{\kappa}q^{\mu-k}\big)}&=f\big((-1)^{\kappa}q^{3k+\mu},(-1)^{\kappa}q^{-3k+2\mu}\big)\\
		&\quad+(-1)^{\kappa+1}q^k f\big((-1)^{\kappa}q^{3k+2\mu},(-1)^{\kappa}q^{-3k+\mu}\big).
	\end{align*}
	Thus, we write
	\begin{align}\label{eq:S}
		\left(\frac{f\big({-q^{2k}},-q^{\mu-2k}\big)f\big({-q^{\mu}},-q^{2\mu}\big)}{f\big((-1)^{\kappa}q^k,(-1)^{\kappa}q^{\mu-k}\big)}\right)^m = S_{I}+S_{I\!I}+S_{I\!I\!I},
	\end{align}
	where
	\begin{align*}
		S_{I}&:=f\big((-1)^{\kappa}q^{3k+\mu},(-1)^{\kappa}q^{-3k+2\mu}\big)^m,\\
		S_{I\!I}&:=(-1)^{(\kappa+1)m}q^{km} f\big((-1)^{\kappa}q^{3k+2\mu},(-1)^{\kappa}q^{-3k+\mu}\big)^m,\\
		S_{I\!I\!I}&:=\sum_{m_1=1}^{m-1}\binom{m}{m_1}(-1)^{(\kappa+1)(m-m_1)}q^{k(m-m_1)}\\
		&\quad\times f\big((-1)^{\kappa}q^{3k+\mu},(-1)^{\kappa}q^{-3k+2\mu}\big)^{m_1} f\big((-1)^{\kappa}q^{3k+2\mu},(-1)^{\kappa}q^{-3k+\mu}\big)^{m-m_1}.
	\end{align*}
	
	By Corollary \ref{coro:power-pairing} with $(k,A',A)\mapsto (3k,\mu,3\mu)$, we have
	\begin{align}\label{eq:S-I}
		S_I=\sum_{s=0}^{m-1}\sQ_{s,I}^{(0)}\sM_{s}^{(m)}(q^{3\mu}),
	\end{align}
	where
	\begin{align*}
		\sQ_{0,I}^{(0)}= f\big((-1)^{\kappa m}q^{3km+\mu m},(-1)^{\kappa m}q^{-3km+2\mu m}\big),
	\end{align*}
	and for $1\le s\le m-1$,
	\begin{align*}
		\sQ_{s,I}^{(0)}&=\frac{1}{2}\Big\{(-1)^{\kappa s}q^{(3k+\mu)s}\\
		&\quad\quad\quad\times f\big((-1)^{\kappa m}q^{3km+\mu(m+3 s)},(-1)^{\kappa m}q^{-3km+\mu(2m-3 s)}\big)\notag\\
		&\quad\quad+(-1)^{\kappa (m-s)}q^{(3k+\mu)(m-s)}\\
		&\quad\quad\quad\times f\big((-1)^{\kappa m}q^{3km+\mu(4m-3 s)},(-1)^{\kappa m}q^{-3km+\mu(-m+3 s)}\big)\Big\}.
	\end{align*}
	
	Next, we have
	\begin{align*}
		S_{I\!I}=(-1)^{(\kappa+1)m}q^{km} f\big((-1)^{\kappa}q^{-3k+\mu},(-1)^{\kappa}q^{3k+2\mu}\big)^m.
	\end{align*}
	Applying Corollary \ref{coro:power-pairing} with $(k,A',A)\mapsto (-3k,\mu,3\mu)$ gives
	\begin{align}\label{eq:S-II}
		S_{I\!I}=\sum_{s=0}^{m-1}\sQ_{s,I\!I}^{(0)}\sM_{s}^{(m)}(q^{3\mu}),
	\end{align}
	where
	\begin{align*}
		\sQ_{0,I\!I}^{(0)}= (-1)^{(\kappa+1)m}q^{km} f\big((-1)^{\kappa m}q^{-3km+\mu m},(-1)^{\kappa m}q^{3km+2\mu m}\big),
	\end{align*}
	and for $1\le s\le m-1$,
	\begin{align*}
		\sQ_{s,I\!I}^{(0)}&=\frac{1}{2}\Big\{(-1)^{(\kappa+1)m+\kappa s}q^{km+(-3k+\mu)s}\\
		&\quad\quad\quad\times f\big((-1)^{\kappa m}q^{-3km+\mu(m+3 s)},(-1)^{\kappa m}q^{3km+\mu(2m-3 s)}\big)\notag\\
		&\quad\quad+(-1)^{(\kappa+1)m+\kappa (m-s)}q^{km+(-3k+\mu)(m-s)}\\
		&\quad\quad\quad\times f\big((-1)^{\kappa m}q^{-3km+\mu(4m-3 s)},(-1)^{\kappa m}q^{3km+\mu(-m+3 s)}\big)\Big\}.
	\end{align*}
	
	For $S_{I\!I\!I}$, we notice that
	\begin{align*}
		2S_{I\!I\!I}&=\sum_{m_1=1}^{m-1}\binom{m}{m_1}(-1)^{(\kappa+1)(m-m_1)}q^{k(m-m_1)}\\
		&\quad\quad\times f\big((-1)^{\kappa}q^{3k+\mu},(-1)^{\kappa}q^{-3k+2\mu}\big)^{m_1} f\big((-1)^{\kappa}q^{3k+2\mu},(-1)^{\kappa}q^{-3k+\mu}\big)^{m-m_1}\\
		&\quad+\sum_{m_1=1}^{m-1}\binom{m}{m-m_1}(-1)^{(\kappa+1)m_1}q^{km_1}\\
		&\quad\quad\times f\big((-1)^{\kappa}q^{3k+\mu},(-1)^{\kappa}q^{-3k+2\mu}\big)^{m-m_1} f\big((-1)^{\kappa}q^{3k+2\mu},(-1)^{\kappa}q^{-3k+\mu}\big)^{m_1}\\
		&=:S_{I\!I\!I_{1}}+S_{I\!I\!I_{2}}.
	\end{align*}
	By Theorem \ref{th:theta-product-2} with $(k,A',A)\mapsto (3k,\mu,3\mu)$, we have
	\begin{align}\label{eq:S-III-1}
		S_{I\!I\!I_{1}} = \sum_{m_1=1}^{m-1}\sum_{s=0}^{m-1}\sQ_{s,I\!I\!I_{1}}^{(m_1)}\sN_{s}^{(m_1,m-m_1)}(q^\mu,q^{3\mu}),
	\end{align}
	where for $0\le s\le m-1$,
	\begin{align*}
		\sQ_{s,I\!I\!I_{1}}^{(m_1)}&=\binom{m}{m_1}(-1)^{(\kappa+1)(m-m_1)+\kappa s}q^{k(m-m_1)+3ks}\\
		&\quad\times f\big((-1)^{\kappa m}q^{3km+\mu(2m-m_1+3 s)},(-1)^{\kappa m}q^{-3km+\mu(m+m_1-3s)}\big).
	\end{align*}
	Also,
	\begin{align*}
		S_{I\!I\!I_{2}} &= \sum_{m_1=1}^{m-1}\binom{m}{m_1}(-1)^{(\kappa+1)m_1}q^{km_1}\\
		&\quad\times f\big((-1)^{\kappa}q^{-3k+\mu},(-1)^{\kappa}q^{3k+2\mu}\big)^{m_1}f\big((-1)^{\kappa}q^{-3k+2\mu},(-1)^{\kappa}q^{3k+\mu}\big)^{m-m_1}.
	\end{align*}
	Applying Theorem \ref{th:theta-product-2} with $(k,A',A)\mapsto (-3k,\mu,3\mu)$ gives
	\begin{align}\label{eq:S-III-2}
		S_{I\!I\!I_{2}} = \sum_{m_1=1}^{m-1}\sum_{s=0}^{m-1}\sQ_{s,I\!I\!I_{2}}^{(m_1)}\sN_{s}^{(m_1,m-m_1)}(q^\mu,q^{3\mu}),
	\end{align}
	where for $0\le s\le m-1$,
	\begin{align*}
		\sQ_{s,I\!I\!I_{2}}^{(m_1)}&=\binom{m}{m_1}(-1)^{(\kappa+1)m_1+\kappa s}q^{km_1-3ks}\\
		&\quad\times f\big((-1)^{\kappa m}q^{-3km+\mu(2m-m_1+3 s)},(-1)^{\kappa m}q^{3km+\mu(m+m_1-3s)}\big).
	\end{align*}
	
	Recalling \eqref{eq:S}, we deduce from \eqref{eq:S-I}, \eqref{eq:S-II}, \eqref{eq:S-III-1} and \eqref{eq:S-III-2} that
	\begin{align*}
		\text{LHS of \eqref{eq:S}}&=S_{I}+S_{I\!I}+\frac{1}{2}\big(S_{I\!I\!I_{1}}+S_{I\!I\!I_{2}}\big)\\
		&=\sum_{s=0}^{m-1}\big(\sQ_{s,I}^{(0)}+\sQ_{s,I\!I}^{(0)}\big) \sM_{s}^{(m)}(q^{3\mu})\\
		&\quad+\frac{1}{2}\sum_{m_1=1}^{m-1}\sum_{s=0}^{m-1} \big(\sQ_{s,I\!I\!I_{1}}^{(m_1)}+\sQ_{s,I\!I\!I_{2}}^{(m_1)}\big)\sN_{s}^{(m_1,m-m_1)}(q^{\mu},q^{3\mu}).
	\end{align*}
	This implies our desired result.
\end{proof}

\begin{corollary}\label{coro:only-quotient-power}
	For any nonnegative integers $m'$ and $\mu$, any integer $k$ such that $\gcd(k,3(2m'+1))=1$, and any $\kappa\in\{0,1\}$,
	\begin{align}
		\mathbf{H}_{3(2m'+1)}\Bigg(q^{-2(2m'+1)k}\bigg(\frac{f\big({-q^{2k}},-q^{3(2m'+1)\mu -2k}\big)}{f\big((-1)^{\kappa}q^k,(-1)^{\kappa}q^{3(2m'+1)\mu -k}\big)}\bigg)^{2m'+1}\Bigg) = 0.
	\end{align}
\end{corollary}

\begin{proof}
	For notational convenience, we write $m=2m'+1$ and keep in mind that $m$ is odd. We also write $M=3m=3(2m'+1)$. By Corollary \ref{coro:f/f-expansion} with certain common factors extracted and powers of $(-1)$ modified, we have
	\begin{align*}
		\left(\frac{f\big({-q^{2k}},-q^{M\mu-2k}\big)}{f\big((-1)^{\kappa} q^k,(-1)^{\kappa}q^{M\mu-k}\big)}\right)^m
		&=\sA_0^{(0)}\sF_0^{(0)}+\sum_{s=1}^{m-1}\sA_{s,1}^{(0)}
		\sF_{s,1}^{(0)}+\sum_{s=1}^{m-1}\sA_{s,2}^{(0)}\sF_{s,2}^{(0)}\\
		&\quad+\sum_{m_1=1}^{m-1}\sum_{s=0}^{m-1}\sA_s^{(m_1)}\sF_s^{(m_1)},
	\end{align*}
	where each $\sF_{\star}^{(\star)}$ is a series in $q^M$, and
	\begin{align*}
		\sA_{0}^{(0)}&=f\big((-1)^{\kappa}q^{Mk+M\mu m},(-1)^{\kappa}q^{-Mk+2M\mu m}\big)\notag\\
		& +(-1)^{\kappa+1}q^{km} f\big((-1)^{\kappa}q^{Mk+2M\mu m},(-1)^{\kappa}q^{-Mk+M\mu m}\big),
	\end{align*}
	and for $1\le s\le m-1$,
	\begin{align*}
		\sA_{s,1}^{(0)}&=q^{3ks} f\big((-1)^{\kappa}q^{Mk+M\mu(m+3 s)},(-1)^{\kappa}q^{-Mk+M\mu(2m-3 s)}\big)\notag\\
		&+(-1)^{\kappa+1}q^{km-3ks} f\big((-1)^{\kappa}q^{Mk+M\mu(2m-3 s)},(-1)^{\kappa}q^{-Mk+M\mu(m+3 s)}\big),\notag\\
		\sA_{s,2}^{(0)}&=(-1)^{\kappa+1}q^{km-3k(m-s)} f\big((-1)^{\kappa}q^{Mk+M\mu(-m+3 s)},(-1)^{\kappa}q^{-Mk+M\mu(4m-3 s)}\big)\notag\\
		&+q^{3k(m-s)} f\big((-1)^{\kappa}q^{Mk+M\mu(4m-3 s)},(-1)^{\kappa}q^{-Mk+M\mu(-m+3 s)}\big),
	\end{align*}
	and for $1\le m_1\le m-1$ and $0\le s\le m-1$,
	\begin{align*}
		\sA_{s}^{(m_1)}&=q^{k(m-m_1)+3ks} f\big((-1)^{\kappa}q^{Mk+M\mu(2m-m_1+3 s)},(-1)^{\kappa}q^{-Mk+M\mu(m+m_1-3s)}\big)\notag\\
		&+(-1)^{\kappa+1}q^{km_1-3ks} f\big((-1)^{\kappa}q^{Mk+M\mu(m+m_1-3s)},(-1)^{\kappa}q^{-Mk+M\mu(2m-m_1+3 s)}\big).
	\end{align*}
	To show our desired result, it suffices to verify that for each choice of $\sA_{\star}^{(\star)}$,
	\begin{align*}
		\mathbf{H}_M\big(q^{-2km}\sA_{\star}^{(\star)}\big) = 0.
	\end{align*}
	
	We first notice that
	\begin{align*}
		q^{-2km}\sA_{0}^{(0)}&=q^{-2km}f\big((-1)^{\kappa}q^{Mk+M\mu m},(-1)^{\kappa}q^{-Mk+2M\mu m}\big)\notag\\
		&\quad +(-1)^{\kappa+1}q^{-km} f\big((-1)^{\kappa}q^{Mk+2M\mu m},(-1)^{\kappa}q^{-Mk+M\mu m}\big).
	\end{align*}
	The two theta functions in the above are series in $q^M$. Also, since $\gcd(k,M)=1$, we have
	\begin{align*}
		\left\{
		\begin{aligned}
			-2km&\not\equiv 0 \pmod{M},\\
			-km&\not\equiv 0 \pmod{M}.
		\end{aligned}
		\right.
	\end{align*}
	Thus,
	\begin{align*}
		\mathbf{H}_M\big(q^{-2km}\sA_{0}^{(0)}\big) = 0+0 =0.
	\end{align*}
	
	For $\sA_{s,1}^{(0)}$ and $\sA_{s,2}^{(0)}$ with $1\le s\le m-1$, and $\sA_{s}^{(m_1)}$ with $1\le m_1\le m-1$ and $0\le s\le m-1$, we first prove an auxiliary result, namely, if
	\begin{align*}
		\sA&=f\big((-1)^{\kappa}q^{Mk+M^2\mu},(-1)^{\kappa}q^{-Mk}\big)\notag\\
		&\quad+(-1)^{\kappa+1}q^{-Mk} f\big((-1)^{\kappa}q^{Mk},(-1)^{\kappa}q^{-Mk+M^2\mu}\big),
	\end{align*}
	then
	\begin{align*}
		\sA=0.
	\end{align*}
	To show this, we observe that
	\begin{align*}
		\sA&=\sum_{s\in\mathbb{Z}}(-1)^{\kappa s}q^{M^2\mu\binom{s}{2}-Mks}+(-1)^{\kappa+1}\sum_{t\in\mathbb{Z}}(-1)^{\kappa t}q^{M^2\mu\binom{t}{2}+Mkt-Mk}\\
		&=\sum_{s\in\mathbb{Z}}(-1)^{\kappa s}q^{M^2\mu\binom{s}{2}-Mks}+(-1)^{\kappa+1}\sum_{t\in\mathbb{Z}}(-1)^{\kappa (1-t)}q^{M^2\mu\binom{1-t}{2}+Mk(1-t)-Mk}\\
		&=\sum_{s\in\mathbb{Z}}(-1)^{\kappa s}q^{M^2\mu\binom{s}{2}-Mks}-\sum_{t\in\mathbb{Z}}(-1)^{\kappa t}q^{M^2\mu\binom{t}{2}-Mkt}\\
		&=0.
	\end{align*}
	
	Now we notice that
	\begin{align*}
		q^{-2km}\sA_{s,1}^{(0)}&=q^{-2km+3ks} f\big((-1)^{\kappa}q^{Mk+M\mu(m+3 s)},(-1)^{\kappa}q^{-Mk+M\mu(2m-3 s)}\big)\notag\\
		&+(-1)^{\kappa+1}q^{-km-3ks} f\big((-1)^{\kappa}q^{Mk+M\mu(2m-3 s)},(-1)^{\kappa}q^{-Mk+M\mu(m+3 s)}\big).
	\end{align*}
	Here the two theta functions are series in $q^M$. Since $1\le s\le m-1$, we find that $-2km+3ks \equiv 0 \pmod{M}$ only if $3s=2m$ provided that $m$ is a multiple of $3$. Similarly, $-km-3ks\equiv 0 \pmod{M}$ only if $3s=2m$. Thus, when $3s\ne 2m$, we have
	\begin{align*}
		\mathbf{H}_M\big(q^{-2km}\sA_{s,1}^{(0)}\big) = 0+0 =0.
	\end{align*}
	When $3s=2m$, we have
	\begin{align*}
		q^{-2km}\sA_{s,1}^{(0)}=\sA=0,
	\end{align*}
	which implies that
	\begin{align*}
		\mathbf{H}_M\big(q^{-2km}\sA_{s,1}^{(0)}\big) =0.
	\end{align*}
	Also, for $\sA_{s,2}^{(0)}$, we only need to consider the two cases $3s\ne m$ and $3s=m$ by a similar argument.
	
	Finally, for $\sA_{s}^{(m_1)}$ with $1\le m_1\le m-1$ and $0\le s\le m-1$, we have
	\begin{align*}
		&q^{-2km}\sA_{s}^{(m_1)}\\
		&\!=q^{-km-km_1+3ks} f\big((-1)^{\kappa}q^{Mk+M\mu(2m-m_1+3 s)},(-1)^{\kappa}q^{-Mk+M\mu(m+m_1-3s)}\big)\notag\\
		&\!+(-1)^{\kappa+1}q^{-2km+km_1-3ks} f\big((-1)^{\kappa}q^{Mk+M\mu(m+m_1-3s)},(-1)^{\kappa}q^{-Mk+M\mu(2m-m_1+3 s)}\big).
	\end{align*}
	We still observe that the two theta functions are series in $q^M$. Since $1\le m_1\le m-1$ and $0\le s\le m-1$, we have $-m<3s-m_1<3m$. Then $-km-km_1+3ks\equiv 0 \pmod{M}$ only if $3s-m_1=m$. Also, $-2km+km_1-3ks\equiv 0 \pmod{M}$ only if $3s-m_1=m$. Therefore, when $3s-m_1\ne m$,
	\begin{align*}
		\mathbf{H}_M\big(q^{-2km}\sA_{s}^{(m_1)}\big) = 0+0 =0.
	\end{align*}
	When $3s-m_1= m$, we have
	\begin{align*}
		q^{-2km}\sA_{s}^{(m_1)}=\sA=0,
	\end{align*}
	which implies that
	\begin{align*}
		\mathbf{H}_M\big(q^{-2km}\sA_{s}^{(m_1)}\big) =0.
	\end{align*}
	
	The above discussion therefore completes our proof.
\end{proof}

\section{Sublattices of $\mathbb{Z}^2$ and the solution set of linear congruences}\label{sec:linear-cong}

\subsection{Homogeneous linear congruences}

We start by constructing a sublattice of $\mathbb{Z}^2$ to represent the solution set of a family of homogeneous linear congruences.

\begin{lemma}\label{le:w=0}
	Let $M$ be a positive integer. Let $u$ and $v$ be integers and write $d^*=\gcd(u,v,M)$. Let $d_u=\gcd(u,M)$ and $d_v=\gcd(v,M)$. Also, let $u'=u/d_u$ and $v'=v/d_v$. Assume that integers $a_1$, $b_1$, $a_2$ and $b_2$ with $\gcd(a_1,b_1)=1$ and $\gcd(a_2,b_2)=1$ are such that
	\begin{align}
		u'a_1+v'a_2&\equiv 0 \pmod{d^*M/(d_ud_v)},\label{eq:cond_a}\\
		u'b_1+v'b_2&\equiv 0 \pmod{d^*M/(d_ud_v)},\label{eq:cond_b}
	\end{align}
	and
	\begin{align}\label{eq:cond_det}
		a_1b_2-a_2b_1=\pm d^*M/(d_ud_v).
	\end{align}
	Then
	\begin{align}
		&\{(m,n)\in\mathbb{Z}^2:um+vn\equiv 0 \ (\operatorname{mod}\,M) \}\notag\\
		&\qquad\qquad=\big\{\big((d_v/d^*)(a_1s+b_1t),(d_u/d^*)(a_2s+b_2t)\big):(s,t)\in\mathbb{Z}^2\big\}.
	\end{align}
	In particular, if we write $d=\gcd(u,v)$, then the linear Diophantine equation
	\begin{align}\label{eq:u-v-M}
		ud_vx+vd_uy= dM
	\end{align}
	is solvable. Let $(x,y)=(\alpha,\beta)$ be a specific solution. Then the integers $a_1$, $b_1$, $a_2$ and $b_2$ can be chosen as
	\begin{align}\label{eq:sol-ab}
		\left\{
		\begin{aligned}
			a_1&=\alpha,\\
			b_1&= -vd^*/(dd_v),\\
			a_2&=\beta,\\
			b_2&=ud^*/(dd_u).
		\end{aligned}
		\right.
	\end{align}
\end{lemma}

\begin{proof}
	For convenience, we write
	\begin{align*}
		\sL:=\{(m,n)\in\mathbb{Z}^2:um+vn\equiv 0 \ (\operatorname{mod}\,M) \},
	\end{align*}
	and
	\begin{align*}
		\sR:=\big\{\big((d_v/d^*)(a_1s+b_1t),(d_u/d^*)(a_2s+b_2t)\big):(s,t)\in\mathbb{Z}^2\big\}.
	\end{align*}
	We start by noticing that
	\begin{align*}
		\sL=\{(m,n)\in\mathbb{Z}^2:u^{\divideontimes}m+v^{\divideontimes}n\equiv 0 \ (\operatorname{mod}\,M^{\divideontimes}) \},
	\end{align*}
	where $u^{\divideontimes}=u/d^*$, $v^{\divideontimes}=v/d^*$ and $M^{\divideontimes}=M/d^*$. Notice also that $\gcd(u^{\divideontimes},v^{\divideontimes},M^{\divideontimes})=1$. Let us write
	\begin{align*}
		d_u^{\divideontimes}&=\gcd(u^{\divideontimes},M^{\divideontimes})=\frac{d_u}{d^*},\\
		d_v^{\divideontimes}&=\gcd(v^{\divideontimes},M^{\divideontimes})=\frac{d_v}{d^*}.
	\end{align*}
	Then, $\gcd(u^{\divideontimes},d_v^{\divideontimes})=1$ and $\gcd(v^{\divideontimes},d_u^{\divideontimes})=1$ (and so $\gcd(d_u^{\divideontimes},d_v^{\divideontimes})=1$) since $\gcd(u^{\divideontimes},v^{\divideontimes},M^{\divideontimes})=1$. This also implies that
	$$\frac{d^*M}{d_ud_v}=\frac{M^{\divideontimes}}{d_u^{\divideontimes}d_v^{\divideontimes}}$$
	is an integer. Further, we notice that
	\begin{align*}
		\frac{u^{\divideontimes}}{d_u^{\divideontimes}}&=\frac{u}{d_u}=u',\\
		\frac{v^{\divideontimes}}{d_v^{\divideontimes}}&=\frac{v}{d_v}=v'.
	\end{align*}
	Also, since $d_u^{\divideontimes}=\gcd(u^{\divideontimes},M^{\divideontimes})$ and $d_v^{\divideontimes}=\gcd(v^{\divideontimes},M^{\divideontimes})$, we have
	$$\gcd(u',d^*M/(d_ud_v))=\gcd(u^{\divideontimes}/d_u^{\divideontimes},M^{\divideontimes}/(d_u^{\divideontimes}d_v^{\divideontimes}))=1,$$
	and
	$$\gcd(v',d^*M/(d_ud_v))=\gcd(v^{\divideontimes}/d_v^{\divideontimes},M^{\divideontimes}/(d_u^{\divideontimes}d_v^{\divideontimes}))=1.$$
	
	It can be seen that for any solution $(m,n)$ to
	\begin{align*}
		u^{\divideontimes}m+v^{\divideontimes}n\equiv 0 \pmod{M^{\divideontimes}},
	\end{align*}
	we have $d_v^{\divideontimes}\mid m$ and $d_u^{\divideontimes}\mid n$. We shall write $m'=m/d_v^{\divideontimes}$ and $n'=n/d_u^{\divideontimes}$. Then by recalling that $u'=u^{\divideontimes}/d_u^{\divideontimes}$, $v'=v^{\divideontimes}/d_v^{\divideontimes}$ and $d^*M/(d_ud_v)=M^{\divideontimes}/(d_u^{\divideontimes}d_v^{\divideontimes})$, we have
	\begin{align*}
		\sL = \{(d_v^{\divideontimes}m',d_u^{\divideontimes}n')\in\mathbb{Z}^2:u'm'+v'n'\equiv 0 \ (\operatorname{mod}\,d^*M/(d_ud_v)) \}.
	\end{align*}
	Now it suffices to show that
	\begin{align*}
		&\sL':=\{(m',n')\in\mathbb{Z}^2:u'm'+v'n'\equiv 0 \ (\operatorname{mod}\,d^*M/(d_ud_v)) \}\\
		&\qquad\qquad\qquad\qquad\qquad = \{(a_1s+b_1t,a_2s+b_2t):(s,t)\in\mathbb{Z}^2\}=:\sR'.
	\end{align*}
	
	Let us fix $m'$. Since
	\begin{align*}
		u'm'+v'n'\equiv 0 \pmod{d^*M/(d_ud_v)},
	\end{align*}
	we find that all $n'$ such that $(m',n')\in\sL'$ form a bilateral arithmetic progression of common difference $\pm d^*M/(d_ud_v)$. Namely,
	$$n'\equiv -u'\overline{v'}m' \pmod{d^*M/(d_ud_v)},$$
	where $\overline{v'}$ is such that $v'\overline{v'}\equiv 1 \pmod{d^*M/(d_ud_v)}$; this is doable since $v'$ is coprime to $d^*M/(d_ud_v)$ as shown in the previous discussion. Now we consider the linear Diophantine equation
	\begin{align*}
		a_1s+b_1t = m'.
	\end{align*}
	Since $\gcd(a_1,b_1)=1$, it is always solvable and has infinitely many solutions. Let $(s_0,t_0)$ be a specific solution. Then the general solutions $(s,t)$ satisfy
	\begin{align}\label{eq:Dio-eqn-m}
		a_1(s-s_0)+b_1(t-t_0)=0.
	\end{align}
	Recalling that $\gcd(a_1,b_1)=1$, we have $s-s_0=kb_1$ and $t-t_0=-ka_1$ for any $k\in\mathbb{Z}$. Thus, the general solutions $(s,t)$ to \eqref{eq:Dio-eqn-m} can be parameterized as
	\begin{align*}
		\left\{
		\begin{aligned}
			s&=s_0+kb_1\\
			t&=t_0-ka_1
		\end{aligned}
		\right.	
		\qquad\text{(with $k\in\mathbb{Z}$)}.
	\end{align*}
	We then observe that the integers
	\begin{align*}
		a_2s+b_2t &= a_2(s_0+kb_1)+b_2(t_0-ka_1)\\
		&= a_2s_0 + b_2t_0 +k(a_2b_1-a_1b_2)
	\end{align*}
	also form a bilateral arithmetic progression of common difference $\pm d^*M/(d_ud_v)$ since $a_1b_2-a_2b_1=\pm d^*M/(d_ud_v)$ according to \eqref{eq:cond_det}. To show that the list $\{n'\}$ is identical to the list $\{a_2s+b_2t\}$, it suffices to examine that one element in $\{a_2s+b_2t\}$ is also in $\{n'\}$. Evidently, $a_2s_0 + b_2t_0$ is such an element. This is because $a_1s_0+b_1t_0 = m'$ by our assumption. Then
	\begin{align*}
		u'm'+v'(a_2s_0 + b_2t_0)&=u'(a_1s_0+b_1t_0)+v'(a_2s_0 + b_2t_0)\\
		&=(u'a_1+v'a_2)s_0+(u'b_1+v'b_2)t_0\\
		&\equiv 0+0=0 \pmod{d^*M/(d_ud_v)}.
	\end{align*}
	Therefore, ranging $m'$ over $\mathbb{Z}$ gives $\sR'= \sL'$, which further yields $\sR=\sL$.
	
	Finally, we notice that
	\begin{align*}
		\gcd(u',v')=\frac{d}{d^*}.
	\end{align*}
	Recalling that $d^*M/(d_ud_v)$ is an integer, we conclude that \eqref{eq:u-v-M}, which can be rewritten as
	\begin{align*}
		u'x+v'y=\frac{d}{d^*}\cdot \frac{d^*M}{d_ud_v},
	\end{align*}
	is solvable by B\'ezout's lemma. If $a_1$, $b_1$, $a_2$ and $b_2$ are chosen as in \eqref{eq:sol-ab}, we first notice that
	\begin{align*}
		b_1=-\frac{vd^*}{dd_v}=-\frac{v\cdot \gcd(u,v,M)}{\gcd(u,v)\gcd(v,M)}
	\end{align*} 
	and
	\begin{align*}
		b_2=\frac{ud^*}{dd_u}=\frac{u\cdot \gcd(u,v,M)}{\gcd(u,v)\gcd(u,M)}
	\end{align*} 
	are integers. Now we examine the following conditions by recalling that $u'\alpha+v'\beta=dM/(d_ud_v)$:
	\begin{itemize}[leftmargin=*,align=left,itemsep=5pt]
		\item[\textit{Eq}.~\eqref{eq:cond_a}:] 
		\begin{align*}
			u'a_1+v'a_2=u'\alpha+v'\beta=dM/(d_ud_v)\equiv 0 \pmod{d^*M/(d_ud_v)};
		\end{align*}
		
		\item[\textit{Eq}.~\eqref{eq:cond_b}:] 
		\begin{align*}
			u'b_1+v'b_2=u'(-v'd^*/d)+v'(u'd^*/d)=0\equiv 0 \pmod{d^*M/(d_ud_v)};
		\end{align*}
		
		\item[\textit{Eq}.~\eqref{eq:cond_det}:] 
		\begin{align*}
			a_1b_2-a_2b_1 = \alpha (u'd^*/d) - \beta (-v'd^*/d) = d^*M/(d_ud_v).
		\end{align*}
	\end{itemize}
	Further, if $\gcd(\alpha,v'd^*/d)=\ell>1$, then $d^*M/(d_ud_v)$ is also divisible by $\ell$ since $(u'd^*/d)\alpha+(v'd^*/d)\beta=d^*M/(d_ud_v)$. Thus, $\ell>1$ is a common divisor of $v'd^*/d$ and $d^*M/(d_ud_v)$; this violates the fact that $v'$ (which is a multiple of $v'd^*/d$ as $d^*\mid d$) is coprime to $d^*M/(d_ud_v)$. Thus, $\gcd(\alpha,v'd^*/d)=1$. Similarly, we have $\gcd(\beta,u'd^*/d)=1$.
\end{proof}

\subsection{Inhomogeneous linear congruences}

Theorem \ref{le:w=0} allows us to analyze the inhomogeneous case.

\begin{theorem}\label{th:w}
	Let $M$ be a positive integer. Let $u$, $v$ and $w$ be integers and write $d^*=\gcd(u,v,M)$.
	\begin{enumerate}[label=\textup{(\roman*)}, widest=ii, itemindent=*, leftmargin=*]
		
		\item 
		Assume that $d^*\nmid w$. Then
		\begin{align}\label{eq:sublattice-general-0}
			\{(m,n)\in\mathbb{Z}^2:um+vn+w\equiv 0 \ (\operatorname{mod}\,M) \}=\varnothing.
		\end{align}
		
		\item 
		Assume that $d^*\mid w$. Let $d_u=\gcd(u,M)$ and $d_v=\gcd(v,M)$. Also, let $u'=u/d_u$ and $v'=v/d_v$. Assume that integers $a_1$, $b_1$, $a_2$ and $b_2$ with $\gcd(a_1,b_1)=1$ and $\gcd(a_2,b_2)=1$ are such that
		\begin{align}
			u'a_1+v'a_2&\equiv 0 \pmod{d^*M/(d_ud_v)},\label{eq:cond_a-w}\\
			u'b_1+v'b_2&\equiv 0 \pmod{d^*M/(d_ud_v)},\label{eq:cond_b-w}
		\end{align}
		and
		\begin{align}\label{eq:cond_det-w}
			a_1b_2-a_2b_1=\pm d^*M/(d_ud_v).
		\end{align}
		Let $(m_0,n_0)$ be a pair of integers such that
		\begin{align}
			um_0+vn_0+w\equiv 0 \pmod{M}.
		\end{align}
		Then
		\begin{align}\label{eq:sublattice-general}
			&\{(m,n)\in\mathbb{Z}^2:um+vn+w\equiv 0 \ (\operatorname{mod}\,M) \}\notag\\
			&\qquad=\big\{\big((d_v/d^*)(a_1s+b_1t)+m_0,(d_u/d^*)(a_2s+b_2t)+n_0\big):(s,t)\in\mathbb{Z}^2\big\}.
		\end{align}
		In particular, we write $d=\gcd(u,v)$ and let $(x,y)=(\alpha,\beta)$ be a specific solution of the linear Diophantine equation
		\begin{align}\label{eq:u-v-M-w}
			ud_vx+vd_uy= dM.
		\end{align}
		Then
		\begin{align}\label{eq:sublattice-particular}
			&\{(m,n)\in\mathbb{Z}^2:um+vn+w\equiv 0 \ (\operatorname{mod}\,M) \}\notag\\
			&\qquad=\{(\alpha sd_v/d^*-vt/d+m_0,\beta s d_u/d^*+ut/d+n_0):(s,t)\in\mathbb{Z}^2\}.
		\end{align}
		
	\end{enumerate}
\end{theorem}

\begin{proof}
	For the first part, we notice that to find a solution $(m,n)$ to $um+vn+w\equiv 0 \pmod{M}$, it suffices to find integers $m$, $n$ and $k$ such that $um+vn+Mk=-w$. By B\'ezout's lemma, this linear Diophantine equation is solvable only if $w$ is a multiple of $d^*=\gcd(u,v,M)$.
	
	Recall that for $d^*\mid w$, $(m_0,n_0)$ is a pair of integers such that $um_0+vn_0+w\equiv 0 \pmod{M}$. Considering
	\begin{align*}
		\{(\tilde{m},\tilde{n})\in\mathbb{Z}^2:u\tilde{m}+v\tilde{n}\equiv 0 \ (\operatorname{mod}\,M) \}
	\end{align*}
	and
	\begin{align*}
		\{(m,n)\in\mathbb{Z}^2:um+vn+w\equiv 0 \ (\operatorname{mod}\,M) \},
	\end{align*}
	there is a trivial correspondence between $(\tilde{m},\tilde{n})$ and $(m,n)$ given by
	\begin{align*}
		(m,n) = (\tilde{m}+m_0,\tilde{n}+n_0).
	\end{align*}
	Recalling Lemma \ref{le:w=0} gives the desired result.
\end{proof}

\section{Dissecting theta products}\label{sec:dissecting}

\subsection{Huffing operator}

Recall from \eqref{eq:U-operator} that the $\mathbf{H}$-operator is defined for $G(q)=\sum_{n}g_nq^n$ by
\begin{align*}
	\mathbf{H}_M\big(G(q)\big)=\sum_{n} g_{Mn}q^{Mn}.
\end{align*}

Throughout, we consider
\begin{align}\label{eq:H-sum-form}
	\sH(q):=\sum_{m,n\in\mathbb{Z}}(-1)^{\kappa m+\lambda n}q^{AM\binom{m}{2}+A'Mm+BM\binom{n}{2}+B'Mn+um+vn+w},
\end{align}
where $M\ge 1$, $A\ge 1$, $B\ge 1$, $A'$, $B'$, $u$, $v$, $w$, $\kappa$ and $\lambda$ are integers. Notice that by the definition of Ramanujan's theta series,
\begin{align}
	\sH(q) &= q^{w}f\big((-1)^{\kappa}q^{u+A'M},(-1)^{\kappa}q^{-u+(A-A')M}\big)\notag\\
	&\quad\times f\big((-1)^{\lambda}q^{v+B'M},(-1)^{\lambda}q^{-v+(B-B')M}\big).
\end{align}
Let $d^*$, $d_u$, $d_v$, $a_1$, $b_1$, $a_2$, $b_2$, $m_0$ and $n_0$ be as in Theorem \ref{th:w}. We further require that $d^*\mid w$. It follows from \eqref{eq:sublattice-general} with the following changes of variables in \eqref{eq:H-sum-form}:
\begin{align*}
	\left\{
	\begin{aligned}
		m&= (d_v/d^*)(a_1s+b_1t)+m_0,\\
		n&= (d_u/d^*)(a_2s+b_2t)+n_0,
	\end{aligned}
	\right.
\end{align*}
that
\begin{align}\label{eq:UM(H)-general}
	&\mathbf{H}_M\big(\sH(q)\big)\notag\\
	&=(-1)^{\kappa m_0+\lambda n_0} q^{M\big(A\binom{m_0}{2}+A'm_0+B\binom{n_0}{2}+B'n_0\big)+(um_0+vn_0+w)}\notag\\
	&\quad\times\sum_{s,t\in\mathbb{Z}}(-1)^{(\kappa a_1d_v+\lambda a_2d_u)(d^*)^{-1} s+(\kappa b_1d_v+\lambda b_2d_u)(d^*)^{-1} t}\notag\\
	&\quad\times q^{M (Aa_1^2d_v^2+Ba_2^2d_u^2)(d^*)^{-2}\binom{s}{2}+M (Ab_1^2d_v^2+Bb_2^2d_u^2)(d^*)^{-2}\binom{t}{2}}\notag\\
	&\quad\times q^{M\big(A\binom{a_1d_v(d^*)^{-1}}{2}+B\binom{a_2d_u(d^*)^{-1}}{2}\big)s+M(Aa_1d_v m_0+A'a_1d_v+Ba_2d_u n_0+B'a_2d_u)(d^*)^{-1}s}\notag\\
	&\quad\times q^{M\big(A\binom{b_1d_v(d^*)^{-1}}{2}+B\binom{b_2d_u(d^*)^{-1}}{2}\big)t+M(Ab_1d_v m_0+A'b_1d_v+Bb_2d_u n_0+B'b_2d_u)(d^*)^{-1}t}\notag\\
	&\quad\times q^{M(Aa_1 b_1d_v^2+Ba_2 b_2d_u^2)(d^*)^{-2}st}\notag\\
	&\quad\times q^{(ua_1d_v+va_2d_u)(d^*)^{-1}s}\cdot q^{(ub_1d_v+vb_2d_u)(d^*)^{-1}t}.
\end{align}

\begin{theorem}\label{th:UMH}
	Let $\sH(q)$ be as in \eqref{eq:H-sum-form}, namely,
	\begin{align*}
		\sH(q)&=\sum_{m,n\in\mathbb{Z}}(-1)^{\kappa m+\lambda n}q^{AM\binom{m}{2}+A'Mm+BM\binom{n}{2}+B'Mn+um+vn+w}\\
		&= q^{w}f\big((-1)^{\kappa}q^{u+A'M},(-1)^{\kappa}q^{-u+(A-A')M}\big)\notag\\
		&\quad\times f\big((-1)^{\lambda}q^{v+B'M},(-1)^{\lambda}q^{-v+(B-B')M}\big).
	\end{align*}
	Let $d^*=\gcd(u,v,M)$, $d=\gcd(u,v)$, $d_u=\gcd(u,M)$ and $d_v=\gcd(v,M)$.
	\begin{enumerate}[label=\textup{(\roman*)}, widest=ii, itemindent=*, leftmargin=*]
		
		\item
		Assume that $d^*\nmid w$. Then
		\begin{align}
			\mathbf{H}_M\big(\sH(q)\big) = 0.
		\end{align}
		
		\item 
		Assume that $d^*\mid w$. Let $(m_0,n_0)$ be a pair of integers such that
		\begin{align}\label{eq:m0n0}
			um_0+vn_0+w\equiv 0 \pmod{M}.
		\end{align}
		Assume that
		\begin{align}\label{eq:div-assump}
			\left\{
			\begin{aligned}
				d_u(Av^2+Bu^2)&\mid Av\cdot dM,\\
				d_v(Av^2+Bu^2)&\mid Bu\cdot dM.
			\end{aligned}
			\right.
		\end{align}
		Define
		\begin{align*}
			\tA&=\dfrac{ABd^2M^2}{(d^*)^2(Av^2+Bu^2)},\\
			\tA'&=\dfrac{ABd^2M^2}{2(d^*)^2(Av^2+Bu^2)}-\dfrac{AB(u+v)dM}{2d^*(Av^2+Bu^2)}\notag\\
			&\quad+\dfrac{(A'Bu+AB'v)dM}{d^*(Av^2+Bu^2)}+\dfrac{AB(um_0+vn_0)dM}{d^*(Av^2+Bu^2)},\\
			\tB&=\dfrac{Av^2+Bu^2}{d^2},\\
			\tB'&=\dfrac{Av^2+Bu^2}{2d^2}+\dfrac{Av-Bu}{2d}-\dfrac{A'v-B'u}{d}-\dfrac{Avm_0-Bun_0}{d}.
		\end{align*}
		Then
		\begin{align}
			&\mathbf{H}_M\big(\sH(q)\big)\notag\\
			&=(-1)^{\kappa m_0+\lambda n_0}q^{M\big(A\binom{m_0}{2}+A'm_0+B\binom{n_0}{2}+B'n_0\big)+(um_0+vn_0+w)}\notag\\
			&\quad\times\sum_{s,t\in\mathbb{Z}}(-1)^{\frac{(Bu\kappa+Av\lambda)dM}{d^*(Av^2+Bu^2)} s+\frac{-v\kappa+u\lambda}{d} t} q^{\tA M\binom{s}{2}+\tA'Ms+\tB M\binom{t}{2}+\tB'Mn+d(d^*)^{-1}Ms}\label{eq:UM-1}\\
			&=(-1)^{\kappa m_0+\lambda n_0}q^{M\big(A\binom{m_0}{2}+A'm_0+B\binom{n_0}{2}+B'n_0\big)+(um_0+vn_0+w)}\notag\\
			&\quad\times f\big((-1)^{\frac{(Bu\kappa+Av\lambda)dM}{d^*(Av^2+Bu^2)}}q^{d(d^*)^{-1}M+\tA'M},(-1)^{\frac{(Bu\kappa+Av\lambda)dM}{d^*(Av^2+Bu^2)}}q^{-d(d^*)^{-1}M+(\tA-\tA')M}\big)\notag\\
			&\quad\times f\big((-1)^{\frac{-v\kappa+u\lambda}{d}}q^{\tB'M},(-1)^{\frac{-v\kappa+u\lambda}{d}}q^{(\tB-\tB')M}\big).\label{eq:UM-2}
		\end{align}
		
	\end{enumerate}
\end{theorem}

\begin{proof}
	The first part is a direct consequence of Theorem \ref{th:w}~(i). For the second part, we notice that the pair
	\begin{align*}
		(\alpha,\beta) = \left(\frac{Bu\cdot dM}{d_v(Av^2+Bu^2)},\frac{Av\cdot dM}{d_u(Av^2+Bu^2)}\right)
	\end{align*}
	satisfies
	\begin{align*}
		ud_v\alpha+vd_u\beta= dM.
	\end{align*}
	By \eqref{eq:div-assump}, $\alpha$ and $\beta$ are integers. In view of \eqref{eq:sublattice-particular}, or \eqref{eq:sol-ab}, we may choose
	\begin{align*}
		\left\{
		\begin{aligned}
			a_1&=\alpha,\\
			b_1&= -vd^*/(dd_v),\\
			a_2&=\beta,\\
			b_2&=ud^*/(dd_u),
		\end{aligned}
		\right.
	\end{align*}
	in \eqref{eq:UM(H)-general}. Making such substitutions gives \eqref{eq:UM-1}. Furthermore, \eqref{eq:UM-2} follows from the definition of Ramanujan's theta series.
\end{proof}

\begin{corollary}\label{coro:UMH=0-J}
	Let $\sH(q)$ be as in \eqref{eq:H-sum-form}, namely,
	\begin{align*}
		\sH(q)&=q^{w}f\big((-1)^{\kappa}q^{u+A'M},(-1)^{\kappa}q^{-u+(A-A')M}\big)\notag\\
		&\quad\times f\big((-1)^{\lambda}q^{v+B'M},(-1)^{\lambda}q^{-v+(B-B')M}\big).
	\end{align*}
	Let $d^*=\gcd(u,v,M)$, $d=\gcd(u,v)$, $d_u=\gcd(u,M)$ and $d_v=\gcd(v,M)$. Assume that
	\begin{align}\label{eq:div-assump-coro}
		\left\{
		\begin{aligned}
			d^*&\mid w,\\
			d_u(Av^2+Bu^2)&\mid Av\cdot dM,\\
			d_v(Av^2+Bu^2)&\mid Bu\cdot dM.
		\end{aligned}
		\right.
	\end{align}
	Assume also that
	\begin{align}\label{eq:-1-assump}
		(-1)^{\frac{-v\kappa+u\lambda}{d}} = -1.
	\end{align}
	If there exists an integer $J$ such that
	\begin{align}\label{eq:div-assump-coro-J}
		\left\{
		\begin{aligned}
			2d(Av^2+Bu^2)&\mid (2dMAv\cdot J-2dAvw-dAuv+dBu^2\\
			&\quad+2dA'uv-2dB'u^2-Auv^2-Bu^3),\\
			2d(Av^2+Bu^2)&\mid (2dMBu\cdot J-2dBuw+dAv^2-dBuv\\
			&\quad-2dA'v^2+2dB'uv+Av^3+Bu^2v),
		\end{aligned}
		\right.
	\end{align}
	then
	\begin{align}
		\mathbf{H}_M\big(\sH(q)\big) = 0.
	\end{align}
\end{corollary}

\begin{proof}
	Since \eqref{eq:div-assump-coro} satisfies the assumptions in Theorem \ref{th:UMH}~(ii), we have
	\begin{align*}
		&\mathbf{H}_M\big(\sH(q)\big)\notag\\
		&=(-1)^{\kappa m_0+\lambda n_0}q^{M\big(A\binom{m_0}{2}+A'm_0+B\binom{n_0}{2}+B'n_0\big)+(um_0+vn_0+w)}\notag\\
		&\quad\times f\big((-1)^{\frac{(Bu\kappa+Av\lambda)dM}{d^*(Av^2+Bu^2)}}q^{d(d^*)^{-1}M+\tA'M},(-1)^{\frac{(Bu\kappa+Av\lambda)dM}{d^*(Av^2+Bu^2)}}q^{-d(d^*)^{-1}M+(\tA-\tA')M}\big)\notag\\
		&\quad\times f\big((-1)^{\frac{-v\kappa+u\lambda}{d}}q^{\tB'M},(-1)^{\frac{-v\kappa+u\lambda}{d}}q^{(\tB-\tB')M}\big).
	\end{align*}
	We choose
	\begin{align*}
		m_0&=\frac{2dMBu\cdot J-2dBuw+dAv^2-dBuv-2dA'v^2+2dB'uv+Av^3+Bu^2v}{2d(Av^2+Bu^2)},\\
		n_0&=\frac{2dMAv\cdot J-2dAvw-dAuv+dBu^2+2dA'uv-2dB'u^2-Auv^2-Bu^3}{2d(Av^2+Bu^2)};
	\end{align*}
	both are integers according to the assumptions \eqref{eq:div-assump-coro-J}. We also find that
	\begin{align*}
		um_0+vn_0+w=JM,
	\end{align*}
	so \eqref{eq:m0n0} is satisfied. With such a choice of $m_0$ and $n_0$, as well as the assumption \eqref{eq:-1-assump}, we have
	\begin{align*}
		f\big((-1)^{\frac{-v\kappa+u\lambda}{d}}q^{\tB'M},(-1)^{\frac{-v\kappa+u\lambda}{d}}q^{(\tB-\tB')M}\big)&=f\big({-1},-q^{\tB M}\big)\\
		&=(1,q^{\tB M},q^{\tB M};q^{\tB M})_\infty=0.
	\end{align*}
	This, in consequence, implies the desired result.
\end{proof}

\subsection{The first cancelation}

In this part, we construct a pair of products of two theta functions such that they differ by at most a negative sign under the action of the $\mathbf{H}$-operator.

\begin{theorem}\label{th:cancel-1}
	Let $\sH(q)$ be as in \eqref{eq:H-sum-form}, namely,
	\begin{align*}
		\sH(q)&= q^{w}f\big((-1)^{\kappa}q^{u+A'M},(-1)^{\kappa}q^{-u+(A-A')M}\big)\notag\\
		&\quad\times f\big((-1)^{\lambda}q^{v+B'M},(-1)^{\lambda}q^{-v+(B-B')M}\big).
	\end{align*}
	Also, define
	\begin{align}
		\hat{\sH}(q)& =q^{\hw}f\big((-1)^{\kappa}q^{u+A'M},(-1)^{\kappa}q^{-u+(A-A')M}\big)\notag\\
		&\quad\times f\big((-1)^{\lambda}q^{v+(B-B')M},(-1)^{\lambda}q^{-v+B'M}\big),
	\end{align}
	where
	\begin{align*}
		\hw=w+\left(1-\frac{2B'}{B}\right)v,
	\end{align*}
	provided that $B\mid 2B'v$.
	
	Let $d^*=\gcd(u,v,M)$, $d=\gcd(u,v)$, $d_u=\gcd(u,M)$ and $d_v=\gcd(v,M)$. Then the following statements hold true.
	
	\begin{enumerate}[label=\textup{(\roman*)}, widest=ii, itemindent=*, leftmargin=*]
		
		\item 
		If $d^*B\mid 2B'v$, then $w$ and $\hw$ are simultaneously multiples of $d^*$, or simultaneously nonmultiples of $d^*$. In particular, if $w$ is a nonmultiple of $d^*$, we have
		\begin{align}\label{eq:cancel-1-=0}
			\mathbf{H}_M\big(\sH(q)\big) = \mathbf{H}_M\big(\hat{\sH}(q)\big) = 0.
		\end{align}
		
		\item If $w$ is a multiple of $d^*$, we further let $K$ be such that
		\begin{align}\label{eq:K-cong}
			\left\{
			\begin{aligned}
				K&\equiv 1 \pmod{M/d^*},\\
				K&\equiv 0 \pmod{d/\gcd(u,v,w)}.
			\end{aligned}
			\right.
		\end{align}
		Assume that
		\begin{align}\label{eq:assump-d*}
			\left\{
			\begin{aligned}
				d^*B&\mid 2B'v,\\
				d_u(Av^2+Bu^2)&\mid Av\cdot dM,\\
				d_v(Av^2+Bu^2)&\mid Bu\cdot dM,\\
				(Av^2+Bu^2)&\mid (A-2A')v^2-(B-2B')uv-2Buw\cdot K,\\
				B(Av^2+Bu^2)&\mid B(A-2A')uv+A(B-2B')v^2+2ABvw\cdot K.
			\end{aligned}
			\right.
		\end{align}
		Then
		\begin{align}
			\mathbf{H}_M\big(\sH(q)\big) = (-1)^{\epsilon} \mathbf{H}_M\big(\hat{\sH}(q)\big),
		\end{align}
		where
		\begin{align*}
			\epsilon &= \kappa\cdot \frac{(A-2A')v^2-(B-2B')uv-2Buw\cdot K}{Av^2+Bu^2}\notag\\
			&\quad - \lambda\cdot \frac{B(A-2A')uv+A(B-2B')v^2+2ABvw\cdot K}{B(Av^2+Bu^2)}.
		\end{align*}
	\end{enumerate}
\end{theorem}

\begin{proof}
	If $d^*B\mid 2B'v$, we find that
	\begin{align*}
		\hw=w+v-d^*\cdot \frac{2B'v}{d^*B}.
	\end{align*}
	Noticing that $v$ and $d^*\cdot \frac{2B'v}{d^*B}$ are multiples of $d^*$ yields the first argument in the first part. Further, if $w$ is a nonmultiple of $d^*$, and so is $\hw$, we deduce \eqref{eq:cancel-1-=0} by Theorem \ref{th:UMH}~(i).
	
	For the second part, we know from the first part that $d^*\mid \hw$ as we have assumed that $d^*\mid w$. In Theorem \ref{th:UMH}, we replace $B'$ with $B-B'$ and obtain
	\begin{align}\label{eq:UMHt}
		&\mathbf{H}_M\big(\hat{\sH}(q)\big)\notag\\
		&=(-1)^{\kappa \hm_0+\lambda \hn_0}q^{M\big(A\binom{\hm_0}{2}+A'\hm_0+B\binom{\hn_0}{2}+(B-B')\hn_0\big)+(u\hm_0+v\hn_0+\hw)}\notag\\
		&\quad\times f\big((-1)^{\frac{(Bu\kappa+Av\lambda)dM}{d^*(Av^2+Bu^2)}}q^{d(d^*)^{-1}M+\hat{\tA}'M},(-1)^{\frac{(Bu\kappa+Av\lambda)dM}{d^*(Av^2+Bu^2)}}q^{-d(d^*)^{-1}M+(\hat{\tA}-\hat{\tA}')M}\big)\notag\\
		&\quad\times f\big((-1)^{\frac{-v\kappa+u\lambda}{d}}q^{\hat{\tB}'M},(-1)^{\frac{-v\kappa+u\lambda}{d}}q^{(\hat{\tB}-\hat{\tB}')M}\big),
	\end{align}
	where
	\begin{align*}
		\hat{\tA}&=\dfrac{ABd^2M^2}{(d^*)^2(Av^2+Bu^2)},\\
		\hat{\tA}'&=\dfrac{ABd^2M^2}{2(d^*)^2(Av^2+Bu^2)}-\dfrac{AB(u+v)dM}{2d^*(Av^2+Bu^2)}\notag\\
		&\quad+\dfrac{(A'Bu+A(B-B')v)dM}{d^*(Av^2+Bu^2)}+\dfrac{AB(u\hm_0+v\hn_0)dM}{d^*(Av^2+Bu^2)},\\
		\hat{\tB}&=\dfrac{Av^2+Bu^2}{d^2},\\
		\hat{\tB}'&=\dfrac{Av^2+Bu^2}{2d^2}+\dfrac{Av-Bu}{2d}-\dfrac{A'v-(B-B')u}{d}-\dfrac{Av\hm_0-Bu\hn_0}{d},
	\end{align*}
	and $\hm_0$ and $\hn_0$ are such that
	\begin{align}\label{eq:hm0hn0}
		u\hm_0+v\hn_0+\hw \equiv 0 \pmod{M}.
	\end{align}
	Recalling from \eqref{eq:m0n0}, $m_0$ and $n_0$ are such that
	\begin{align*}
		um_0+vn_0+w\equiv 0 \pmod{M}.
	\end{align*}
	Since $w$ is a multiple of $d^*=\gcd(u,v,M)$, the above is equivalent to
	\begin{align}\label{eq:m0n0-cong-assump}
		\frac{u}{d^*}m_0+\frac{v}{d^*}n_0+\frac{w}{d^*}\equiv 0 \pmod{M/d^*}.
	\end{align}
	The assumption in \eqref{eq:K-cong} that $K$ is a multiple of $d/\gcd(u,v,w)$ implies that
	\begin{align*}
		\frac{wK}{d} = \frac{w}{\gcd(u,v,w)}\cdot \frac{K}{d/\gcd(u,v,w)}
	\end{align*}
	is an integer. We then choose $m_0$ and $n_0$ such that
	\begin{align}\label{eq:sol-m0n0}
		\frac{u}{d}m_0+\frac{v}{d}n_0=-\frac{wK}{d};
	\end{align}
	this linear Diophantine equation in $m_0$ and $n_0$ is solvable since $u/d$ and $v/d$ are coprime. Also, we may always find such a $K$ since $M/d^*$ is coprime to $d/\gcd(u,v,w)$. This is because $d/\gcd(u,v,w)$ is a divisor of $d/d^*$ by recalling that $w$ is a multiple of $d^*$, while in the meantime $M/d^*$ and $d/d^*$ are coprime. Thus, by \eqref{eq:K-cong} that $K\equiv 1 \pmod{M/d^*}$, we have
	\begin{align*}
		\frac{u}{d^*}m_0+\frac{v}{d^*}n_0 = -\frac{wK}{d}\cdot \frac{d}{d^*} = -K \frac{w}{d^*}\equiv -\frac{w}{d^*} \pmod{M/d^*},
	\end{align*}
	confirming \eqref{eq:m0n0-cong-assump}. Now we choose
	\begin{align*}
		\hm_0&=-m_0+\frac{(A-2A')v^2-(B-2B')uv-2Buw\cdot K}{Av^2+Bu^2},\\
		\hn_0&=-n_0-\frac{B(A-2A')uv+A(B-2B')v^2+2ABvw\cdot K}{B(Av^2+Bu^2)}.
	\end{align*}
	By \eqref{eq:assump-d*}, the above two are integers. We also want to point out that the choice of $K$ will not affect the last two divisibility conditions in \eqref{eq:assump-d*} and the value of $(-1)^{\epsilon}$. This is because these $K$ differ by a multiple of $dM/(d^*\gcd(u,v,w))$. Further, we notice that
	\begin{align*}
		\frac{2Avw\cdot \frac{dM}{d^*\gcd(u,v,w)}}{Av^2+Bu^2}&=2\cdot \frac{w}{\gcd(u,v,w)}\cdot \frac{Av\cdot dM}{d_u(Av^2+Bu^2)}\cdot \frac{d_u}{d^*},\\
		\frac{2Buw\cdot \frac{dM}{d^*\gcd(u,v,w)}}{Av^2+Bu^2}&=2\cdot \frac{w}{\gcd(u,v,w)}\cdot \frac{Bu\cdot dM}{d_v(Av^2+Bu^2)}\cdot \frac{d_v}{d^*},
	\end{align*}
	so both are even integers. Next, since $w$ is a multiple of $d^*$ and $K\equiv 1 \pmod{M/d^*}$, we have
	\begin{align*}
		u \hm_0+ v \hn_0 &= -Kw - \left(1-\frac{2B'}{B}\right)v\\
		&\equiv -w - \left(1-\frac{2B'}{B}\right)v=-\hw \pmod{M}.
	\end{align*}
	Thus, \eqref{eq:hm0hn0} is satisfied. With such a choice of $\hm_0$ and $\hn_0$, we find that
	\begin{align*}
		\left\{
		\begin{aligned}
			\hat{\tA}&=\tA,\\
			\hat{\tA}'&=\tA',
		\end{aligned}
		\right.
		\qquad\text{and}\qquad
		\left\{
		\begin{aligned}
			\hat{\tB}&=\tB,\\
			\hat{\tB}'&=\tB-\tB'.
		\end{aligned}
		\right.
	\end{align*}
	Also,
	\begin{align*}
		A\binom{\hm_0}{2}+A'\hm_0+B\binom{\hn_0}{2}&+(B-B')\hn_0\\
		&= A\binom{m_0}{2}+A'm_0+B\binom{n_0}{2}+B'n_0,
	\end{align*}
	and
	\begin{align*}
		u\hm_0+v\hn_0+\hw = um_0+vn_0+w.
	\end{align*}
	Recall that
	\begin{align*}
		\epsilon &= \kappa\cdot \frac{(A-2A')v^2-(B-2B')uv-2Buw\cdot K}{Av^2+Bu^2}\\
		&\quad - \lambda\cdot \frac{B(A-2A')uv+A(B-2B')v^2+2ABvw\cdot K}{B(Av^2+Bu^2)}.
	\end{align*}
	Substituting the above into \eqref{eq:UMHt} gives
	\begin{align*}
		&\mathbf{H}_M\big(\hat{\sH}(q)\big)\notag\\
		&=(-1)^{\epsilon}\cdot (-1)^{\kappa m_0+\lambda n_0}q^{M\big(A\binom{m_0}{2}+A'm_0+B\binom{n_0}{2}+B'n_0\big)+(um_0+vn_0+w)}\notag\\
		&\quad\times f\big((-1)^{\frac{(Bu\kappa+Av\lambda)dM}{d^*(Av^2+Bu^2)}}q^{d(d^*)^{-1}M+\tA'M},(-1)^{\frac{(Bu\kappa+Av\lambda)dM}{d^*(Av^2+Bu^2)}}q^{-d(d^*)^{-1}M+(\tA-\tA')M}\big)\notag\\
		&\quad\times f\big((-1)^{\frac{-v\kappa+u\lambda}{d}}q^{(\tB-\tB')M},(-1)^{\frac{-v\kappa+u\lambda}{d}}q^{\tB'M}\big).
	\end{align*}
	Finally, it follows from \eqref{eq:UM-2} that
	\begin{align*}
		\mathbf{H}_M\big(\hat{\sH}(q)\big) = (-1)^{\epsilon} \mathbf{H}_M\big(\sH(q)\big).
	\end{align*}
	This implies our desired result.
\end{proof}

\subsection{The second cancelation}

We construct another pair of products of two theta functions such that they differ by at most a negative sign under the action of the $\mathbf{H}$-operator.

\begin{theorem}\label{th:cancel-2}
	Let $\sH(q)$ be as in \eqref{eq:H-sum-form}, namely,
	\begin{align*}
		\sH(q)&= q^{w}f\big((-1)^{\kappa}q^{u+A'M},(-1)^{\kappa}q^{-u+(A-A')M}\big)\notag\\
		&\quad\times f\big((-1)^{\lambda}q^{v+B'M},(-1)^{\lambda}q^{-v+(B-B')M}\big).
	\end{align*}
	Also, define
	\begin{align}
		\check{\sH}(q)& =q^{\chw}f\big((-1)^{\kappa}q^{u+(A-A')M},(-1)^{\kappa}q^{-u+A'M}\big)\notag\\
		&\quad\times f\big((-1)^{\lambda}q^{v+(B-B')M},(-1)^{\lambda}q^{-v+B'M}\big),
	\end{align}
	where
	\begin{align*}
		\chw=w+\left(1-\frac{2A'}{A}\right)u+\left(1-\frac{2B'}{B}\right)v,
	\end{align*}
	provided that $AB \mid 2(A'Bu+AB'v)$.
	
	Let $d^*=\gcd(u,v,M)$, $d=\gcd(u,v)$, $d_u=\gcd(u,M)$ and $d_v=\gcd(v,M)$. Then the following statements hold true.
	
	\begin{enumerate}[label=\textup{(\roman*)}, widest=ii, itemindent=*, leftmargin=*]
		
		\item 
		If $d^*AB \mid 2(A'Bu+AB'v)$, then $w$ and $\chw$ are simultaneously multiples of $d^*$, or simultaneously nonmultiples of $d^*$. In particular, if $w$ is a nonmultiple of $d^*$, we have
		\begin{align}
			\mathbf{H}_M\big(\sH(q)\big) = \mathbf{H}_M\big(\check{\sH}(q)\big) = 0.
		\end{align}
		
		\item
		If $w$ is a multiple of $d^*$, we further let $K$ be such that
		\begin{align*}
			\left\{
			\begin{aligned}
				K&\equiv 1 \pmod{M/d^*},\\
				K&\equiv 0 \pmod{d/\gcd(u,v,w)}.
			\end{aligned}
			\right.
		\end{align*}
		Assume that
		\begin{align}\label{eq:assump-2-d*}
			\left\{
			\begin{aligned}
				d^*AB &\mid 2(A'Bu+AB'v),\\
				d_u(Av^2+Bu^2)&\mid Av\cdot dM,\\
				d_v(Av^2+Bu^2)&\mid Bu\cdot dM,\\
				A(Av^2+Bu^2)&\mid B(A-2A')u^2+A(B-2B')uv+2ABuw\cdot K,\\
				B(Av^2+Bu^2)&\mid B(A-2A')uv+A(B-2B')v^2+2ABvw\cdot K.
			\end{aligned}
			\right.
		\end{align}
		Then
		\begin{align}
			\mathbf{H}_M\big(\sH(q)\big) = (-1)^{\varepsilon} \mathbf{H}_M\big(\check{\sH}(q)\big),
		\end{align}
		where
		\begin{align*}
			\varepsilon &= -\kappa\cdot \frac{B(A-2A')u^2+A(B-2B')uv+2ABuw\cdot K}{A(Av^2+Bu^2)}\notag\\
			&\quad - \lambda\cdot \frac{B(A-2A')uv+A(B-2B')v^2+2ABvw\cdot K}{B(Av^2+Bu^2)}.
		\end{align*}
	\end{enumerate}
\end{theorem}

\begin{proof}
	The proof is almost identical to that for Theorem \ref{th:cancel-1}. For the first part, we only need to notice that
	\begin{align*}
		\chw=w+u+v-d^*\cdot \frac{2(A'Bu+AB'v)}{d^*AB}.
	\end{align*}
	For the second part, the first difference is that we have to replace $A'$ with $A-A'$ and $B'$ with $B-B'$ in Theorem \ref{th:UMH}. Then
	\begin{align*}
		&\mathbf{H}_M\big(\check{\sH}(q)\big)\notag\\
		&=(-1)^{\kappa \chm_0+\lambda \chn_0}q^{M\big(A\binom{\chm_0}{2}+(A-A')\chm_0+B\binom{\chn_0}{2}+(B-B')\chn_0\big)+(u\chm_0+v\chn_0+\chw)}\notag\\
		&\quad\times f\big((-1)^{\frac{(Bu\kappa+Av\lambda)dM}{d^*(Av^2+Bu^2)}}q^{d(d^*)^{-1}M+\check{\tA}'M},(-1)^{\frac{(Bu\kappa+Av\lambda)dM}{d^*(Av^2+Bu^2)}}q^{-d(d^*)^{-1}M+(\check{\tA}-\check{\tA}')M}\big)\notag\\
		&\quad\times f\big((-1)^{\frac{-v\kappa+u\lambda}{d}}q^{\check{\tB}'M},(-1)^{\frac{-v\kappa+u\lambda}{d}}q^{(\check{\tB}-\check{\tB}')M}\big),
	\end{align*}
	where
	\begin{align*}
		\check{\tA}&=\dfrac{ABd^2M^2}{(d^*)^2(Av^2+Bu^2)},\\
		\check{\tA}'&=\dfrac{ABd^2M^2}{2(d^*)^2(Av^2+Bu^2)}-\dfrac{AB(u+v)dM}{2d^*(Av^2+Bu^2)}\notag\\
		&\quad+\dfrac{((A-A')Bu+A(B-B')v)dM}{d^*(Av^2+Bu^2)}+\dfrac{AB(u\chm_0+v\chn_0)dM}{d^*(Av^2+Bu^2)},\\
		\check{\tB}&=\dfrac{Av^2+Bu^2}{d^2},\\
		\check{\tB}'&=\dfrac{Av^2+Bu^2}{2d^2}+\dfrac{Av-Bu}{2d}-\dfrac{(A-A')v-(B-B')u}{d}-\dfrac{Av\chm_0-Bu\chn_0}{d},
	\end{align*}
	and $\chm_0$ and $\chn_0$ are such that
	\begin{align*}
		u\chm_0+v\chn_0+\chw \equiv 0 \pmod{M}.
	\end{align*}
	The second and major difference comes when we determine $\chm_0$ and $\chn_0$ such that the above congruence holds true. Let $m_0$ and $n_0$ be as in \eqref{eq:sol-m0n0}. This time we choose
	\begin{align*}
		\chm_0&=-m_0-\frac{B(A-2A')u^2+A(B-2B')uv+2ABuw\cdot K}{A(Av^2+Bu^2)},\\
		\chn_0&=-n_0-\frac{B(A-2A')uv+A(B-2B')v^2+2ABvw\cdot K}{B(Av^2+Bu^2)};
	\end{align*}
	such a choice ensures the above congruence for $\chm_0$ and $\chn_0$. We then find that
	\begin{align*}
		\left\{
		\begin{aligned}
			\check{\tA}&=\tA,\\
			\check{\tA}'&=\tA',
		\end{aligned}
		\right.
		\qquad\text{and}\qquad
		\left\{
		\begin{aligned}
			\check{\tB}&=\tB,\\
			\check{\tB}'&=\tB-\tB'.
		\end{aligned}
		\right.
	\end{align*}
	Also,
	\begin{align*}
		A\binom{\chm_0}{2}+(A-A')\chm_0+B\binom{\chn_0}{2}&+(B-B')\chn_0\\
		&= A\binom{m_0}{2}+A'm_0+B\binom{n_0}{2}+B'n_0,
	\end{align*}
	and
	\begin{align*}
		u\chm_0+v\chn_0+\chw = um_0+vn_0+w.
	\end{align*}
	The desired result therefore follows.
\end{proof}

\section{Coefficient-vanishing results}\label{sec:proof}

\subsection{Type I --- Theorem \ref{th:Type-I}}

We consider in this section the three coefficient-vanishing results in Theorem \ref{th:Type-I}. In principle, our proof is based on a pairing-and-cancelation process. To begin with, the following preparation is necessary.

First, by Corollary \ref{coro:power-pairing} with certain common factors extracted and powers of $(-1)$ modified, we have, for any $\ell\ge 1$,
\begin{align*}
	f\big((-1)^{\kappa}q^{ak},(-1)^{\kappa}q^{-ak+M\mu}\big)^{\ell}=\sum_{s=0}^{\ell-1}\sA_s \sF_{s},
\end{align*}
where each $\sF_{\star}$ is a series in $q^M$, and
\begin{align}\label{eq:sA0}
	\sA_0 = f\big((-1)^{\kappa \ell}q^{ak\ell},(-1)^{\kappa \ell}q^{-ak\ell+M\mu \ell}\big),
\end{align}
and for $1\le s\le \ell-1$,
\begin{align*}
	\sA_s&=q^{aks} f\big((-1)^{\kappa \ell}q^{ak\ell+M\mu s},(-1)^{\kappa \ell}q^{-ak\ell+M\mu (\ell-s)}\big)\notag\\
	&\quad+(-1)^{\kappa \ell}q^{ak(\ell-s)} f\big((-1)^{\kappa \ell}q^{ak\ell+M\mu (\ell-s)},(-1)^{\kappa \ell}q^{-ak\ell+M\mu s}\big)\notag\\
	&=:\sA_{s,I}+(-1)^{\kappa \ell}\sA_{s,I\!I}.
\end{align*}
More generally, we consider
\begin{align}
	\sA_{I}&:=q^{a\xi k} f\big((-1)^{\kappa \ell}q^{ak\ell+M\mu \xi},(-1)^{\kappa \ell}q^{-ak\ell+M\mu (\ell-\xi)}\big),\label{eq:sAI}\\
	\sA_{I\!I}&:=q^{a(\ell-\xi) k} f\big((-1)^{\kappa \ell}q^{ak\ell+M\mu (\ell-\xi)},(-1)^{\kappa \ell}q^{-ak\ell+M\mu \xi}\big),\label{eq:sAII}
\end{align}
for generic $\xi\in\mathbb{Z}$.

Also, by Corollary \ref{coro:f/f-expansion}, we have, for any $m\ge 1$,
\begin{align*}
	&\left(\frac{f\big({-q^{2k}},-q^{-2k+M\mu}\big)}{f\big((-1)^{\lambda}q^k,(-1)^{\lambda}q^{-k+M\mu}\big)}\right)^m \\
	&\quad= \sB_0^{(0)}\sG_0^{(0)}+\sum_{t=1}^{m-1}\sB_{t,1}^{(0)} \sG_{t,1}^{(0)}+\sum_{t=1}^{m-1}\sB_{t,2}^{(0)} \sG_{t,2}^{(0)} +\sum_{m_1=1}^{m-1}\sum_{t=0}^{m-1} \sB_t^{(m_1)} \sG_t^{(m_1)},
\end{align*}
where each $\sG_{\star}^{(\star)}$ is a series in $q^M$, and
\begin{align*}
	\sB_{0}^{(0)}&=f\big((-1)^{\lambda m}q^{3km+M\mu m},(-1)^{\lambda m}q^{-3km+2M\mu m}\big)\notag\\
	&\quad +(-1)^{(\lambda+1)m}q^{km} f\big((-1)^{\lambda m}q^{3km+2M\mu m},(-1)^{\lambda m}q^{-3km+M\mu m}\big)\\
	&=: \sB_{0,I}^{(0)}+(-1)^{(\lambda+1)m}\sB_{0,I\!I}^{(0)},
\end{align*}
and for $1\le t\le m-1$,
\begin{align*}
	\sB_{t,1}^{(0)}&=q^{3kt} f\big((-1)^{\lambda m}q^{3km+M\mu(m+3 t)},(-1)^{\lambda m}q^{-3km+M\mu(2m-3 t)}\big)\notag\\
	&\quad+(-1)^{(\lambda+1)m}q^{km-3kt} f\big((-1)^{\lambda m}q^{3km+M\mu(2m-3 t)},(-1)^{\lambda m}q^{-3km+M\mu(m+3 t)}\big)\\
	&=: \sB_{t,1,I}^{(0)}+(-1)^{(\lambda+1)m}\sB_{t,1,I\!I}^{(0)},\notag\\
	\sB_{t,2}^{(0)}&=q^{km-3k(m-t)} f\big((-1)^{\lambda m}q^{3km+M\mu(-m+3 t)},(-1)^{\lambda m}q^{-3km+M\mu(4m-3 t)}\big)\notag\\
	&\quad+(-1)^{(\lambda+1)m}q^{3k(m-t)} f\big((-1)^{\lambda m}q^{3km+M\mu(4m-3 t)},(-1)^{\lambda m}q^{-3km+M\mu(-m+3 t)}\big)\\
	&=: \sB_{t,2,I}^{(0)}+(-1)^{(\lambda+1)m}\sB_{t,2,I\!I}^{(0)},
\end{align*}
and for $1\le m_1\le m-1$ and $0\le t\le m-1$,
\begin{align*}
	\sB_{t}^{(m_1)}&=q^{k(m-m_1)+3kt}\notag\\
	&\quad\quad\times f\big((-1)^{\lambda m}q^{3km+M\mu(2m-m_1+3 t)},(-1)^{\lambda m}q^{-3km+M\mu(m+m_1-3t)}\big)\notag\\
	&\quad+(-1)^{(\lambda+1)m}q^{km_1-3kt}\notag\\
	&\quad\quad\times f\big((-1)^{\lambda m}q^{3km+M\mu(m+m_1-3t)},(-1)^{\lambda m}q^{-3km+M\mu(2m-m_1+3 t)}\big)\\
	&=: \sB_{t,I}^{(m_1)}+(-1)^{(\lambda+1)m}\sB_{t,I\!I}^{(m_1)}.
\end{align*}
We point out that these $\sB_{\star,I}^{(\star)}$ and $\sB_{\star,I\!I}^{(\star)}$ are special cases of
\begin{align}
	\sB_{I}&:=q^{\tau k} f\big((-1)^{\lambda m}q^{3km+M\mu(m+\tau)},(-1)^{\lambda m}q^{-3km+M\mu(2m-\tau)}\big),\label{eq:sBI}\\
	\sB_{I\!I}&:=q^{(m-\tau)k} f\big((-1)^{\lambda m}q^{3km+M\mu(2m-\tau)},(-1)^{\lambda m}q^{-3km+M\mu(m+\tau)}\big),\label{eq:sBII}
\end{align}
for generic $\tau\in\mathbb{Z}$.

With all parameters chosen in each case below, our generic target is to show the following by Theorems \ref{th:cancel-1} and \ref{th:cancel-2}:
\begin{align}\label{eq:A0B}
	\mathbf{H}_M\big(q^{\sigma}\sA_0\sB_{I}\big) =-(-1)^{(\lambda+1)m}\mathbf{H}_M\big(q^{\sigma}\sA_0\sB_{I\!I}\big),
\end{align}
and
\begin{align}\label{eq:AsB}
	\mathbf{H}_M\big(q^{\sigma}\sA_{I}\sB_{I}\big) =-(-1)^{\kappa \ell+(\lambda+1)m}\mathbf{H}_M\big(q^{\sigma}\sA_{I\!I}\sB_{I\!I}\big).
\end{align}
Upon setting $\tau\mapsto m-\tau$ in \eqref{eq:AsB}, we automatically have
\begin{align}\label{eq:AsB-2}
	\mathbf{H}_M\big(q^{\sigma}\sA_{I}\sB_{I\!I}\big) =-(-1)^{\kappa \ell+(\lambda+1)m}\mathbf{H}_M\big(q^{\sigma}\sA_{I\!I}\sB_{I}\big).
\end{align}
Once the above relations are established, it is safe to conclude by our pairing that
\begin{align}
	\mathbf{H}_{M}\Bigg(q^{\sigma}f\big((-1)^{\kappa}q^{ak},(-1)^{\kappa}q^{-ak+ M\mu}\big)^{\ell}\bigg(\frac{f\big({-q^{2k}},-q^{-2k+ M\mu}\big)}{f\big((-1)^{\lambda}q^{k},(-1)^{\lambda}q^{-k+ M\mu}\big)}\bigg)^{m}\Bigg)=0.
\end{align}

\subsubsection{Equation \eqref{eq:I-1-e-o}}

Recall that $(\ell,m)\mapsto (2\ell,2m+1)$ and $a=1$. Also, $M=2\ell+6m+3$ and $\sigma=-(2\ell+4m+2)k$. Further, $\kappa\in\{0,1\}$, $\lambda\in\{0,1\}$ and $k$ is such that $\gcd(k,M)=1$.

Notice that the $\ell=0$ case is shown by Corollary \ref{coro:only-quotient-power}. Below, we assume that $\ell\ge 1$.

In Theorems \ref{th:cancel-1} and \ref{th:cancel-2}, we set
\begin{align*}
	\begin{array}{rclp{3cm}rcl}
		\kappa & \mapsto & (2\ell)\kappa && \lambda & \mapsto & (2m+1)\lambda\\
		u & \mapsto & (2\ell)k && v & \mapsto & 3(2m+1)k\\
		A & \mapsto & (2\ell)\mu && B & \mapsto & 3(2m+1)\mu\\
		M & \mapsto & \multicolumn{5}{l}{2\ell+6m+3}\\
	\end{array}
\end{align*}
Let
\begin{align*}
	d_0=\gcd\big(2\ell,3(2m+1)\big).
\end{align*}
Then
\begin{align*}
	d=\gcd(u,v)=d_0 k.
\end{align*}
Also, noticing that $M=(2\ell)+3(2m+1)$ and that $k$ is coprime to $M$, we have
\begin{align*}
	d^*&=\gcd(u,v,M)=d_0,\\
	d_u&=\gcd(u,M)=d_0,\\
	d_v&=\gcd(v,M)=d_0.
\end{align*}
We compute that
\begin{align*}
	\frac{Av\cdot dM}{d_u(Av^2+Bu^2)}&=1,\\
	\frac{Bu\cdot dM}{d_v(Av^2+Bu^2)}&=1.
\end{align*}
Thus, the second and third assumptions in \eqref{eq:assump-d*} and \eqref{eq:assump-2-d*} are satisfied.

\begin{itemize}[leftmargin=*,align=left,itemsep=5pt]
	\renewcommand{\labelitemi}{\scriptsize$\blacktriangleright$}
	
	\item
	\textit{Examining \eqref{eq:A0B}}. Let us keep in mind that we have made the substitutions in \eqref{eq:sA0}, \eqref{eq:sBI} and \eqref{eq:sBII}: $(\ell,m)\mapsto (2\ell,2m+1)$ and $a=1$. Also, $\sigma=-(2\ell+4m+2)k$. Then \eqref{eq:A0B} becomes
	\begin{align}\label{eq:A0B-Eq1}
		\mathbf{H}_M\big(q^{\sigma}\sA_0\sB_{I}\big) =-(-1)^{\lambda+1}\mathbf{H}_M\big(q^{\sigma}\sA_0\sB_{I\!I}\big).
	\end{align}
	In Theorem \ref{th:cancel-1}, we further set
	\begin{align*}
		\begin{array}{rclp{4cm}rcl}
			A' & \mapsto & 0 && B' & \mapsto & (2m+1+\tau)\mu\\
			w & \mapsto & \multicolumn{5}{l}{-(2\ell+4m+2)k+\tau k}\\
		\end{array}
	\end{align*}
	Then
	\begin{align*}
		\hw = -(2\ell+4m+2)k +(2m+1-\tau) k.
	\end{align*}
	So,
	\begin{align*}
		\sH(q)&=q^{\sigma}\sA_0\sB_{I},\\
		\hat{\sH}(q)&=q^{\sigma}\sA_0\sB_{I\!I}.
	\end{align*}
	Now we compute that
	\begin{align*}
		\frac{2B'v}{d^* B} = \frac{2(2m+1+\tau )k}{d^*}.
	\end{align*}
	If $d^* \nmid 2(2m+1+\tau )k$, then $d_0 \nmid (2m+1+\tau )k$. By noticing that $d_0 \mid (2\ell)$ and $d_0\mid 3(2m+1)$, we have
	\begin{align*}
		w&\equiv (2m+1+\tau )k \not\equiv 0 \pmod{d_0},\\
		\hw&\equiv -(2m+1+\tau )k \not\equiv 0 \pmod{d_0}.
	\end{align*}
	So $w$ and $\hw$ are nonmultiples of $d_0=d^*$. In this case we know from Theorem \ref{th:UMH}~(i) that
	\begin{align*}
		\mathbf{H}_M\big(\sH(q)\big) = \mathbf{H}_M\big(\hat{\sH}(q)\big) = 0,
	\end{align*}
	which gives
	\begin{align*}
		\mathbf{H}_M\big(q^{\sigma}\sA_0\sB_{I}\big) =-(-1)^{\lambda+1}\mathbf{H}_M\big(q^{\sigma}\sA_0\sB_{I\!I}\big)=0.
	\end{align*}
	Below, we assume that $d^* \mid 2(2m+1+\tau )k$. Then the first assumption in \eqref{eq:assump-d*} is satisfied. We may also assume that $w$ is a multiple of $d^*=d_0$ according to Theorem \ref{th:cancel-1}~(i). It is easily seen that $d^*k$ is a divisor of $\gcd(u,v,w)$. Thus, we solve the following stronger system
	\begin{align*}
		\left\{
		\begin{aligned}
			K&\equiv 1 \pmod{M/d^*},\\
			K&\equiv 0 \pmod{d/(d^*k)},
		\end{aligned}
		\right.
	\end{align*}
	and choose $K=1$. Finally, we compute that
	\begin{align*}
		\frac{(A-2A')v^2-(B-2B')uv-2Buw\cdot K}{Av^2+Bu^2}&=2,\\
		\frac{B(A-2A')uv+A(B-2B')v^2+2ABvw\cdot K}{B(Av^2+Bu^2)}&=-1.
	\end{align*}
	So the remaining assumptions in \eqref{eq:assump-d*} are satisfied. Finally, we compute that
	\begin{align*}
		\epsilon=(4\ell)\kappa+(2m+1)\lambda.
	\end{align*}
	Thus, by Theorem \ref{th:cancel-1}~(ii),
	\begin{align*}
		\mathbf{H}_M\big(\sH(q)\big) = (-1)^{\lambda}\mathbf{H}_M\big(\hat{\sH}(q)\big).
	\end{align*}
	It follows that
	\begin{align*}
		\mathbf{H}_M\big(q^{\sigma}\sA_0\sB_{I}\big)+(-1)^{\lambda+1}\mathbf{H}_M\big(q^{\sigma}\sA_0\sB_{I\!I}\big)&=\mathbf{H}_M\big(\sH(q)\big)+(-1)^{\lambda+1}\mathbf{H}_M\big(\hat{\sH}(q)\big)\\
		&=0.
	\end{align*}
	Thus, \eqref{eq:A0B-Eq1} is established.
	
	\item
	\textit{Examining \eqref{eq:AsB}}. Let us keep in mind that we have made the substitutions in \eqref{eq:sAI}, \eqref{eq:sAII}, \eqref{eq:sBI} and \eqref{eq:sBII}: $(\ell,m)\mapsto (2\ell,2m+1)$ and $a=1$. Also, $\sigma=-(2\ell+4m+2)k$. Then \eqref{eq:AsB} becomes
	\begin{align}\label{eq:AsB-Eq1}
		\mathbf{H}_M\big(q^{\sigma}\sA_{I}\sB_{I}\big) =-(-1)^{\lambda+1}\mathbf{H}_M\big(q^{\sigma}\sA_{I\!I}\sB_{I\!I}\big).
	\end{align}
	In Theorem \ref{th:cancel-2}, we further set
	\begin{align*}
		\begin{array}{rclp{4.5cm}rcl}
			A' & \mapsto & \xi \mu && B' & \mapsto & (2m+1+\tau)\mu\\
			w & \mapsto & \multicolumn{5}{l}{-(2\ell+4m+2)k+\xi k+\tau k}\\
		\end{array}
	\end{align*}
	Then
	\begin{align*}
		\chw = -(2\ell+4m+2)k +(2\ell-\xi )k +(2m+1-\tau) k.
	\end{align*}
	So,
	\begin{align*}
		\sH(q)&=q^{\sigma}\sA_{I}\sB_{I},\\
		\check{\sH}(q)&=q^{\sigma}\sA_{I\!I}\sB_{I\!I}.
	\end{align*}
	Now we compute that
	\begin{align*}
		\frac{2(A'Bu+AB'v)}{d^*AB} = \frac{2(2m+1+\xi +\tau )k}{d^*}.
	\end{align*}
	If $d^* \nmid 2(2m+1+\xi +\tau )k$, then $d_0 \nmid (2m+1+\xi +\tau )k$, and thus,
	\begin{align*}
		w&\equiv (2m+1+\xi +\tau )k \not\equiv 0 \pmod{d_0},\\
		\chw&\equiv -(2m+1+\xi +\tau )k \not\equiv 0 \pmod{d_0},
	\end{align*}
	implying from Theorem \ref{th:UMH}~(i) that
	\begin{align*}
		\mathbf{H}_M\big(q^{\sigma}\sA_{I}\sB_{I}\big) =-(-1)^{\lambda+1}\mathbf{H}_M\big(q^{\sigma}\sA_{I\!I}\sB_{I\!I}\big)=0.
	\end{align*}
	Below, we assume that $d^* \mid 2(2m+1+\xi +\tau )k$. Then the first assumption in \eqref{eq:assump-2-d*} is satisfied. We may also assume that $w$ is a multiple of $d^*=d_0$ according to Theorem \ref{th:cancel-2}~(i). With the choice of $K=1$, we compute that
	\begin{align*}
		\frac{B(A-2A')u^2+A(B-2B')uv+2ABuw\cdot K}{A(Av^2+Bu^2)}&=-1,\\
		\frac{B(A-2A')uv+A(B-2B')v^2+2ABvw\cdot K}{B(Av^2+Bu^2)}&=-1.
	\end{align*}
	So the remaining assumptions in \eqref{eq:assump-2-d*} are satisfied. Finally, we compute that
	\begin{align*}
		\varepsilon=(2\ell)\kappa+(2m+1)\lambda.
	\end{align*}
	Thus, by Theorem \ref{th:cancel-2}~(ii),
	\begin{align*}
		\mathbf{H}_M\big(\sH(q)\big) = (-1)^{\lambda}\mathbf{H}_M\big(\check{\sH}(q)\big).
	\end{align*}
	It follows that
	\begin{align*}
		\mathbf{H}_M\big(q^{\sigma}\sA_{I}\sB_{I}\big)+(-1)^{\lambda+1}\mathbf{H}_M\big(q^{\sigma}\sA_{I\!I}\sB_{I\!I}\big)&=\mathbf{H}_M\big(\sH(q)\big)+(-1)^{\lambda+1}\mathbf{H}_M\big(\check{\sH}(q)\big)\\
		&=0.
	\end{align*}
	Thus, \eqref{eq:AsB-Eq1} is established.
	
\end{itemize}

\subsubsection{Equation \eqref{eq:I-2-e-o}}

Recall that $(\ell,m)\mapsto (2\ell,2m+1)$ and $a=2$. Also, $M=8\ell+6m+3$ and $\sigma=-(6\ell+4m+2)k$. Further, $\kappa\in\{0,1\}$, $\lambda\in\{0,1\}$ and $k$ is such that $\gcd(k,M)=1$.

Notice that the $\ell=0$ case is shown by Corollary \ref{coro:only-quotient-power}. Below, we assume that $\ell\ge 1$.

In Theorems \ref{th:cancel-1} and \ref{th:cancel-2}, we set
\begin{align*}
\begin{array}{rclp{3cm}rcl}
\kappa & \mapsto & (2\ell)\kappa && \lambda & \mapsto & (2m+1)\lambda\\
u & \mapsto & 2(2\ell)k && v & \mapsto & 3(2m+1)k\\
A & \mapsto & (2\ell)\mu && B & \mapsto & 3(2m+1)\mu\\
M & \mapsto & \multicolumn{5}{l}{8\ell+6m+3}\\
\end{array}
\end{align*}
Let
\begin{align*}
d_0=\gcd\big(2(2\ell),3(2m+1)\big).
\end{align*}
Then
\begin{align*}
d=\gcd(u,v)=d_0 k.
\end{align*}
Also, we observe that $d_0$ must be odd since $d_0\mid 3(2m+1)$. This implies that $d_0\mid (2\ell)$. Finally, noticing that $M=2\cdot 2(2\ell)+3(2m+1)$ and that $k$ is coprime to $M$, we have
\begin{align*}
d^*&=\gcd(u,v,M)=d_0,\\
d_u&=\gcd(u,M)=d_0,\\
d_v&=\gcd(v,M)=d_0.
\end{align*}
We compute that
\begin{align*}
\frac{Av\cdot dM}{d_u(Av^2+Bu^2)}&=1,\\
\frac{Bu\cdot dM}{d_v(Av^2+Bu^2)}&=2.
\end{align*}
Thus, the second and third assumptions in \eqref{eq:assump-d*} and \eqref{eq:assump-2-d*} are satisfied.

\begin{itemize}[leftmargin=*,align=left,itemsep=5pt]
	\renewcommand{\labelitemi}{\scriptsize$\blacktriangleright$}
	
	\item
	\textit{Examining \eqref{eq:A0B}}. Let us keep in mind that we have made the substitutions in \eqref{eq:sA0}, \eqref{eq:sBI} and \eqref{eq:sBII}: $(\ell,m)\mapsto (2\ell,2m+1)$ and $a=2$. Also, $\sigma=-(6\ell+4m+2)k$. Then \eqref{eq:A0B} becomes
	\begin{align}\label{eq:A0B-Eq2}
	H_M\big(q^{\sigma}\sA_0\sB_{I}\big) =-(-1)^{\lambda+1}H_M\big(q^{\sigma}\sA_0\sB_{I\!I}\big).
	\end{align}
	In Theorem \ref{th:cancel-1}, we further set
	\begin{align*}
	\begin{array}{rclp{4cm}rcl}
	A' & \mapsto & 0 && B' & \mapsto & (2m+1+\tau)\mu\\
	w & \mapsto & \multicolumn{5}{l}{-(6\ell+4m+2)k+\tau k}\\
	\end{array}
	\end{align*}
	Then
	\begin{align*}
	\hw = -(6\ell+4m+2)k +(2m+1-\tau) k.
	\end{align*}
	So,
	\begin{align*}
	\sH(q)&=q^{\sigma}\sA_0\sB_{I},\\
	\hat{\sH}(q)&=q^{\sigma}\sA_0\sB_{I\!I}.
	\end{align*}
	Now we compute that
	\begin{align*}
	\frac{2B'v}{d^* B} = \frac{2(2m+1+\tau )k}{d^*}.
	\end{align*}
	If $d^* \nmid 2(2m+1+\tau )k$, then $d_0 \nmid (2m+1+\tau )k$. By noticing that $d_0 \mid (2\ell)$ and $d_0\mid 3(2m+1)$, we have
	\begin{align*}
	w&\equiv (2m+1+\tau )k \not\equiv 0 \pmod{d_0},\\
	\hw&\equiv -(2m+1+\tau )k \not\equiv 0 \pmod{d_0}.
	\end{align*}
	So $w$ and $\hw$ are nonmultiples of $d_0=d^*$. In this case we know from Theorem \ref{th:UMH}~(i) that
	\begin{align*}
	H_M\big(\sH(q)\big) = H_M\big(\hat{\sH}(q)\big) = 0,
	\end{align*}
	which gives
	\begin{align*}
	H_M\big(q^{\sigma}\sA_0\sB_{I}\big) =-(-1)^{\lambda+1}H_M\big(q^{\sigma}\sA_0\sB_{I\!I}\big)=0.
	\end{align*}
	Below, we assume that $d^* \mid 2(2m+1+\tau )k$. Then the first assumption in \eqref{eq:assump-d*} is satisfied. We may also assume that $w$ is a multiple of $d^*=d_0$ according to Theorem \ref{th:cancel-1}~(i). It is easily seen that $d^*k$ is a divisor of $\gcd(u,v,w)$. Thus, we solve the following stronger system
	\begin{align*}
	\left\{
	\begin{aligned}
	K&\equiv 1 \pmod{M/d^*},\\
	K&\equiv 0 \pmod{d/(d^*k)},
	\end{aligned}
	\right.
	\end{align*}
	and choose $K=1$. Finally, we compute that
	\begin{align*}
	\frac{(A-2A')v^2-(B-2B')uv-2Buw\cdot K}{Av^2+Bu^2}&=3,\\
	\frac{B(A-2A')uv+A(B-2B')v^2+2ABvw\cdot K}{B(Av^2+Bu^2)}&=-1.
	\end{align*}
	So the remaining assumptions in \eqref{eq:assump-d*} are satisfied. Finally, we compute that
	\begin{align*}
	\epsilon=(6\ell)\kappa+(2m+1)\lambda.
	\end{align*}
	Thus, by Theorem \ref{th:cancel-1}~(ii),
	\begin{align*}
	H_M\big(\sH(q)\big) = (-1)^{\lambda}H_M\big(\hat{\sH}(q)\big).
	\end{align*}
	It follows that
	\begin{align*}
	H_M\big(q^{\sigma}\sA_0\sB_{I}\big)+(-1)^{\lambda+1}H_M\big(q^{\sigma}\sA_0\sB_{I\!I}\big)&=H_M\big(\sH(q)\big)+(-1)^{\lambda+1}H_M\big(\hat{\sH}(q)\big)\\
	&=0.
	\end{align*}
	Thus, \eqref{eq:A0B-Eq2} is established.
	
	\item
	\textit{Examining \eqref{eq:AsB}}. Let us keep in mind that we have made the substitutions in \eqref{eq:sAI}, \eqref{eq:sAII}, \eqref{eq:sBI} and \eqref{eq:sBII}: $(\ell,m)\mapsto (2\ell,2m+1)$ and $a=2$. Also, $\sigma=-(6\ell+4m+2)k$. Then \eqref{eq:AsB} becomes
	\begin{align}\label{eq:AsB-Eq2}
	H_M\big(q^{\sigma}\sA_{I}\sB_{I}\big) =-(-1)^{\lambda+1}H_M\big(q^{\sigma}\sA_{I\!I}\sB_{I\!I}\big).
	\end{align}
	In Theorem \ref{th:cancel-2}, we further set
	\begin{align*}
	\begin{array}{rclp{4cm}rcl}
	A' & \mapsto & \xi \mu && B' & \mapsto & (2m+1+\tau)\mu\\
	w & \mapsto & \multicolumn{5}{l}{-(6\ell+4m+2)k+2\xi k+\tau k}\\
	\end{array}
	\end{align*}
	Then
	\begin{align*}
	\chw = -(6\ell+4m+2)k +2(2\ell-\xi )k +(2m+1-\tau) k.
	\end{align*}
	So,
	\begin{align*}
	\sH(q)&=q^{\sigma}\sA_{I}\sB_{I},\\
	\check{\sH}(q)&=q^{\sigma}\sA_{I\!I}\sB_{I\!I}.
	\end{align*}
	Now we compute that
	\begin{align*}
	\frac{2(A'Bu+AB'v)}{d^*AB} = \frac{2(2m+1+2\xi +\tau )k}{d^*}.
	\end{align*}
	If $d^* \nmid 2(2m+1+2\xi +\tau )k$, then $d_0 \nmid (2m+1+2\xi +\tau )k$, and thus,
	\begin{align*}
	w&\equiv (2m+1+2\xi +\tau )k \not\equiv 0 \pmod{d_0},\\
	\chw&\equiv -(2m+1+2\xi +\tau )k \not\equiv 0 \pmod{d_0},
	\end{align*}
	implying from Theorem \ref{th:UMH}~(i) that
	\begin{align*}
	H_M\big(q^{\sigma}\sA_{I}\sB_{I}\big) =-(-1)^{\lambda+1}H_M\big(q^{\sigma}\sA_{I\!I}\sB_{I\!I}\big)=0.
	\end{align*}
	Below, we assume that $d^* \mid 2(2m+1+2\xi +\tau )k$. Then the first assumption in \eqref{eq:assump-2-d*} is satisfied. We may also assume that $w$ is a multiple of $d^*=d_0$ according to Theorem \ref{th:cancel-2}~(i). With the choice of $K=1$, we compute that
	\begin{align*}
	\frac{B(A-2A')u^2+A(B-2B')uv+2ABuw\cdot K}{A(Av^2+Bu^2)}&=-2,\\
	\frac{B(A-2A')uv+A(B-2B')v^2+2ABvw\cdot K}{B(Av^2+Bu^2)}&=-1.
	\end{align*}
	So the remaining assumptions in \eqref{eq:assump-2-d*} are satisfied. Finally, we compute that
	\begin{align*}
	\varepsilon=(4\ell)\kappa+(2m+1)\lambda.
	\end{align*}
	Thus, by Theorem \ref{th:cancel-2}~(ii),
	\begin{align*}
	H_M\big(\sH(q)\big) = (-1)^{\lambda}H_M\big(\check{\sH}(q)\big).
	\end{align*}
	It follows that
	\begin{align*}
	H_M\big(q^{\sigma}\sA_{I}\sB_{I}\big)+(-1)^{\lambda+1}H_M\big(q^{\sigma}\sA_{I\!I}\sB_{I\!I}\big)&=H_M\big(\sH(q)\big)+(-1)^{\lambda+1}H_M\big(\check{\sH}(q)\big)\\
	&=0.
	\end{align*}
	Thus, \eqref{eq:AsB-Eq2} is established.
	
\end{itemize}

\subsubsection{Equation \eqref{eq:I-2-o-e}}

Recall that $(\ell,m)\mapsto (2\ell+1,4m+2)$ and $a=2$. Also, $M=4\ell+6m+5$ and $\sigma=-(2\ell+2m+2)k$. Further, $\kappa\in\{1\}$, $\lambda\in\{0,1\}$ and $k$ is such that $\gcd(k,M)=1$.

In Theorems \ref{th:cancel-1} and \ref{th:cancel-2}, we set
\begin{align*}
\begin{array}{rclp{3cm}rcl}
\kappa & \mapsto & (2\ell+1)\kappa && \lambda & \mapsto & (4m+2)\lambda\\
u & \mapsto & 2(2\ell+1)k && v & \mapsto & 3(4m+2)k\\
A & \mapsto & (2\ell+1)\mu && B & \mapsto & 3(4m+2)\mu\\
M & \mapsto & \multicolumn{5}{l}{4\ell+6m+5}\\
\end{array}
\end{align*}
Let
\begin{align*}
d_0=\gcd\big(2(2\ell+1),3(4m+2)\big).
\end{align*}
We observe that $d_0$ is of the form
$$d_0=2\delta_0$$
with $\delta_0$ odd. Also,
\begin{align*}
d=\gcd(u,v)=d_0 k=2\delta_0 k.
\end{align*}
Noticing that $M=2(2\ell+1)+\frac{1}{2}\cdot 3(4m+2)$ is odd, and that $k$ is coprime to $M$, we have
\begin{align*}
d^*&=\gcd(u,v,M)=\delta_0,\\
d_u&=\gcd(u,M)=\delta_0,\\
d_v&=\gcd(v,M)=\delta_0.
\end{align*}
We compute that
\begin{align*}
\frac{Av\cdot dM}{d_u(Av^2+Bu^2)}&=1,\\
\frac{Bu\cdot dM}{d_v(Av^2+Bu^2)}&=2.
\end{align*}
Thus, the second and third assumptions in \eqref{eq:assump-d*} and \eqref{eq:assump-2-d*} are satisfied.

\begin{itemize}[leftmargin=*,align=left,itemsep=5pt]
	\renewcommand{\labelitemi}{\scriptsize$\blacktriangleright$}
	
	\item
	\textit{Examining \eqref{eq:A0B}}. Let us keep in mind that we have made the substitutions in \eqref{eq:sA0}, \eqref{eq:sBI} and \eqref{eq:sBII}: $(\ell,m)\mapsto (2\ell+1,4m+2)$ and $a=2$. Also, $\sigma=-(2\ell+2m+2)k$. Then \eqref{eq:A0B} becomes
	\begin{align}\label{eq:A0B-Eq4}
	H_M\big(q^{\sigma}\sA_0\sB_{I}\big) =-H_M\big(q^{\sigma}\sA_0\sB_{I\!I}\big).
	\end{align}
	In Theorem \ref{th:cancel-1}, we further set
	\begin{align*}
	\begin{array}{rclp{4cm}rcl}
	A' & \mapsto & 0 && B' & \mapsto & (4m+2+\tau)\mu\\
	w & \mapsto & \multicolumn{5}{l}{-(2\ell+2m+2)k+\tau k}\\
	\end{array}
	\end{align*}
	Then
	\begin{align*}
	\hw = -(2\ell+2m+2)k +(4m+2-\tau) k.
	\end{align*}
	So,
	\begin{align*}
	\sH(q)&=q^{\sigma}\sA_0\sB_{I},\\
	\hat{\sH}(q)&=q^{\sigma}\sA_0\sB_{I\!I}.
	\end{align*}
	Now we compute that
	\begin{align*}
	\frac{2B'v}{d^* B} = \frac{2(4m+2+\tau )k}{d^*}.
	\end{align*}
	Assume that $d^* \nmid 2(4m+2+\tau )k$. Recall that $d^*=\delta_0$ with $\delta_0$ odd. Then, $\delta_0 \nmid (4m+2+\tau )k$. Also, by noticing that $d_0 \mid 2(2\ell+1)$ and $d_0\mid 3(4m+2)$, we have $\delta_0 \mid (2\ell+1)$ and $\delta_0 \mid 3(2m+1)$ since $d_0=2\delta_0$. Thus,
	\begin{align*}
	w&\equiv (4m+2+\tau )k \not\equiv 0 \pmod{\delta_0},\\
	\hw&\equiv -(4m+2+\tau )k \not\equiv 0 \pmod{\delta_0}.
	\end{align*}
	So $w$ and $\hw$ are nonmultiples of $\delta_0=d^*$. Now, we know from Theorem \ref{th:UMH}~(i) that
	\begin{align*}
	H_M\big(\sH(q)\big) = H_M\big(\hat{\sH}(q)\big) = 0,
	\end{align*}
	which gives
	\begin{align*}
	H_M\big(q^{\sigma}\sA_0\sB_{I}\big) =-H_M\big(q^{\sigma}\sA_0\sB_{I\!I}\big)=0.
	\end{align*}
	Below, we assume that $d^* \mid 2(4m+2+\tau )k$. Then the first assumption in \eqref{eq:assump-d*} is satisfied. We may also assume that $w$ is a multiple of $d^*=\delta_0$ according to Theorem \ref{th:cancel-1}~(i). It is easily seen that $d^*k$ is a divisor of $\gcd(u,v,w)$. Thus, we solve the following stronger system
	\begin{align*}
	\left\{
	\begin{aligned}
	K&\equiv 1 \pmod{M/d^*},\\
	K&\equiv 0 \pmod{d/(d^*k)},
	\end{aligned}
	\right.
	\end{align*}
	and choose $K=M+1$ by recalling that $M$ is odd. Finally, we compute that
	\begin{align*}
	\frac{(A-2A')v^2-(B-2B')uv-2Buw\cdot K}{Av^2+Bu^2}&=4\ell+4m+5-2\tau,\\
	\frac{B(A-2A')uv+A(B-2B')v^2+2ABvw\cdot K}{B(Av^2+Bu^2)}&=-2\ell-2m-2+\tau.
	\end{align*}
	So the remaining assumptions in \eqref{eq:assump-d*} are satisfied. Finally, we compute that
	\begin{align*}
	\epsilon=(2\ell+1)(4\ell+4m+5-2\tau)\kappa+(4m+2)(2\ell+2m+2-\tau)\lambda.
	\end{align*}
	Thus, by Theorem \ref{th:cancel-1}~(ii),
	\begin{align*}
	H_M\big(\sH(q)\big) = (-1)^{\kappa}H_M\big(\hat{\sH}(q)\big).
	\end{align*}
	It follows from $\kappa\in\{1\}$ that
	\begin{align*}
	H_M\big(q^{\sigma}\sA_0\sB_{I}\big)+H_M\big(q^{\sigma}\sA_0\sB_{I\!I}\big)&=H_M\big(\sH(q)\big)+H_M\big(\hat{\sH}(q)\big)\\
	&=0.
	\end{align*}
	Thus, \eqref{eq:A0B-Eq4} is established.
	
	\item
	\textit{Examining \eqref{eq:AsB}}. Let us keep in mind that we have made the substitutions in \eqref{eq:sAI}, \eqref{eq:sAII}, \eqref{eq:sBI} and \eqref{eq:sBII}: $(\ell,m)\mapsto (2\ell+1,4m+2)$ and $a=2$. Also, $\sigma=-(2\ell+2m+2)k$. Then \eqref{eq:AsB} becomes
	\begin{align}\label{eq:AsB-Eq4}
	H_M\big(q^{\sigma}\sA_{I}\sB_{I}\big) =-(-1)^{\kappa}H_M\big(q^{\sigma}\sA_{I\!I}\sB_{I\!I}\big).
	\end{align}
	In Theorem \ref{th:cancel-2}, we further set
	\begin{align*}
	\begin{array}{rclp{4cm}rcl}
	A' & \mapsto & \xi \mu && B' & \mapsto & (4m+2+\tau)\mu\\
	w & \mapsto & \multicolumn{5}{l}{-(2\ell+2m+2)k+2\xi k+\tau k}\\
	\end{array}
	\end{align*}
	Then
	\begin{align*}
	\chw = -(2\ell+2m+2)k +2(2\ell+1-\xi )k +(4m+2-\tau) k.
	\end{align*}
	So,
	\begin{align*}
	\sH(q)&=q^{\sigma}\sA_{I}\sB_{I},\\
	\check{\sH}(q)&=q^{\sigma}\sA_{I\!I}\sB_{I\!I}.
	\end{align*}
	Now we compute that
	\begin{align*}
	\frac{2(A'Bu+AB'v)}{d^*AB} = \frac{2(4m+2+2\xi +\tau )k}{d^*}.
	\end{align*}
	If $d^* \nmid 2(4m+2+2\xi +\tau )k$, then $\delta_0 \nmid (4m+2+2\xi +\tau )k$, and thus,
	\begin{align*}
	w&\equiv (4m+2+2\xi +\tau )k \not\equiv 0 \pmod{\delta_0},\\
	\chw&\equiv -(4m+2+2\xi +\tau )k \not\equiv 0 \pmod{\delta_0},
	\end{align*}
	implying from Theorem \ref{th:UMH}~(i) that
	\begin{align*}
	H_M\big(q^{\sigma}\sA_{I}\sB_{I}\big) =-(-1)^{\kappa}H_M\big(q^{\sigma}\sA_{I\!I}\sB_{I\!I}\big)=0.
	\end{align*}
	Below, we assume that $d^* \mid 2(4m+2+2\xi +\tau )k$. Then the first assumption in \eqref{eq:assump-2-d*} is satisfied. We may also assume that $w$ is a multiple of $d^*=\delta_0$ according to Theorem \ref{th:cancel-2}~(i). With the choice of $K=M+1$, we compute that
	\begin{align*}
	\frac{B(A-2A')u^2+A(B-2B')uv+2ABuw\cdot K}{A(Av^2+Bu^2)}&=-4\ell-4m-4+4\xi+2\tau,\\
	\frac{B(A-2A')uv+A(B-2B')v^2+2ABvw\cdot K}{B(Av^2+Bu^2)}&=-2\ell-2m-2+2\xi+\tau.
	\end{align*}
	So the remaining assumptions in \eqref{eq:assump-2-d*} are satisfied. Finally, we compute that
	\begin{align*}
	\varepsilon=2(2\ell+1)(2\ell+2m+2-2\xi-\tau)\kappa+2(2m+1)(2\ell+2m+2-2\xi-\tau)\lambda.
	\end{align*}
	Thus, by Theorem \ref{th:cancel-2}~(iii),
	\begin{align*}
	H_M\big(\sH(q)\big) = H_M\big(\check{\sH}(q)\big).
	\end{align*}
	It follows from $\kappa\in\{1\}$ that
	\begin{align*}
	H_M\big(q^{\sigma}\sA_{I}\sB_{I}\big)+(-1)^{\kappa}H_M\big(q^{\sigma}\sA_{I\!I}\sB_{I\!I}\big)&=H_M\big(\sH(q)\big)+(-1)^{\kappa}H_M\big(\check{\sH}(q)\big)\\
	&=0.
	\end{align*}
	Thus, \eqref{eq:AsB-Eq4} is established.
	
\end{itemize}

\subsection{Type I --- Theorem \ref{th:Type-I.2}}

This section deals with the coefficient-vanishing results in Theorem \ref{th:Type-I.2}. Our proof still relies on the pairing-and-cancelation process but we also need to show that one particular term reduces to zero after applying the $\mathbf{H}$-operator.

We first deduce from Corollary \ref{coro:power-pairing} with certain common factors extracted and powers of $(-1)$ modified that, for any $\ell\ge 1$,
\begin{align*}
	f\big((-1)^{\kappa}q^{ak},(-1)^{\kappa}q^{-ak+M\mu}\big)^{\ell}=\sum_{s=0}^{\ell-1}\sA_s \sF_{s},
\end{align*}
where each $\sF_{\star}$ is a series in $q^M$, and
\begin{align}\label{eq:sA0-Type-I.2}
	\sA_0 = f\big((-1)^{\kappa \ell}q^{ak\ell},(-1)^{\kappa \ell}q^{-ak\ell+M\mu \ell}\big),
\end{align}
and for $1\le s\le \ell-1$,
\begin{align*}
	\sA_s&=q^{aks} f\big((-1)^{\kappa \ell}q^{ak\ell+M\mu s},(-1)^{\kappa \ell}q^{-ak\ell+M\mu (\ell-s)}\big)\notag\\
	&\quad+(-1)^{\kappa \ell}q^{ak(\ell-s)} f\big((-1)^{\kappa \ell}q^{ak\ell+M\mu (\ell-s)},(-1)^{\kappa \ell}q^{-ak\ell+M\mu s}\big)\notag\\
	&=:\sA_{s,I}+(-1)^{\kappa \ell}\sA_{s,I\!I}.
\end{align*}
More generally, we consider
\begin{align}
	\sA_{I}&:=q^{a\xi k} f\big((-1)^{\kappa \ell}q^{ak\ell+M\mu \xi},(-1)^{\kappa \ell}q^{-ak\ell+M\mu (\ell-\xi)}\big),\label{eq:sAI-Type-I.2}\\
	\sA_{I\!I}&:=q^{a(\ell-\xi) k} f\big((-1)^{\kappa \ell}q^{ak\ell+M\mu (\ell-\xi)},(-1)^{\kappa \ell}q^{-ak\ell+M\mu \xi}\big),\label{eq:sAII-Type-I.2}
\end{align}
for generic $\xi\in\mathbb{Z}$.

Also, by Corollary \ref{coro:power-pairing-A'-A'}, we write for any $m\ge 1$,
\begin{align*}
	f\big((-1)^{\lambda}q^{bk+M\mu},(-1)^{\lambda}q^{-bk+M\mu}\big)^{m}= \sB_0\sG_0+\sum_{t=1}^{m-1}\sB_{t,1} \sG_{t,1}+\sum_{t=1}^{m-1}\sB_{t,2} \sG_{t,2},
\end{align*}
where each $\sG_{\star}$ is a series in $q^M$, and
\begin{align}\label{eq:sB0-Type-I.2}
	\sB_{0}=f\big((-1)^{\lambda m}q^{bkm+M\mu m},(-1)^{\lambda m}q^{-bkm+M\mu m}\big),
\end{align}
and for $1\le t\le m-1$,
\begin{align*}
	\sB_{t,1}&=q^{bkt} f\big((-1)^{\lambda m}q^{bkm+M\mu(m+2 t)},(-1)^{\lambda m}q^{-bkm+M\mu(m-2 t)}\big)\notag\\
	&\quad+q^{-bkt} f\big((-1)^{\lambda m}q^{bkm+M\mu(m-2 t)},(-1)^{\lambda m}q^{-bkm+M\mu(m+2 t)}\big)\\
	&=: \sB_{t,1,I}+\sB_{t,1,I\!I},\notag\\
	\sB_{t,2}&=q^{bk(m-t)} f\big((-1)^{\lambda m}q^{bkm+M\mu(3m-2 t)},(-1)^{\lambda m}q^{-bkm+M\mu(-m+2 t)}\big)\notag\\
	&\quad+q^{-bk(m-t)} f\big((-1)^{\lambda m}q^{bkm+M\mu(-m+2 t)},(-1)^{\lambda m}q^{-bkm+M\mu(3m-2 t)}\big)\\
	&=: \sB_{t,2,I}+\sB_{t,2,I\!I}.
\end{align*}
We point out that these $\sB_{\star,I}$ and $\sB_{\star,I\!I}$ are special cases of
\begin{align}
	\sB_{I}&:=q^{b\tau k} f\big((-1)^{\lambda m}q^{bkm+M\mu (m+2\tau)},(-1)^{\lambda m}q^{-bkm+M\mu (m-2\tau)}\big),\label{eq:sBI-Type-I.2}\\
	\sB_{I\!I}&:=q^{-b\tau k} f\big((-1)^{\lambda m}q^{bkm+M\mu (m-2\tau)},(-1)^{\lambda m}q^{-bkm+M\mu (m+2\tau)}\big),\label{eq:sBII-Type-I.2}
\end{align}
for generic $\tau\in\mathbb{Z}$.

With all parameters chosen in each case below, our generic target is to show the following by Corollary \ref{coro:UMH=0-J} and Theorems \ref{th:cancel-1} and \ref{th:cancel-2}:
\begin{align}\label{eq:A0B0-Type-I.2}
	\mathbf{H}_M\big(q^{\sigma}\sA_0\sB_{0}\big) =0,
\end{align}
\begin{align}\label{eq:A0B-Type-I.2}
	\mathbf{H}_M\big(q^{\sigma}\sA_0\sB_{I}\big) =-\mathbf{H}_M\big(q^{\sigma}\sA_0\sB_{I\!I}\big),
\end{align}
\begin{align}\label{eq:AB0-Type-I.2}
	\mathbf{H}_M\big(q^{\sigma}\sB_{0}\sA_{I}\big) =-(-1)^{\kappa \ell}\mathbf{H}_M\big(q^{\sigma}\sB_{0}\sA_{I\!I}\big),
\end{align}
and
\begin{align}\label{eq:AB-Type-I.2}
	\mathbf{H}_M\big(q^{\sigma}\sA_{I}\sB_{I}\big) =-(-1)^{\kappa \ell}\mathbf{H}_M\big(q^{\sigma}\sA_{I\!I}\sB_{I\!I}\big).
\end{align}
Upon setting $\tau\mapsto -\tau$ in \eqref{eq:AB-Type-I.2}, we automatically have
\begin{align}\label{eq:AB-2-Type-I.2}
	\mathbf{H}_M\big(q^{\sigma}\sA_{I}\sB_{I\!I}\big) =-(-1)^{\kappa \ell}\mathbf{H}_M\big(q^{\sigma}\sA_{I\!I}\sB_{I}\big).
\end{align}
Once the above relations are established, it is safe to conclude by our pairing that
\begin{align}
	\mathbf{H}_{M}\Big(q^{\sigma}f\big((-1)^{\kappa}q^{ak},(-1)^{\kappa}q^{-ak+ M\mu}\big)^{\ell}f\big((-1)^{\lambda}q^{bk+M\mu},(-1)^{\lambda}q^{-bk+M\mu}\big)^{m}\Big)=0.
\end{align}

\subsubsection{Equation \eqref{eq:I.2}}

Recall that $(\ell,m)\mapsto (2\ell+1,2m+1)$, $a=1$ and $b=1$. Also, $M=4\ell+2m+3$ and $\sigma=-(3\ell+m+2)k$. Further, $(\kappa,\lambda)\in\{(0,1),(1,0)\}$ and $k$ is such that $\gcd(k,M)=1$.

In Corollary \ref{coro:UMH=0-J} and Theorems \ref{th:cancel-1} and \ref{th:cancel-2}, we set
\begin{align*}
	\begin{array}{rclp{3cm}rcl}
		\kappa & \mapsto & (2\ell+1)\kappa && \lambda & \mapsto & (2m+1)\lambda\\
		u & \mapsto & (2\ell+1)k && v & \mapsto & (2m+1)k\\
		A & \mapsto & (2\ell+1)\mu && B & \mapsto & 2(2m+1)\mu\\
		M & \mapsto & \multicolumn{5}{l}{4\ell+2m+3}\\
	\end{array}
\end{align*}
Let
\begin{align*}
	d_0=\gcd\big(2\ell+1,2m+1\big).
\end{align*}
Then
\begin{align*}
	d=\gcd(u,v)=d_0 k.
\end{align*}
Noticing that $M=2(2\ell+1)+(2m+1)$, and that $k$ is coprime to $M$, we have
\begin{align*}
	d^*&=\gcd(u,v,M)=d_0,\\
	d_u&=\gcd(u,M)=d_0,\\
	d_v&=\gcd(v,M)=d_0.
\end{align*}
We compute that
\begin{align*}
	\frac{Av\cdot dM}{d_u(Av^2+Bu^2)}&=1,\\
	\frac{Bu\cdot dM}{d_v(Av^2+Bu^2)}&=2.
\end{align*}
Thus, the second and third assumptions in \eqref{eq:div-assump-coro}, \eqref{eq:assump-d*} and \eqref{eq:assump-2-d*} are satisfied.

\begin{itemize}[leftmargin=*,align=left,itemsep=5pt]
	\renewcommand{\labelitemi}{\scriptsize$\blacktriangleright$}
	
	\item 
	\textit{Examining \eqref{eq:A0B0-Type-I.2}}. Let us keep in mind that we have made the substitutions in \eqref{eq:sA0-Type-I.2} and \eqref{eq:sB0-Type-I.2}: $(\ell,m)\mapsto (2\ell+1,2m+1)$, $a=1$ and $b=1$. Also, $\sigma=-(3\ell+m+2)k$. We want to show \eqref{eq:A0B0-Type-I.2},
	\begin{align}\label{eq:A0B0-Type-I.2-Eq1}
		\mathbf{H}_M\big(q^{\sigma}\sA_0\sB_{0}\big) =0.
	\end{align}
	In Corollary \ref{coro:UMH=0-J}, we further set
	\begin{align*}
		\begin{array}{rclp{3cm}rcl}
			A' & \mapsto & 0 && B' & \mapsto & (2m+1)\mu\\
			w & \mapsto & \multicolumn{5}{l}{-(3\ell+m+2)k}\\
		\end{array}
	\end{align*}
	So,
	\begin{align*}
		\sH(q)=q^{\sigma}\sA_0\sB_0.
	\end{align*}
	We may further assume that $w$ is a multiple of $d^*$. Otherwise, we know from Theorem \ref{th:UMH}~(i) that
	\begin{align*}
		\mathbf{H}_M\big(\sH(q)\big) = 0,
	\end{align*}
	which gives
	\begin{align*}
		\mathbf{H}_M\big(q^{\sigma}\sA_0\sB_0\big) =0.
	\end{align*} 
	We further compute that
	\begin{align*}
		\frac{-v\kappa+u\lambda}{d}=-\frac{(2\ell+1)(2m+1)}{d_0}(\kappa-\lambda).
	\end{align*}
	The fact that $d_0=\gcd(2\ell+1,2m+1)$ then implies that $(2\ell+1)(2m+1)/d_0$ is an odd integer. Thus, \eqref{eq:div-assump-coro} is satisfied since $(\kappa,\lambda)\in\{(0,1),(1,0)\}$. Finally, we choose $J=0$ in \eqref{eq:div-assump-coro-J}. Then
	\begin{equation*}
		\scalebox{0.85}{%
			$
			\begin{aligned}
				\dfrac{2dMAv\cdot J-2dAvw-dAuv+dBu^2+2dA'uv-2dB'u^2-Auv^2-Bu^3}{2d(Av^2+Bu^2)}&=\frac{d_0-(2\ell+1)}{2d_0},\\
				\dfrac{2dMBu\cdot J-2dBuw+dAv^2-dBuv-2dA'v^2+2dB'uv+Av^3+Bu^2v}{2d(Av^2+Bu^2)}&=\frac{3d_0+(2m+1)}{2d_0}.
			\end{aligned}
			$}
	\end{equation*}
	Recall that $2\ell+1\equiv 2m+1\equiv d_0 \pmod{2d_0}$ since $d_0=\gcd(2\ell+1,2m+1)$ is odd. So \eqref{eq:div-assump-coro-J} is also satisfied. We conclude by Corollary \ref{coro:UMH=0-J} that
	\begin{align*}
		\mathbf{H}_M\big(\sH(q)\big) = 0.
	\end{align*}
	Thus, \eqref{eq:A0B0-Type-I.2-Eq1} is established.
	
	\item
	\textit{Examining \eqref{eq:A0B-Type-I.2}}. Let us keep in mind that we have made the substitutions in \eqref{eq:sA0-Type-I.2}, \eqref{eq:sBI-Type-I.2} and \eqref{eq:sBII-Type-I.2}: $(\ell,m)\mapsto (2\ell+1,2m+1)$, $a=1$ and $b=1$. Also, $\sigma=-(3\ell+m+2)k$. We want to show \eqref{eq:A0B-Type-I.2},
	\begin{align}\label{eq:A0B-Type-I.2-Eq2}
		\mathbf{H}_M\big(q^{\sigma}\sA_0\sB_{I}\big) =-\mathbf{H}_M\big(q^{\sigma}\sA_0\sB_{I\!I}\big).
	\end{align}
	In Theorem \ref{th:cancel-1}, we further set
	\begin{align*}
		\begin{array}{rclp{3.5cm}rcl}
			A' & \mapsto & 0 && B' & \mapsto & (2m+1+2\tau)\mu\\
			w & \mapsto & \multicolumn{5}{l}{-(3\ell+m+2)k+\tau k}\\
		\end{array}
	\end{align*}
	Then
	\begin{align*}
		\hw = -(3\ell+m+2)k -\tau k.
	\end{align*}
	So,
	\begin{align*}
		\sH(q)&=q^{\sigma}\sA_0\sB_{I},\\
		\hat{\sH}(q)&=q^{\sigma}\sA_0\sB_{I\!I}.
	\end{align*}
	Now we compute that
	\begin{align*}
		\frac{2B'v}{d^* B} = \frac{(2m+1+2\tau )k}{d^*}.
	\end{align*}
	If $d^* \nmid (2m+1+2\tau )k$, then $\tau k \not\equiv 0 \pmod{d_0}$, and thus,
	\begin{align*}
		w&=-\tfrac{3}{2} (2\ell+1)k-\tfrac{1}{2}(2m+1)k+\tau k\equiv \tau k \not\equiv 0 \pmod{d_0},\\
		\hw&=-\tfrac{3}{2} (2\ell+1)k-\tfrac{1}{2}(2m+1)k-\tau k\equiv -\tau k \not\equiv 0 \pmod{d_0}.
	\end{align*}
	So $w$ and $\hw$ are nonmultiples of $d_0=d^*$. In this case we know from Theorem \ref{th:UMH}~(i) that
	\begin{align*}
		\mathbf{H}_M\big(\sH(q)\big) = \mathbf{H}_M\big(\hat{\sH}(q)\big) = 0,
	\end{align*}
	which gives
	\begin{align*}
		\mathbf{H}_M\big(q^{\sigma}\sA_0\sB_{I}\big) =-\mathbf{H}_M\big(q^{\sigma}\sA_0\sB_{I\!I}\big)=0.
	\end{align*}
	Below, we assume that $d^* \mid (2m+1+2\tau )k$. Then the first assumption in \eqref{eq:assump-d*} is satisfied. We may also assume that $w$ is a multiple of $d^*=d_0$ according to Theorem \ref{th:cancel-1}~(i). It is easily seen that $d^*k$ is a divisor of $\gcd(u,v,w)$. Thus, we solve the following stronger system
	\begin{align*}
		\left\{
		\begin{aligned}
			K&\equiv 1 \pmod{M/d^*},\\
			K&\equiv 0 \pmod{d/(d^*k)},
		\end{aligned}
		\right.
	\end{align*}
	and choose $K=1$. Finally, we compute that
	\begin{align*}
		\frac{(A-2A')v^2-(B-2B')uv-2Buw\cdot K}{Av^2+Bu^2}&=3,\\
		\frac{B(A-2A')uv+A(B-2B')v^2+2ABvw\cdot K}{B(Av^2+Bu^2)}&=-1.
	\end{align*}
	So the remaining assumptions in \eqref{eq:assump-d*} are satisfied. Finally, we compute that
	\begin{align*}
		\epsilon=(6\ell+3)\kappa+(2m+1)\lambda.
	\end{align*}
	Thus, by Theorem \ref{th:cancel-1}~(ii),
	\begin{align*}
		\mathbf{H}_M\big(\sH(q)\big) = (-1)^{\kappa+\lambda}\mathbf{H}_M\big(\hat{\sH}(q)\big).
	\end{align*}
	It follows from $(\kappa,\lambda)\in\{(0,1),(1,0)\}$ that
	\begin{align*}
		\mathbf{H}_M\big(q^{\sigma}\sA_0\sB_{I}\big)+\mathbf{H}_M\big(q^{\sigma}\sA_0\sB_{I\!I}\big)&=\mathbf{H}_M\big(\sH(q)\big)+\mathbf{H}_M\big(\hat{\sH}(q)\big)\\
		&=0.
	\end{align*}
	Thus, \eqref{eq:A0B-Type-I.2-Eq2} is established.
	
	\item 
	\textit{Examining \eqref{eq:AB0-Type-I.2}}.
	Let us keep in mind that we have made the substitutions in \eqref{eq:sAI-Type-I.2}, \eqref{eq:sAII-Type-I.2} and \eqref{eq:sB0-Type-I.2}: $(\ell,m)\mapsto (2\ell+1,2m+1)$, $a=1$ and $b=1$. Also, $\sigma=-(3\ell+m+2)k$. Then \eqref{eq:AB0-Type-I.2} becomes
	\begin{align}\label{eq:AB0-Type-I.2-Eq2}
		\mathbf{H}_M\big(q^{\sigma}\sB_0\sA_{I}\big) =-(-1)^{\kappa}\mathbf{H}_M\big(q^{\sigma}\sB_0\sA_{I\!I}\big).
	\end{align}
	To make use of Theorem \ref{th:cancel-1}, we need to swap the choice of $(\kappa,u,A)$ and $(\lambda,v,B)$ in our initial setting. In other words, in Theorem \ref{th:cancel-1}, we set
	\begin{align*}
		\begin{array}{rclp{3cm}rcl}
			\kappa & \mapsto & (2m+1)\lambda && \lambda & \mapsto & (2\ell+1)\kappa\\
			u & \mapsto & (2m+1)k && v & \mapsto & (2\ell+1)k\\
			A & \mapsto & 2(2m+1)\mu && B & \mapsto & (2\ell+1)\mu\\
			A' & \mapsto & (2m+1)\mu && B' & \mapsto & \xi \mu\\
			M & \mapsto & \multicolumn{5}{l}{4\ell+2m+3}\\
			w & \mapsto & \multicolumn{5}{l}{-(3\ell+m+2)k+\xi k}\\
		\end{array}
	\end{align*}
	Then
	\begin{align*}
		\hw = -(3\ell+m+2)k +(2\ell+1-\xi) k.
	\end{align*}
	So,
	\begin{align*}
		\sH(q)&=q^{\sigma}\sB_0\sA_{I},\\
		\hat{\sH}(q)&=q^{\sigma}\sB_0\sA_{I\!I}.
	\end{align*}
	Now we compute that
	\begin{align*}
		\frac{2B'v}{d^* B} = \frac{2\xi k}{d^*}.
	\end{align*}
	If $d^* \nmid 2\xi k$, then $\xi k \not\equiv 0 \pmod{d_0}$, and thus,
	\begin{align*}
		w&\equiv \xi k \not\equiv 0 \pmod{d_0},\\
		\hw&\equiv -\xi k \not\equiv 0 \pmod{d_0},
	\end{align*}
	implying from Theorem \ref{th:UMH}~(i) that
	\begin{align*}
		\mathbf{H}_M\big(q^{\sigma}\sB_0\sA_{I}\big) =-(-1)^{\kappa}\mathbf{H}_M\big(q^{\sigma}\sB_0\sA_{I\!I}\big)=0.
	\end{align*}
	Below, we assume that $d^* \mid 2\xi k$. Then the first assumption in \eqref{eq:assump-d*} is satisfied. We may also assume that $w$ is a multiple of $d^*=d_0$ according to Theorem \ref{th:cancel-1}~(i). With the choice of $K=1$, we compute that
	\begin{align*}
		\frac{(A-2A')v^2-(B-2B')uv-2Buw\cdot K}{Av^2+Bu^2}&=1,\\
		\frac{B(A-2A')uv+A(B-2B')v^2+2ABvw\cdot K}{B(Av^2+Bu^2)}&=-2.
	\end{align*}
	So the remaining assumptions in \eqref{eq:assump-d*} are satisfied. Finally, we compute that
	\begin{align*}
		\epsilon=(4\ell+2)\kappa+(2m+1)\lambda.
	\end{align*}
	Thus, by Theorem \ref{th:cancel-1}~(ii),
	\begin{align*}
		\mathbf{H}_M\big(\sH(q)\big) = (-1)^{\lambda}\mathbf{H}_M\big(\hat{\sH}(q)\big).
	\end{align*}
	It follows from $(\kappa,\lambda)\in\{(0,1),(1,0)\}$ that
	\begin{align*}
		\mathbf{H}_M\big(q^{\sigma}\sB_0\sA_{I}\big)+(-1)^{\kappa}\mathbf{H}_M\big(q^{\sigma}\sB_0\sA_{I\!I}\big)&=\mathbf{H}_M\big(\sH(q)\big)+(-1)^{\kappa}\mathbf{H}_M\big(\hat{\sH}(q)\big)\\
		&=0.
	\end{align*}
	Thus, \eqref{eq:AB0-Type-I.2-Eq2} is established.
	
	\item
	\textit{Examining \eqref{eq:AB-Type-I.2}}. Let us keep in mind that we have made the substitutions in \eqref{eq:sAI-Type-I.2}, \eqref{eq:sAII-Type-I.2}, \eqref{eq:sBI-Type-I.2} and \eqref{eq:sBII-Type-I.2}: $(\ell,m)\mapsto (2\ell+1,2m+1)$, $a=1$ and $b=1$. Also, $\sigma=-(3\ell+m+2)k$. Then \eqref{eq:AB-Type-I.2} becomes
	\begin{align}\label{eq:AB-Type-I.2-Eq2}
		\mathbf{H}_M\big(q^{\sigma}\sA_{I}\sB_{I}\big) =-(-1)^{\kappa}\mathbf{H}_M\big(q^{\sigma}\sA_{I\!I}\sB_{I\!I}\big).
	\end{align}
	In Theorem \ref{th:cancel-2}, we further set
	\begin{align*}
		\begin{array}{rclp{4.5cm}rcl}
			A' & \mapsto & \xi \mu && B' & \mapsto & (2m+1+2\tau)\mu\\
			w & \mapsto & \multicolumn{5}{l}{-(3\ell+m+2)k+\xi k+\tau k}\\
		\end{array}
	\end{align*}
	Then
	\begin{align*}
		\chw = -(3\ell+m+2)k +(2\ell+1-\xi) k-\tau k.
	\end{align*}
	So,
	\begin{align*}
		\sH(q)&=q^{\sigma}\sA_{I}\sB_{I},\\
		\check{\sH}(q)&=q^{\sigma}\sA_{I\!I}\sB_{I\!I}.
	\end{align*}
	Now we compute that
	\begin{align*}
		\frac{2(A'Bu+AB'v)}{d^*AB} = \frac{(2m+1+2\xi +2\tau )k}{d^*}.
	\end{align*}
	Similarly, if $d^* \nmid (2m+1+2\xi +2\tau )k$, then $(\xi+\tau) k \not\equiv 0 \pmod{d_0}$, and thus,
	\begin{align*}
		w&\equiv (\xi+\tau) k \not\equiv 0 \pmod{d_0},\\
		\chw&\equiv -(\xi+\tau) k \not\equiv 0 \pmod{d_0},
	\end{align*}
	implying from Theorem \ref{th:UMH}~(i) that
	\begin{align*}
		\mathbf{H}_M\big(q^{\sigma}\sA_{I}\sB_{I}\big) =-(-1)^{\kappa}\mathbf{H}_M\big(q^{\sigma}\sA_{I\!I}\sB_{I\!I}\big)=0.
	\end{align*}
	Below, we assume that $d^* \mid (2m+1+2\xi +2\tau )k$. Then the first assumption in \eqref{eq:assump-2-d*} is satisfied. We may also assume that $w$ is a multiple of $d^*=d_0$ according to Theorem \ref{th:cancel-2}~(i). With the choice of $K=1$, we compute that
	\begin{align*}
		\frac{B(A-2A')u^2+A(B-2B')uv+2ABuw\cdot K}{A(Av^2+Bu^2)}&=-2,\\
		\frac{B(A-2A')uv+A(B-2B')v^2+2ABvw\cdot K}{B(Av^2+Bu^2)}&=-1.
	\end{align*}
	So the remaining assumptions in \eqref{eq:assump-2-d*} are satisfied. Finally, we compute that
	\begin{align*}
		\varepsilon=(4\ell+2)\kappa+(2m+1)\lambda.
	\end{align*}
	Thus, by Theorem \ref{th:cancel-2}~(ii),
	\begin{align*}
		\mathbf{H}_M\big(\sH(q)\big) = (-1)^{\lambda}\mathbf{H}_M\big(\check{\sH}(q)\big).
	\end{align*}
	It follows from $(\kappa,\lambda)\in\{(0,1),(1,0)\}$ that
	\begin{align*}
		\mathbf{H}_M\big(q^{\sigma}\sA_{I}\sB_{I}\big)+(-1)^{\kappa}\mathbf{H}_M\big(q^{\sigma}\sA_{I\!I}\sB_{I\!I}\big)&=\mathbf{H}_M\big(\sH(q)\big)+(-1)^{\kappa}\mathbf{H}_M\big(\check{\sH}(q)\big)\\
		&=0.
	\end{align*}
	Thus, \eqref{eq:AB-Type-I.2-Eq2} is established.
	
\end{itemize}

\subsubsection{Equation \eqref{eq:I.2-2}}

Recall that $(\ell,m)\mapsto (2\ell+1,2m+1)$, $a=1$ and $b=2$. Also, $M=2\ell+4m+3$ and $\sigma=-(2\ell+2m+2)k$. Further, $(\kappa,\lambda)\in\{(0,1),(1,1)\}$ and $k$ is such that $\gcd(k,M)=1$.

In Corollary \ref{coro:UMH=0-J} and Theorems \ref{th:cancel-1} and \ref{th:cancel-2}, we set
\begin{align*}
\begin{array}{rclp{3cm}rcl}
\kappa & \mapsto & (2\ell+1)\kappa && \lambda & \mapsto & (2m+1)\lambda\\
u & \mapsto & (2\ell+1)k && v & \mapsto & 2(2m+1)k\\
A & \mapsto & (2\ell+1)\mu && B & \mapsto & 2(2m+1)\mu\\
M & \mapsto & \multicolumn{5}{l}{2\ell+4m+3}\\
\end{array}
\end{align*}
Let
\begin{align*}
d_0=\gcd\big(2\ell+1,2(2m+1)\big)=\gcd\big(2\ell+1,2m+1\big).
\end{align*}
Then
\begin{align*}
d=\gcd(u,v)=d_0 k.
\end{align*}
Noticing that $M=(2\ell+1)+2(2m+1)$, and that $k$ is coprime to $M$, we have
\begin{align*}
d^*&=\gcd(u,v,M)=d_0,\\
d_u&=\gcd(u,M)=d_0,\\
d_v&=\gcd(v,M)=d_0.
\end{align*}
We compute that
\begin{align*}
\frac{Av\cdot dM}{d_u(Av^2+Bu^2)}&=1,\\
\frac{Bu\cdot dM}{d_v(Av^2+Bu^2)}&=1.
\end{align*}
Thus, the second and third assumptions in \eqref{eq:div-assump-coro}, \eqref{eq:assump-d*} and \eqref{eq:assump-2-d*} are satisfied.

\begin{itemize}[leftmargin=*,align=left,itemsep=5pt]
	\renewcommand{\labelitemi}{\scriptsize$\blacktriangleright$}
	
	\item 
	\textit{Examining \eqref{eq:A0B0-Type-I.2}}. Let us keep in mind that we have made the substitutions in \eqref{eq:sA0-Type-I.2} and \eqref{eq:sB0-Type-I.2}: $(\ell,m)\mapsto (2\ell+1,2m+1)$, $a=1$ and $b=2$. Also, $\sigma=-(2\ell+2m+2)k$. We want to show \eqref{eq:A0B0-Type-I.2},
	\begin{align}\label{eq:A0B0-Type-I.2-2-Eq1}
	H_M\big(q^{\sigma}\sA_0\sB_{0}\big) =0.
	\end{align}
	In Corollary \ref{coro:UMH=0-J}, we further set
	\begin{align*}
	\begin{array}{rclp{3cm}rcl}
	A' & \mapsto & 0 && B' & \mapsto & (2m+1)\mu\\
	w & \mapsto & \multicolumn{5}{l}{-(2\ell+2m+2)k}\\
	\end{array}
	\end{align*}
	So,
	\begin{align*}
	\sH(q)=q^{\sigma}\sA_0\sB_0.
	\end{align*}
	We may further assume that $w$ is a multiple of $d^*$. Otherwise, we know from Theorem \ref{th:UMH}~(i) that
	\begin{align*}
	H_M\big(\sH(q)\big) = 0,
	\end{align*}
	which gives
	\begin{align*}
	H_M\big(q^{\sigma}\sA_0\sB_0\big) =0.
	\end{align*} 
	We further compute that
	\begin{align*}
	\frac{-v\kappa+u\lambda}{d}=-\frac{(2\ell+1)(2m+1)}{d_0}(2\kappa-\lambda).
	\end{align*}
	The fact that $d_0=\gcd(2\ell+1,2m+1)$ then implies that $(2\ell+1)(2m+1)/d_0$ is an odd integer. Thus, \eqref{eq:div-assump-coro} is satisfied since $(\kappa,\lambda)\in\{(0,1),(1,1)\}$. Finally, we choose $J=0$ in \eqref{eq:div-assump-coro-J}. Then
	\begin{equation*}
	\scalebox{0.85}{%
		$
		\begin{aligned}
		\dfrac{2dMAv\cdot J-2dAvw-dAuv+dBu^2+2dA'uv-2dB'u^2-Auv^2-Bu^3}{2d(Av^2+Bu^2)}&=\frac{d_0-(2\ell+1)}{2d_0},\\
		\dfrac{2dMBu\cdot J-2dBuw+dAv^2-dBuv-2dA'v^2+2dB'uv+Av^3+Bu^2v}{2d(Av^2+Bu^2)}&=\frac{d_0+(2m+1)}{d_0}.
		\end{aligned}
		$}
	\end{equation*}
	Recall that $2\ell+1\equiv 2m+1\equiv d_0 \pmod{2d_0}$ since $d_0=\gcd(2\ell+1,2m+1)$ is odd. So \eqref{eq:div-assump-coro-J} is also satisfied. We conclude by Corollary \ref{coro:UMH=0-J} that
	\begin{align*}
	H_M\big(\sH(q)\big) = 0.
	\end{align*}
	Thus, \eqref{eq:A0B0-Type-I.2-2-Eq1} is established.
	
	\item
	\textit{Examining \eqref{eq:A0B-Type-I.2}}. Let us keep in mind that we have made the substitutions in \eqref{eq:sA0-Type-I.2}, \eqref{eq:sBI-Type-I.2} and \eqref{eq:sBII-Type-I.2}: $(\ell,m)\mapsto (2\ell+1,2m+1)$, $a=1$ and $b=2$. Also, $\sigma=-(2\ell+2m+2)k$. We want to show \eqref{eq:A0B-Type-I.2},
	\begin{align}\label{eq:A0B-Type-I.2-2-Eq2}
	H_M\big(q^{\sigma}\sA_0\sB_{I}\big) =-H_M\big(q^{\sigma}\sA_0\sB_{I\!I}\big).
	\end{align}
	In Theorem \ref{th:cancel-1}, we further set
	\begin{align*}
	\begin{array}{rclp{3.5cm}rcl}
	A' & \mapsto & 0 && B' & \mapsto & (2m+1+2\tau)\mu\\
	w & \mapsto & \multicolumn{5}{l}{-(2\ell+2m+2)k+2\tau k}\\
	\end{array}
	\end{align*}
	Then
	\begin{align*}
	\hw = -(2\ell+2m+2)k -2\tau k.
	\end{align*}
	So,
	\begin{align*}
	\sH(q)&=q^{\sigma}\sA_0\sB_{I},\\
	\hat{\sH}(q)&=q^{\sigma}\sA_0\sB_{I\!I}.
	\end{align*}
	Now we compute that
	\begin{align*}
	\frac{2B'v}{d^* B} = \frac{2(2m+1+2\tau )k}{d^*}.
	\end{align*}
	If $d^* \nmid 2(2m+1+2\tau )k$, then $2\tau k \not\equiv 0 \pmod{d_0}$, and thus,
	\begin{align*}
	w&=- (2\ell+1)k-(2m+1)k+2\tau k\equiv 2\tau k \not\equiv 0 \pmod{d_0},\\
	\hw&=- (2\ell+1)k-(2m+1)k-2\tau k\equiv -2\tau k \not\equiv 0 \pmod{d_0}.
	\end{align*}
	So $w$ and $\hw$ are nonmultiples of $d_0=d^*$. In this case we know from Theorem \ref{th:UMH}~(i) that
	\begin{align*}
	H_M\big(\sH(q)\big) = H_M\big(\hat{\sH}(q)\big) = 0,
	\end{align*}
	which gives
	\begin{align*}
	H_M\big(q^{\sigma}\sA_0\sB_{I}\big) =-H_M\big(q^{\sigma}\sA_0\sB_{I\!I}\big)=0.
	\end{align*}
	Below, we assume that $d^* \mid 2(2m+1+2\tau )k$. Then the first assumption in \eqref{eq:assump-d*} is satisfied. We may also assume that $w$ is a multiple of $d^*=d_0$ according to Theorem \ref{th:cancel-1}~(i). It is easily seen that $d^*k$ is a divisor of $\gcd(u,v,w)$. Thus, we solve the following stronger system
	\begin{align*}
	\left\{
	\begin{aligned}
	K&\equiv 1 \pmod{M/d^*},\\
	K&\equiv 0 \pmod{d/(d^*k)},
	\end{aligned}
	\right.
	\end{align*}
	and choose $K=1$. Finally, we compute that
	\begin{align*}
	\frac{(A-2A')v^2-(B-2B')uv-2Buw\cdot K}{Av^2+Bu^2}&=2,\\
	\frac{B(A-2A')uv+A(B-2B')v^2+2ABvw\cdot K}{B(Av^2+Bu^2)}&=-1.
	\end{align*}
	So the remaining assumptions in \eqref{eq:assump-d*} are satisfied. Finally, we compute that
	\begin{align*}
	\epsilon=(4\ell+2)\kappa+(2m+1)\lambda.
	\end{align*}
	Thus, by Theorem \ref{th:cancel-1}~(ii),
	\begin{align*}
	H_M\big(\sH(q)\big) = (-1)^{\lambda}H_M\big(\hat{\sH}(q)\big).
	\end{align*}
	It follows from $(\kappa,\lambda)\in\{(0,1),(1,1)\}$ that
	\begin{align*}
	H_M\big(q^{\sigma}\sA_0\sB_{I}\big)+H_M\big(q^{\sigma}\sA_0\sB_{I\!I}\big)&=H_M\big(\sH(q)\big)+H_M\big(\hat{\sH}(q)\big)\\
	&=0.
	\end{align*}
	Thus, \eqref{eq:A0B-Type-I.2-2-Eq2} is established.
	
	\item 
	\textit{Examining \eqref{eq:AB0-Type-I.2}}.
	Let us keep in mind that we have made the substitutions in \eqref{eq:sAI-Type-I.2}, \eqref{eq:sAII-Type-I.2} and \eqref{eq:sB0-Type-I.2}: $(\ell,m)\mapsto (2\ell+1,2m+1)$, $a=1$ and $b=2$. Also, $\sigma=-(2\ell+2m+2)k$. Then \eqref{eq:AB0-Type-I.2} becomes
	\begin{align}\label{eq:AB0-Type-I.2-2-Eq2}
	H_M\big(q^{\sigma}\sB_0\sA_{I}\big) =-(-1)^{\kappa}H_M\big(q^{\sigma}\sB_0\sA_{I\!I}\big).
	\end{align}
	To make use of Theorem \ref{th:cancel-1}, we need to swap the choice of $(\kappa,u,A)$ and $(\lambda,v,B)$ in our initial setting. In other words, in Theorem \ref{th:cancel-1}, we set
	\begin{align*}
	\begin{array}{rclp{3cm}rcl}
	\kappa & \mapsto & (2m+1)\lambda && \lambda & \mapsto & (2\ell+1)\kappa\\
	u & \mapsto & 2(2m+1)k && v & \mapsto & (2\ell+1)k\\
	A & \mapsto & 2(2m+1)\mu && B & \mapsto & (2\ell+1)\mu\\
	A' & \mapsto & (2m+1)\mu && B' & \mapsto & \xi \mu\\
	M & \mapsto & \multicolumn{5}{l}{2\ell+4m+3}\\
	w & \mapsto & \multicolumn{5}{l}{-(2\ell+2m+2)k+\xi k}\\
	\end{array}
	\end{align*}
	Then
	\begin{align*}
	\hw = -(2\ell+2m+2)k +(2\ell+1-\xi) k.
	\end{align*}
	So,
	\begin{align*}
	\sH(q)&=q^{\sigma}\sB_0\sA_{I},\\
	\hat{\sH}(q)&=q^{\sigma}\sB_0\sA_{I\!I}.
	\end{align*}
	Now we compute that
	\begin{align*}
	\frac{2B'v}{d^* B} = \frac{2\xi k}{d^*}.
	\end{align*}
	If $d^* \nmid 2\xi k$, then $\xi k \not\equiv 0 \pmod{d_0}$, and thus,
	\begin{align*}
	w&\equiv \xi k \not\equiv 0 \pmod{d_0},\\
	\hw&\equiv -\xi k \not\equiv 0 \pmod{d_0},
	\end{align*}
	implying from Theorem \ref{th:UMH}~(i) that
	\begin{align*}
	H_M\big(q^{\sigma}\sB_0\sA_{I}\big) =-(-1)^{\kappa}H_M\big(q^{\sigma}\sB_0\sA_{I\!I}\big)=0.
	\end{align*}
	Below, we assume that $d^* \mid 2\xi k$. Then the first assumption in \eqref{eq:assump-d*} is satisfied. We may also assume that $w$ is a multiple of $d^*=d_0$ according to Theorem \ref{th:cancel-1}~(i). With the choice of $K=1$, we compute that
	\begin{align*}
	\frac{(A-2A')v^2-(B-2B')uv-2Buw\cdot K}{Av^2+Bu^2}&=1,\\
	\frac{B(A-2A')uv+A(B-2B')v^2+2ABvw\cdot K}{B(Av^2+Bu^2)}&=-1.
	\end{align*}
	So the remaining assumptions in \eqref{eq:assump-d*} are satisfied. Finally, we compute that
	\begin{align*}
	\epsilon=(2\ell+1)\kappa+(2m+1)\lambda.
	\end{align*}
	Thus, by Theorem \ref{th:cancel-1}~(ii),
	\begin{align*}
	H_M\big(\sH(q)\big) = (-1)^{\kappa+\lambda}H_M\big(\hat{\sH}(q)\big).
	\end{align*}
	It follows from $(\kappa,\lambda)\in\{(0,1),(1,1)\}$ that
	\begin{align*}
	H_M\big(q^{\sigma}\sB_0\sA_{I}\big)+(-1)^{\kappa}H_M\big(q^{\sigma}\sB_0\sA_{I\!I}\big)&=H_M\big(\sH(q)\big)+(-1)^{\kappa}H_M\big(\hat{\sH}(q)\big)\\
	&=0.
	\end{align*}
	Thus, \eqref{eq:AB0-Type-I.2-2-Eq2} is established.
	
	\item
	\textit{Examining \eqref{eq:AB-Type-I.2}}. Let us keep in mind that we have made the substitutions in \eqref{eq:sAI-Type-I.2}, \eqref{eq:sAII-Type-I.2}, \eqref{eq:sBI-Type-I.2} and \eqref{eq:sBII-Type-I.2}: $(\ell,m)\mapsto (2\ell+1,2m+1)$, $a=1$ and $b=2$. Also, $\sigma=-(2\ell+2m+2)k$. Then \eqref{eq:AB-Type-I.2} becomes
	\begin{align}\label{eq:AB-Type-I.2-2-Eq2}
	H_M\big(q^{\sigma}\sA_{I}\sB_{I}\big) =-(-1)^{\kappa}H_M\big(q^{\sigma}\sA_{I\!I}\sB_{I\!I}\big).
	\end{align}
	In Theorem \ref{th:cancel-2}, we further set
	\begin{align*}
	\begin{array}{rclp{4.5cm}rcl}
	A' & \mapsto & \xi \mu && B' & \mapsto & (2m+1+2\tau)\mu\\
	w & \mapsto & \multicolumn{5}{l}{-(2\ell+2m+2)k+\xi k+2\tau k}\\
	\end{array}
	\end{align*}
	Then
	\begin{align*}
	\chw = -(2\ell+2m+2)k +(2\ell+1-\xi) k-2\tau k.
	\end{align*}
	So,
	\begin{align*}
	\sH(q)&=q^{\sigma}\sA_{I}\sB_{I},\\
	\check{\sH}(q)&=q^{\sigma}\sA_{I\!I}\sB_{I\!I}.
	\end{align*}
	Now we compute that
	\begin{align*}
	\frac{2(A'Bu+AB'v)}{d^*AB} = \frac{2(2m+1+\xi +2\tau )k}{d^*}.
	\end{align*}
	Similarly, if $d^* \nmid 2(2m+1+\xi +2\tau )k$, then $(\xi+2\tau) k \not\equiv 0 \pmod{d_0}$, and thus,
	\begin{align*}
	w&\equiv (\xi+2\tau) k \not\equiv 0 \pmod{d_0},\\
	\chw&\equiv -(\xi+2\tau) k \not\equiv 0 \pmod{d_0},
	\end{align*}
	implying from Theorem \ref{th:UMH}~(i) that
	\begin{align*}
	H_M\big(q^{\sigma}\sA_{I}\sB_{I}\big) =-(-1)^{\kappa}H_M\big(q^{\sigma}\sA_{I\!I}\sB_{I\!I}\big)=0.
	\end{align*}
	Below, we assume that $d^* \mid 2(2m+1+\xi +2\tau )k$. Then the first assumption in \eqref{eq:assump-2-d*} is satisfied. We may also assume that $w$ is a multiple of $d^*=d_0$ according to Theorem \ref{th:cancel-2}~(i). With the choice of $K=1$, we compute that
	\begin{align*}
	\frac{B(A-2A')u^2+A(B-2B')uv+2ABuw\cdot K}{A(Av^2+Bu^2)}&=-1,\\
	\frac{B(A-2A')uv+A(B-2B')v^2+2ABvw\cdot K}{B(Av^2+Bu^2)}&=-1.
	\end{align*}
	So the remaining assumptions in \eqref{eq:assump-2-d*} are satisfied. Finally, we compute that
	\begin{align*}
	\varepsilon=(2\ell+1)\kappa+(2m+1)\lambda.
	\end{align*}
	Thus, by Theorem \ref{th:cancel-2}~(ii),
	\begin{align*}
	H_M\big(\sH(q)\big) = (-1)^{\kappa+\lambda}H_M\big(\check{\sH}(q)\big).
	\end{align*}
	It follows from $(\kappa,\lambda)\in\{(0,1),(1,0)\}$ that
	\begin{align*}
	H_M\big(q^{\sigma}\sA_{I}\sB_{I}\big)+(-1)^{\kappa}H_M\big(q^{\sigma}\sA_{I\!I}\sB_{I\!I}\big)&=H_M\big(\sH(q)\big)+(-1)^{\kappa}H_M\big(\check{\sH}(q)\big)\\
	&=0.
	\end{align*}
	Thus, \eqref{eq:AB-Type-I.2-2-Eq2} is established.
	
\end{itemize}

\subsection{Type I --- Theorem \ref{th:Type-I.3}}

In this section, we treat the coefficient-vanishing results in Theorem \ref{th:Type-I.3}. The basic idea is similar to that for Theorem \ref{th:Type-I.2}.

We still deduce from Corollary \ref{coro:power-pairing} with certain common factors extracted and powers of $(-1)$ modified that, for any $\ell\ge 1$,
\begin{align*}
	f\big((-1)^{\kappa}q^{ak},(-1)^{\kappa}q^{-ak+M\mu}\big)^{\ell}=\sum_{s=0}^{\ell-1}\sA_s \sF_{s},
\end{align*}
where each $\sF_{\star}$ is a series in $q^M$, and
\begin{align}\label{eq:sA0-Type-I.3}
	\sA_0 = f\big((-1)^{\kappa \ell}q^{ak\ell},(-1)^{\kappa \ell}q^{-ak\ell+M\mu \ell}\big),
\end{align}
and for $1\le s\le \ell-1$,
\begin{align*}
	\sA_s&=q^{aks} f\big((-1)^{\kappa \ell}q^{ak\ell+M\mu s},(-1)^{\kappa \ell}q^{-ak\ell+M\mu (\ell-s)}\big)\notag\\
	&\quad+(-1)^{\kappa \ell}q^{ak(\ell-s)} f\big((-1)^{\kappa \ell}q^{ak\ell+M\mu (\ell-s)},(-1)^{\kappa \ell}q^{-ak\ell+M\mu s}\big)\notag\\
	&=:\sA_{s,I}+(-1)^{\kappa \ell}\sA_{s,I\!I}.
\end{align*}
More generally, we consider
\begin{align}
	\sA_{I}&:=q^{a\xi k} f\big((-1)^{\kappa \ell}q^{ak\ell+M\mu \xi},(-1)^{\kappa \ell}q^{-ak\ell+M\mu (\ell-\xi)}\big),\label{eq:sAI-Type-I.3}\\
	\sA_{I\!I}&:=q^{a(\ell-\xi) k} f\big((-1)^{\kappa \ell}q^{ak\ell+M\mu (\ell-\xi)},(-1)^{\kappa \ell}q^{-ak\ell+M\mu \xi}\big),\label{eq:sAII-Type-I.3}
\end{align}
for generic $\xi\in\mathbb{Z}$.

Now in this occasion, by Corollary \ref{coro:power-pairing}, we write for any $m\ge 1$,
\begin{align*}
	f\big((-1)^{\lambda}q^{bk},(-1)^{\lambda}q^{-bk+2M\mu}\big)^{m}=\sum_{t=0}^{m-1}\sB_t \sG_{t},
\end{align*}
where each $\sG_{\star}$ is a series in $q^M$, and
\begin{align}\label{eq:sB0-Type-I.3}
	\sB_0 = f\big((-1)^{\lambda m}q^{bkm},(-1)^{\lambda m}q^{-bkm+2M\mu m}\big),
\end{align}
and for $1\le t\le m-1$,
\begin{align*}
	\sB_t&=q^{bkt} f\big((-1)^{\lambda m}q^{bkm+2M\mu t},(-1)^{\lambda m}q^{-bkm+2M\mu (m-t)}\big)\notag\\
	&\quad+(-1)^{\lambda m}q^{bk(m-t)} f\big((-1)^{\lambda m}q^{bkm+2M\mu (m-t)},(-1)^{\lambda m}q^{-bkm+2M\mu t}\big)\notag\\
	&=:\sB_{t,I}+(-1)^{\lambda m}\sB_{t,I\!I}.
\end{align*}
More generally, we consider
\begin{align}
	\sB_{I}&:=q^{b\tau k} f\big((-1)^{\lambda m}q^{bkm+2M\mu \tau},(-1)^{\lambda m}q^{-bkm+2M\mu (m-\tau)}\big),\label{eq:sBI-Type-I.3}\\
	\sB_{I\!I}&:=q^{b(m-\tau) k} f\big((-1)^{\lambda m}q^{bkm+2M\mu (m-\tau)},(-1)^{\lambda m}q^{-bkm+2M\mu \tau}\big),\label{eq:sBII-Type-I.3}
\end{align}
for generic $\tau\in\mathbb{Z}$.

With all parameters chosen in each case below, our generic target is to show the following:
\begin{align}\label{eq:A0B0-Type-I.3}
	\mathbf{H}_M\big(q^{\sigma}\sA_0\sB_{0}\big) =0,
\end{align}
\begin{align}\label{eq:A0B-Type-I.3}
	\mathbf{H}_M\big(q^{\sigma}\sA_0\sB_{I}\big) =-(-1)^{\lambda m}\mathbf{H}_M\big(q^{\sigma}\sA_0\sB_{I\!I}\big),
\end{align}
\begin{align}\label{eq:AB0-Type-I.3}
	\mathbf{H}_M\big(q^{\sigma}\sB_{0}\sA_{I}\big) =-(-1)^{\kappa \ell}\mathbf{H}_M\big(q^{\sigma}\sB_{0}\sA_{I\!I}\big),
\end{align}
and
\begin{align}\label{eq:AB-Type-I.3}
	\mathbf{H}_M\big(q^{\sigma}\sA_{I}\sB_{I}\big) =-(-1)^{\kappa \ell+\lambda m}\mathbf{H}_M\big(q^{\sigma}\sA_{I\!I}\sB_{I\!I}\big).
\end{align}
Upon setting $\tau\mapsto m-\tau$ in \eqref{eq:AB-Type-I.3}, we automatically have
\begin{align}\label{eq:AB-2-Type-I.3}
	\mathbf{H}_M\big(q^{\sigma}\sA_{I}\sB_{I\!I}\big) =-(-1)^{\kappa \ell+\lambda m}\mathbf{H}_M\big(q^{\sigma}\sA_{I\!I}\sB_{I}\big).
\end{align}
Once the above relations are established, it is safe to conclude by our pairing that
\begin{align}
	\mathbf{H}_{M}\Big(q^{\sigma}f\big((-1)^{\kappa}q^{ak},(-1)^{\kappa}q^{-ak+ M\mu}\big)^{\ell}f\big((-1)^{\lambda}q^{bk},(-1)^{\lambda}q^{-bk+2M\mu}\big)^{m}\Big)=0.
\end{align}

\subsubsection{Equation \eqref{eq:I.3-1}}

Recall that $(\ell,m)\mapsto (2\ell+1,2m+1)$, $a=1$ and $b=1$. Also, $M=4\ell+2m+3$ and $\sigma=-(\ell+m+1)k$. Further, $(\kappa,\lambda)\in\{(0,1),(1,0)\}$ and $k$ is such that $\gcd(k,M)=1$.

In Corollary \ref{coro:UMH=0-J} and Theorems \ref{th:cancel-1} and \ref{th:cancel-2}, we set
\begin{align*}
	\begin{array}{rclp{3cm}rcl}
		\kappa & \mapsto & (2\ell+1)\kappa && \lambda & \mapsto & (2m+1)\lambda\\
		u & \mapsto & (2\ell+1)k && v & \mapsto & (2m+1)k\\
		A & \mapsto & (2\ell+1)\mu && B & \mapsto & 2(2m+1)\mu\\
		M & \mapsto & \multicolumn{5}{l}{4\ell+2m+3}\\
	\end{array}
\end{align*}
Let
\begin{align*}
	d_0=\gcd\big(2\ell+1,2m+1\big).
\end{align*}
Then
\begin{align*}
	d=\gcd(u,v)=d_0 k.
\end{align*}
Noticing that $M=2(2\ell+1)+(2m+1)$, and that $k$ is coprime to $M$, we have
\begin{align*}
	d^*&=\gcd(u,v,M)=d_0,\\
	d_u&=\gcd(u,M)=d_0,\\
	d_v&=\gcd(v,M)=d_0.
\end{align*}
We compute that
\begin{align*}
	\frac{Av\cdot dM}{d_u(Av^2+Bu^2)}&=1,\\
	\frac{Bu\cdot dM}{d_v(Av^2+Bu^2)}&=2.
\end{align*}
Thus, the second and third assumptions in \eqref{eq:div-assump-coro}, \eqref{eq:assump-d*} and \eqref{eq:assump-2-d*} are satisfied.

\begin{itemize}[leftmargin=*,align=left,itemsep=5pt]
	\renewcommand{\labelitemi}{\scriptsize$\blacktriangleright$}
	
	\item 
	\textit{Examining \eqref{eq:A0B0-Type-I.3}}. Let us keep in mind that we have made the substitutions in \eqref{eq:sA0-Type-I.3} and \eqref{eq:sB0-Type-I.3}: $(\ell,m)\mapsto (2\ell+1,2m+1)$, $a=1$ and $b=1$. Also, $\sigma=-(\ell+m+1)k$. We want to show \eqref{eq:A0B0-Type-I.3},
	\begin{align}\label{eq:A0B0-Type-I.3-1-Eq1}
		\mathbf{H}_M\big(q^{\sigma}\sA_0\sB_{0}\big) =0.
	\end{align}
	In Corollary \ref{coro:UMH=0-J}, we further set
	\begin{align*}
		\begin{array}{rclp{3cm}rcl}
			A' & \mapsto & 0 && B' & \mapsto & 0\\
			w & \mapsto & \multicolumn{5}{l}{-(\ell+m+1)k}\\
		\end{array}
	\end{align*}
	So,
	\begin{align*}
		\sH(q)=q^{\sigma}\sA_0\sB_0.
	\end{align*}
	We may further assume that $w$ is a multiple of $d^*$. Otherwise, we know from Theorem \ref{th:UMH}~(i) that
	\begin{align*}
		\mathbf{H}_M\big(\sH(q)\big) = 0,
	\end{align*}
	which gives
	\begin{align*}
		\mathbf{H}_M\big(q^{\sigma}\sA_0\sB_0\big) =0.
	\end{align*} 
	We further compute that
	\begin{align*}
		\frac{-v\kappa+u\lambda}{d}=-\frac{(2\ell+1)(2m+1)}{d_0}(\kappa-\lambda).
	\end{align*}
	The fact that $d_0=\gcd(2\ell+1,2m+1)$ then implies that $(2\ell+1)(2m+1)/d_0$ is an odd integer. Thus, \eqref{eq:div-assump-coro} is satisfied since $(\kappa,\lambda)\in\{(0,1),(1,0)\}$. Finally, we choose $J=0$ in \eqref{eq:div-assump-coro-J}. Then
	\begin{equation*}
		\scalebox{0.85}{%
			$
			\begin{aligned}
				\dfrac{2dMAv\cdot J-2dAvw-dAuv+dBu^2+2dA'uv-2dB'u^2-Auv^2-Bu^3}{2d(Av^2+Bu^2)}&=\frac{d_0-(2\ell+1)}{2d_0},\\
				\dfrac{2dMBu\cdot J-2dBuw+dAv^2-dBuv-2dA'v^2+2dB'uv+Av^3+Bu^2v}{2d(Av^2+Bu^2)}&=\frac{d_0+(2m+1)}{2d_0}.
			\end{aligned}
			$}
	\end{equation*}
	Recall that $2\ell+1\equiv 2m+1\equiv d_0 \pmod{2d_0}$ since $d_0=\gcd(2\ell+1,2m+1)$ is odd. So \eqref{eq:div-assump-coro-J} is also satisfied. We conclude by Corollary \ref{coro:UMH=0-J} that
	\begin{align*}
		\mathbf{H}_M\big(\sH(q)\big) = 0.
	\end{align*}
	Thus, \eqref{eq:A0B0-Type-I.3-1-Eq1} is established.
	
	\item
	\textit{Examining \eqref{eq:A0B-Type-I.3}}. Let us keep in mind that we have made the substitutions in \eqref{eq:sA0-Type-I.3}, \eqref{eq:sBI-Type-I.3} and \eqref{eq:sBII-Type-I.3}: $(\ell,m)\mapsto (2\ell+1,2m+1)$, $a=1$ and $b=1$. Also, $\sigma=-(\ell+m+1)k$. Then \eqref{eq:A0B-Type-I.3} becomes
	\begin{align}\label{eq:A0B-Type-I.3-1-Eq2}
		\mathbf{H}_M\big(q^{\sigma}\sA_0\sB_{I}\big) =-(-1)^{\lambda}\mathbf{H}_M\big(q^{\sigma}\sA_0\sB_{I\!I}\big).
	\end{align}
	In Theorem \ref{th:cancel-1}, we further set
	\begin{align*}
		\begin{array}{rclp{3.5cm}rcl}
			A' & \mapsto & 0 && B' & \mapsto & 2\tau\mu\\
			w & \mapsto & \multicolumn{5}{l}{-(\ell+m+1)k+\tau k}\\
		\end{array}
	\end{align*}
	Then
	\begin{align*}
		\hw = -(\ell+m+1)k +(2m+1-\tau) k.
	\end{align*}
	So,
	\begin{align*}
		\sH(q)&=q^{\sigma}\sA_0\sB_{I},\\
		\hat{\sH}(q)&=q^{\sigma}\sA_0\sB_{I\!I}.
	\end{align*}
	Now we compute that
	\begin{align*}
		\frac{2B'v}{d^* B} = \frac{2\tau k}{d^*}.
	\end{align*}
	If $d^* \nmid 2\tau k$, then $\tau k \not\equiv 0 \pmod{d_0}$, and thus,
	\begin{align*}
		w&=-\tfrac{1}{2} (2\ell+1)k-\tfrac{1}{2}(2m+1)k+\tau k\equiv \tau k \not\equiv 0 \pmod{d_0},\\
		\hw&=-\tfrac{1}{2} (2\ell+1)k+\tfrac{1}{2}(2m+1)k-\tau k\equiv -\tau k \not\equiv 0 \pmod{d_0}.
	\end{align*}
	So $w$ and $\hw$ are nonmultiples of $d_0=d^*$. In this case we know from Theorem \ref{th:UMH}~(i) that
	\begin{align*}
		\mathbf{H}_M\big(\sH(q)\big) = \mathbf{H}_M\big(\hat{\sH}(q)\big) = 0,
	\end{align*}
	which gives
	\begin{align*}
		\mathbf{H}_M\big(q^{\sigma}\sA_0\sB_{I}\big) =-(-1)^{\lambda}\mathbf{H}_M\big(q^{\sigma}\sA_0\sB_{I\!I}\big)=0.
	\end{align*}
	Below, we assume that $d^* \mid 2\tau k$. Then the first assumption in \eqref{eq:assump-d*} is satisfied. We may also assume that $w$ is a multiple of $d^*=d_0$ according to Theorem \ref{th:cancel-1}~(i). It is easily seen that $d^*k$ is a divisor of $\gcd(u,v,w)$. Thus, we solve the following stronger system
	\begin{align*}
		\left\{
		\begin{aligned}
			K&\equiv 1 \pmod{M/d^*},\\
			K&\equiv 0 \pmod{d/(d^*k)},
		\end{aligned}
		\right.
	\end{align*}
	and choose $K=1$. Finally, we compute that
	\begin{align*}
		\frac{(A-2A')v^2-(B-2B')uv-2Buw\cdot K}{Av^2+Bu^2}&=1,\\
		\frac{B(A-2A')uv+A(B-2B')v^2+2ABvw\cdot K}{B(Av^2+Bu^2)}&=0.
	\end{align*}
	So the remaining assumptions in \eqref{eq:assump-d*} are satisfied. Finally, we compute that
	\begin{align*}
		\epsilon=(2\ell+1)\kappa.
	\end{align*}
	Thus, by Theorem \ref{th:cancel-1}~(ii),
	\begin{align*}
		\mathbf{H}_M\big(\sH(q)\big) = (-1)^{\kappa}\mathbf{H}_M\big(\hat{\sH}(q)\big).
	\end{align*}
	It follows from $(\kappa,\lambda)\in\{(0,1),(1,0)\}$ that
	\begin{align*}
		\mathbf{H}_M\big(q^{\sigma}\sA_0\sB_{I}\big)+(-1)^{\lambda}\mathbf{H}_M\big(q^{\sigma}\sA_0\sB_{I\!I}\big)&=\mathbf{H}_M\big(\sH(q)\big)+(-1)^{\lambda}\mathbf{H}_M\big(\hat{\sH}(q)\big)\\
		&=0.
	\end{align*}
	Thus, \eqref{eq:A0B-Type-I.3-1-Eq2} is established.
	
	\item 
	\textit{Examining \eqref{eq:AB0-Type-I.3}}.
	Let us keep in mind that we have made the substitutions in \eqref{eq:sAI-Type-I.3}, \eqref{eq:sAII-Type-I.3} and \eqref{eq:sB0-Type-I.3}: $(\ell,m)\mapsto (2\ell+1,2m+1)$, $a=1$ and $b=1$. Also, $\sigma=-(\ell+m+1)k$. Then \eqref{eq:AB0-Type-I.3} becomes
	\begin{align}\label{eq:AB0-Type-I.3-1-Eq2}
		\mathbf{H}_M\big(q^{\sigma}\sB_0\sA_{I}\big) =-(-1)^{\kappa}\mathbf{H}_M\big(q^{\sigma}\sB_0\sA_{I\!I}\big).
	\end{align}
	To make use of Theorem \ref{th:cancel-1}, we need to swap the choice of $(\kappa,u,A)$ and $(\lambda,v,B)$ in our initial setting. In other words, in Theorem \ref{th:cancel-1}, we set
	\begin{align*}
		\begin{array}{rclp{3cm}rcl}
			\kappa & \mapsto & (2m+1)\lambda && \lambda & \mapsto & (2\ell+1)\kappa\\
			u & \mapsto & (2m+1)k && v & \mapsto & (2\ell+1)k\\
			A & \mapsto & 2(2m+1)\mu && B & \mapsto & (2\ell+1)\mu\\
			A' & \mapsto & 0 && B' & \mapsto & \xi \mu\\
			M & \mapsto & \multicolumn{5}{l}{4\ell+2m+3}\\
			w & \mapsto & \multicolumn{5}{l}{-(\ell+m+1)k+\xi k}\\
		\end{array}
	\end{align*}
	Then
	\begin{align*}
		\hw = -(\ell+m+1)k +(2\ell+1-\xi) k.
	\end{align*}
	So,
	\begin{align*}
		\sH(q)&=q^{\sigma}\sB_0\sA_{I},\\
		\hat{\sH}(q)&=q^{\sigma}\sB_0\sA_{I\!I}.
	\end{align*}
	Now we compute that
	\begin{align*}
		\frac{2B'v}{d^* B} = \frac{2\xi k}{d^*}.
	\end{align*}
	If $d^* \nmid 2\xi k$, then $\xi k \not\equiv 0 \pmod{d_0}$, and thus,
	\begin{align*}
		w&\equiv \xi k \not\equiv 0 \pmod{d_0},\\
		\hw&\equiv -\xi k \not\equiv 0 \pmod{d_0},
	\end{align*}
	implying from Theorem \ref{th:UMH}~(i) that
	\begin{align*}
		\mathbf{H}_M\big(q^{\sigma}\sB_0\sA_{I}\big) =-(-1)^{\kappa}\mathbf{H}_M\big(q^{\sigma}\sB_0\sA_{I\!I}\big)=0.
	\end{align*}
	Below, we assume that $d^* \mid 2\xi k$. Then the first assumption in \eqref{eq:assump-d*} is satisfied. We may also assume that $w$ is a multiple of $d^*=d_0$ according to Theorem \ref{th:cancel-1}~(i). With the choice of $K=1$, we compute that
	\begin{align*}
		\frac{(A-2A')v^2-(B-2B')uv-2Buw\cdot K}{Av^2+Bu^2}&=1,\\
		\frac{B(A-2A')uv+A(B-2B')v^2+2ABvw\cdot K}{B(Av^2+Bu^2)}&=0.
	\end{align*}
	So the remaining assumptions in \eqref{eq:assump-d*} are satisfied. Finally, we compute that
	\begin{align*}
		\epsilon=(2m+1)\lambda.
	\end{align*}
	Thus, by Theorem \ref{th:cancel-1}~(ii),
	\begin{align*}
		\mathbf{H}_M\big(\sH(q)\big) = (-1)^{\lambda}\mathbf{H}_M\big(\hat{\sH}(q)\big).
	\end{align*}
	It follows from $(\kappa,\lambda)\in\{(0,1),(1,0)\}$ that
	\begin{align*}
		\mathbf{H}_M\big(q^{\sigma}\sB_0\sA_{I}\big)+(-1)^{\kappa}\mathbf{H}_M\big(q^{\sigma}\sB_0\sA_{I\!I}\big)&=\mathbf{H}_M\big(\sH(q)\big)+(-1)^{\kappa}\mathbf{H}_M\big(\hat{\sH}(q)\big)\\
		&=0.
	\end{align*}
	Thus, \eqref{eq:AB0-Type-I.3-1-Eq2} is established.
	
	\item
	\textit{Examining \eqref{eq:AB-Type-I.3}}. Let us keep in mind that we have made the substitutions in \eqref{eq:sAI-Type-I.3}, \eqref{eq:sAII-Type-I.3}, \eqref{eq:sBI-Type-I.3} and \eqref{eq:sBII-Type-I.3}: $(\ell,m)\mapsto (2\ell+1,2m+1)$, $a=1$ and $b=1$. Also, $\sigma=-(\ell+m+1)k$. Then \eqref{eq:AB-Type-I.3} becomes
	\begin{align}\label{eq:AB-Type-I.3-1-Eq2}
		\mathbf{H}_M\big(q^{\sigma}\sA_{I}\sB_{I}\big) =-(-1)^{\kappa+\lambda}\mathbf{H}_M\big(q^{\sigma}\sA_{I\!I}\sB_{I\!I}\big).
	\end{align}
	In Theorem \ref{th:cancel-2}, we further set
	\begin{align*}
		\begin{array}{rclp{4cm}rcl}
			A' & \mapsto & \xi \mu && B' & \mapsto & 2\tau\mu\\
			w & \mapsto & \multicolumn{5}{l}{-(\ell+m+1)k+\xi k+\tau k}\\
		\end{array}
	\end{align*}
	Then
	\begin{align*}
		\chw = -(\ell+m+1)k +(2\ell+1-\xi) k+(2m+1-\tau) k.
	\end{align*}
	So,
	\begin{align*}
		\sH(q)&=q^{\sigma}\sA_{I}\sB_{I},\\
		\check{\sH}(q)&=q^{\sigma}\sA_{I\!I}\sB_{I\!I}.
	\end{align*}
	Now we compute that
	\begin{align*}
		\frac{2(A'Bu+AB'v)}{d^*AB} = \frac{2(\xi +\tau )k}{d^*}.
	\end{align*}
	Similarly, if $d^* \nmid 2(\xi +\tau )k$, then $(\xi+\tau) k \not\equiv 0 \pmod{d_0}$, and thus,
	\begin{align*}
		w&\equiv (\xi+\tau) k \not\equiv 0 \pmod{d_0},\\
		\chw&\equiv -(\xi+\tau) k \not\equiv 0 \pmod{d_0},
	\end{align*}
	implying from Theorem \ref{th:UMH}~(i) that
	\begin{align*}
		\mathbf{H}_M\big(q^{\sigma}\sA_{I}\sB_{I}\big) =-(-1)^{\kappa+\lambda}\mathbf{H}_M\big(q^{\sigma}\sA_{I\!I}\sB_{I\!I}\big)=0.
	\end{align*}
	Below, we assume that $d^* \mid 2(\xi +\tau )k$. Then the first assumption in \eqref{eq:assump-2-d*} is satisfied. We may also assume that $w$ is a multiple of $d^*=d_0$ according to Theorem \ref{th:cancel-2}~(i). With the choice of $K=1$, we compute that
	\begin{align*}
		\frac{B(A-2A')u^2+A(B-2B')uv+2ABuw\cdot K}{A(Av^2+Bu^2)}&=0,\\
		\frac{B(A-2A')uv+A(B-2B')v^2+2ABvw\cdot K}{B(Av^2+Bu^2)}&=0.
	\end{align*}
	So the remaining assumptions in \eqref{eq:assump-2-d*} are satisfied. Finally, we compute that
	\begin{align*}
		\varepsilon=0.
	\end{align*}
	Thus, by Theorem \ref{th:cancel-2}~(ii),
	\begin{align*}
		\mathbf{H}_M\big(\sH(q)\big) = \mathbf{H}_M\big(\check{\sH}(q)\big).
	\end{align*}
	It follows from $(\kappa,\lambda)\in\{(0,1),(1,0)\}$ that
	\begin{align*}
		\mathbf{H}_M\big(q^{\sigma}\sA_{I}\sB_{I}\big)+(-1)^{\kappa+\lambda}\mathbf{H}_M\big(q^{\sigma}\sA_{I\!I}\sB_{I\!I}\big)&=\mathbf{H}_M\big(\sH(q)\big)+(-1)^{\kappa+\lambda}\mathbf{H}_M\big(\check{\sH}(q)\big)\\
		&=0.
	\end{align*}
	Thus, \eqref{eq:AB-Type-I.3-1-Eq2} is established.
	
\end{itemize}

\subsubsection{Equation \eqref{eq:I.3-2}}

Recall that $(\ell,m)\mapsto (2\ell+1,2m+1)$, $a=1$ and $b=4$. Also, $M=2\ell+16m+9$ and $\sigma=-(2\ell+12m+7)k$. Further, $(\kappa,\lambda)\in\{(0,1),(1,1)\}$ and $k$ is such that $\gcd(k,M)=1$.

In Corollary \ref{coro:UMH=0-J} and Theorems \ref{th:cancel-1} and \ref{th:cancel-2}, we set
\begin{align*}
	\begin{array}{rclp{3cm}rcl}
		\kappa & \mapsto & (2\ell+1)\kappa && \lambda & \mapsto & (2m+1)\lambda\\
		u & \mapsto & (2\ell+1)k && v & \mapsto & 4(2m+1)k\\
		A & \mapsto & (2\ell+1)\mu && B & \mapsto & 2(2m+1)\mu\\
		M & \mapsto & \multicolumn{5}{l}{2\ell+16m+9}\\
	\end{array}
\end{align*}
Let
\begin{align*}
	d_0=\gcd\big(2\ell+1,4(2m+1)\big)=\gcd\big(2\ell+1,2m+1\big).
\end{align*}
Then
\begin{align*}
	d=\gcd(u,v)=d_0 k.
\end{align*}
Noticing that $M=(2\ell+1)+8(2m+1)$, and that $k$ is coprime to $M$, we have
\begin{align*}
	d^*&=\gcd(u,v,M)=d_0,\\
	d_u&=\gcd(u,M)=d_0,\\
	d_v&=\gcd(v,M)=d_0.
\end{align*}
We compute that
\begin{align*}
	\frac{Av\cdot dM}{d_u(Av^2+Bu^2)}&=2,\\
	\frac{Bu\cdot dM}{d_v(Av^2+Bu^2)}&=1.
\end{align*}
Thus, the second and third assumptions in \eqref{eq:div-assump-coro}, \eqref{eq:assump-d*} and \eqref{eq:assump-2-d*} are satisfied.

\begin{itemize}[leftmargin=*,align=left,itemsep=5pt]
	\renewcommand{\labelitemi}{\scriptsize$\blacktriangleright$}
	
	\item 
	\textit{Examining \eqref{eq:A0B0-Type-I.3}}. Let us keep in mind that we have made the substitutions in \eqref{eq:sA0-Type-I.3} and \eqref{eq:sB0-Type-I.3}: $(\ell,m)\mapsto (2\ell+1,2m+1)$, $a=1$ and $b=4$. Also, $\sigma=-(2\ell+12m+7)k$. We want to show \eqref{eq:A0B0-Type-I.3},
	\begin{align}\label{eq:A0B0-Type-I.3-2-Eq1}
		H_M\big(q^{\sigma}\sA_0\sB_{0}\big) =0.
	\end{align}
	In Corollary \ref{coro:UMH=0-J}, we further set
	\begin{align*}
		\begin{array}{rclp{3cm}rcl}
			A' & \mapsto & 0 && B' & \mapsto & 0\\
			w & \mapsto & \multicolumn{5}{l}{-(2\ell+12m+7)k}\\
		\end{array}
	\end{align*}
	So,
	\begin{align*}
		\sH(q)=q^{\sigma}\sA_0\sB_0.
	\end{align*}
	We may further assume that $w$ is a multiple of $d^*$. Otherwise, we know from Theorem \ref{th:UMH}~(i) that
	\begin{align*}
		H_M\big(\sH(q)\big) = 0,
	\end{align*}
	which gives
	\begin{align*}
		H_M\big(q^{\sigma}\sA_0\sB_0\big) =0.
	\end{align*} 
	We further compute that
	\begin{align*}
		\frac{-v\kappa+u\lambda}{d}=-\frac{(2\ell+1)(2m+1)}{d_0}(4\kappa-\lambda).
	\end{align*}
	The fact that $d_0=\gcd(2\ell+1,2m+1)$ then implies that $(2\ell+1)(2m+1)/d_0$ is an odd integer. Thus, \eqref{eq:div-assump-coro} is satisfied since $(\kappa,\lambda)\in\{(0,1),(1,1)\}$. Finally, we choose $J=0$ in \eqref{eq:div-assump-coro-J}. Then
	\begin{equation*}
		\scalebox{0.85}{%
			$
			\begin{aligned}
			\dfrac{2dMAv\cdot J-2dAvw-dAuv+dBu^2+2dA'uv-2dB'u^2-Auv^2-Bu^3}{2d(Av^2+Bu^2)}&=\frac{3d_0-(2\ell+1)}{2d_0},\\
			\dfrac{2dMBu\cdot J-2dBuw+dAv^2-dBuv-2dA'v^2+2dB'uv+Av^3+Bu^2v}{2d(Av^2+Bu^2)}&=\frac{d_0+2(2m+1)}{d_0}.
			\end{aligned}
			$}
	\end{equation*}
	Recall that $2\ell+1\equiv 2m+1\equiv d_0 \pmod{2d_0}$ since $d_0=\gcd(2\ell+1,2m+1)$ is odd. So \eqref{eq:div-assump-coro-J} is also satisfied. We conclude by Corollary \ref{coro:UMH=0-J} that
	\begin{align*}
		H_M\big(\sH(q)\big) = 0.
	\end{align*}
	Thus, \eqref{eq:A0B0-Type-I.3-2-Eq1} is established.
	
	\item
	\textit{Examining \eqref{eq:A0B-Type-I.3}}. Let us keep in mind that we have made the substitutions in \eqref{eq:sA0-Type-I.3}, \eqref{eq:sBI-Type-I.3} and \eqref{eq:sBII-Type-I.3}: $(\ell,m)\mapsto (2\ell+1,2m+1)$, $a=1$ and $b=4$. Also, $\sigma=-(2\ell+12m+7)k$. Then \eqref{eq:A0B-Type-I.3} becomes
	\begin{align}\label{eq:A0B-Type-I.3-2-Eq2}
		H_M\big(q^{\sigma}\sA_0\sB_{I}\big) =-(-1)^{\lambda}H_M\big(q^{\sigma}\sA_0\sB_{I\!I}\big).
	\end{align}
	In Theorem \ref{th:cancel-1}, we further set
	\begin{align*}
		\begin{array}{rclp{4cm}rcl}
			A' & \mapsto & 0 && B' & \mapsto & 2\tau\mu\\
			w & \mapsto & \multicolumn{5}{l}{-(2\ell+12m+7)k+4\tau k}\\
		\end{array}
	\end{align*}
	Then
	\begin{align*}
		\hw = -(2\ell+12m+7)k +4(2m+1-\tau) k.
	\end{align*}
	So,
	\begin{align*}
		\sH(q)&=q^{\sigma}\sA_0\sB_{I},\\
		\hat{\sH}(q)&=q^{\sigma}\sA_0\sB_{I\!I}.
	\end{align*}
	Now we compute that
	\begin{align*}
		\frac{2B'v}{d^* B} = \frac{8\tau k}{d^*}.
	\end{align*}
	If $d^* \nmid 8\tau k$, then $4\tau k \not\equiv 0 \pmod{d_0}$, and thus,
	\begin{align*}
		w&=- (2\ell+1)k-6(2m+1)k+4\tau k\equiv 4\tau k \not\equiv 0 \pmod{d_0},\\
		\hw&=- (2\ell+1)k-2(2m+1)k-4\tau k\equiv -4\tau k \not\equiv 0 \pmod{d_0}.
	\end{align*}
	So $w$ and $\hw$ are nonmultiples of $d_0=d^*$. In this case we know from Theorem \ref{th:UMH}~(i) that
	\begin{align*}
		H_M\big(\sH(q)\big) = H_M\big(\hat{\sH}(q)\big) = 0,
	\end{align*}
	which gives
	\begin{align*}
		H_M\big(q^{\sigma}\sA_0\sB_{I}\big) =-(-1)^{\lambda}H_M\big(q^{\sigma}\sA_0\sB_{I\!I}\big)=0.
	\end{align*}
	Below, we assume that $d^* \mid 8\tau k$. Then the first assumption in \eqref{eq:assump-d*} is satisfied. We may also assume that $w$ is a multiple of $d^*=d_0$ according to Theorem \ref{th:cancel-1}~(i). It is easily seen that $d^*k$ is a divisor of $\gcd(u,v,w)$. Thus, we solve the following stronger system
	\begin{align*}
		\left\{
		\begin{aligned}
			K&\equiv 1 \pmod{M/d^*},\\
			K&\equiv 0 \pmod{d/(d^*k)},
		\end{aligned}
		\right.
	\end{align*}
	and choose $K=1$. Finally, we compute that
	\begin{align*}
		\frac{(A-2A')v^2-(B-2B')uv-2Buw\cdot K}{Av^2+Bu^2}&=2,\\
		\frac{B(A-2A')uv+A(B-2B')v^2+2ABvw\cdot K}{B(Av^2+Bu^2)}&=-2.
	\end{align*}
	So the remaining assumptions in \eqref{eq:assump-d*} are satisfied. Finally, we compute that
	\begin{align*}
		\epsilon=(4\ell+2)\kappa+(4m+2)\lambda.
	\end{align*}
	Thus, by Theorem \ref{th:cancel-1}~(ii),
	\begin{align*}
		H_M\big(\sH(q)\big) = H_M\big(\hat{\sH}(q)\big).
	\end{align*}
	It follows from $(\kappa,\lambda)\in\{(0,1),(1,1)\}$ that
	\begin{align*}
		H_M\big(q^{\sigma}\sA_0\sB_{I}\big)+(-1)^{\lambda}H_M\big(q^{\sigma}\sA_0\sB_{I\!I}\big)&=H_M\big(\sH(q)\big)+(-1)^{\lambda}H_M\big(\hat{\sH}(q)\big)\\
		&=0.
	\end{align*}
	Thus, \eqref{eq:A0B-Type-I.3-2-Eq2} is established.
	
	\item 
	\textit{Examining \eqref{eq:AB0-Type-I.3}}.
	Let us keep in mind that we have made the substitutions in \eqref{eq:sAI-Type-I.3}, \eqref{eq:sAII-Type-I.3} and \eqref{eq:sB0-Type-I.3}: $(\ell,m)\mapsto (2\ell+1,2m+1)$, $a=1$ and $b=4$. Also, $\sigma=-(2\ell+12m+7)k$. Then \eqref{eq:AB0-Type-I.3} becomes
	\begin{align}\label{eq:AB0-Type-I.3-2-Eq2}
		H_M\big(q^{\sigma}\sB_0\sA_{I}\big) =-(-1)^{\kappa}H_M\big(q^{\sigma}\sB_0\sA_{I\!I}\big).
	\end{align}
	To make use of Theorem \ref{th:cancel-1}, we need to swap the choice of $(\kappa,u,A)$ and $(\lambda,v,B)$ in our initial setting. In other words, in Theorem \ref{th:cancel-1}, we set
	\begin{align*}
		\begin{array}{rclp{3cm}rcl}
			\kappa & \mapsto & (2m+1)\lambda && \lambda & \mapsto & (2\ell+1)\kappa\\
			u & \mapsto & 4(2m+1)k && v & \mapsto & (2\ell+1)k\\
			A & \mapsto & 2(2m+1)\mu && B & \mapsto & (2\ell+1)\mu\\
			A' & \mapsto & 0 && B' & \mapsto & \xi \mu\\
			M & \mapsto & \multicolumn{5}{l}{2\ell+16m+9}\\
			w & \mapsto & \multicolumn{5}{l}{-(2\ell+12m+7)k+\xi k}\\
		\end{array}
	\end{align*}
	Then
	\begin{align*}
		\hw = -(2\ell+12m+7)k +(2\ell+1-\xi) k.
	\end{align*}
	So,
	\begin{align*}
		\sH(q)&=q^{\sigma}\sB_0\sA_{I},\\
		\hat{\sH}(q)&=q^{\sigma}\sB_0\sA_{I\!I}.
	\end{align*}
	Now we compute that
	\begin{align*}
		\frac{2B'v}{d^* B} = \frac{2\xi k}{d^*}.
	\end{align*}
	If $d^* \nmid 2\xi k$, then $\xi k \not\equiv 0 \pmod{d_0}$, and thus,
	\begin{align*}
		w&\equiv \xi k \not\equiv 0 \pmod{d_0},\\
		\hw&\equiv -\xi k \not\equiv 0 \pmod{d_0},
	\end{align*}
	implying from Theorem \ref{th:UMH}~(i) that
	\begin{align*}
		H_M\big(q^{\sigma}\sB_0\sA_{I}\big) =-(-1)^{\kappa}H_M\big(q^{\sigma}\sB_0\sA_{I\!I}\big)=0.
	\end{align*}
	Below, we assume that $d^* \mid 2\xi k$. Then the first assumption in \eqref{eq:assump-d*} is satisfied. We may also assume that $w$ is a multiple of $d^*=d_0$ according to Theorem \ref{th:cancel-1}~(i). With the choice of $K=1$, we compute that
	\begin{align*}
		\frac{(A-2A')v^2-(B-2B')uv-2Buw\cdot K}{Av^2+Bu^2}&=3,\\
		\frac{B(A-2A')uv+A(B-2B')v^2+2ABvw\cdot K}{B(Av^2+Bu^2)}&=-1.
	\end{align*}
	So the remaining assumptions in \eqref{eq:assump-d*} are satisfied. Finally, we compute that
	\begin{align*}
		\epsilon=(2\ell+1)\kappa+(6m+3)\lambda.
	\end{align*}
	Thus, by Theorem \ref{th:cancel-1}~(ii),
	\begin{align*}
		H_M\big(\sH(q)\big) = (-1)^{\kappa+\lambda}H_M\big(\hat{\sH}(q)\big).
	\end{align*}
	It follows from $(\kappa,\lambda)\in\{(0,1),(1,1)\}$ that
	\begin{align*}
		H_M\big(q^{\sigma}\sB_0\sA_{I}\big)+(-1)^{\kappa}H_M\big(q^{\sigma}\sB_0\sA_{I\!I}\big)&=H_M\big(\sH(q)\big)+(-1)^{\kappa}H_M\big(\hat{\sH}(q)\big)\\
		&=0.
	\end{align*}
	Thus, \eqref{eq:AB0-Type-I.3-2-Eq2} is established.
	
	\item
	\textit{Examining \eqref{eq:AB-Type-I.3}}. Let us keep in mind that we have made the substitutions in \eqref{eq:sAI-Type-I.3}, \eqref{eq:sAII-Type-I.3}, \eqref{eq:sBI-Type-I.3} and \eqref{eq:sBII-Type-I.3}: $(\ell,m)\mapsto (2\ell+1,2m+1)$, $a=1$ and $b=4$. Also, $\sigma=-(2\ell+12m+7)k$. Then \eqref{eq:AB-Type-I.3} becomes
	\begin{align}\label{eq:AB-Type-I.3-2-Eq2}
		H_M\big(q^{\sigma}\sA_{I}\sB_{I}\big) =-(-1)^{\kappa+\lambda}H_M\big(q^{\sigma}\sA_{I\!I}\sB_{I\!I}\big).
	\end{align}
	In Theorem \ref{th:cancel-2}, we further set
	\begin{align*}
		\begin{array}{rclp{4.5cm}rcl}
			A' & \mapsto & \xi \mu && B' & \mapsto & 2\tau\mu\\
			w & \mapsto & \multicolumn{5}{l}{-(2\ell+12m+7)k+\xi k+4\tau k}\\
		\end{array}
	\end{align*}
	Then
	\begin{align*}
		\chw = -(2\ell+12m+7)k +(2\ell+1-\xi) k+4(2m+1-\tau) k.
	\end{align*}
	So,
	\begin{align*}
		\sH(q)&=q^{\sigma}\sA_{I}\sB_{I},\\
		\check{\sH}(q)&=q^{\sigma}\sA_{I\!I}\sB_{I\!I}.
	\end{align*}
	Now we compute that
	\begin{align*}
		\frac{2(A'Bu+AB'v)}{d^*AB} = \frac{2(\xi +4\tau )k}{d^*}.
	\end{align*}
	Similarly, if $d^* \nmid 2(\xi +4\tau )k$, then $(\xi+4\tau) k \not\equiv 0 \pmod{d_0}$, and thus,
	\begin{align*}
		w&\equiv (\xi+4\tau) k \not\equiv 0 \pmod{d_0},\\
		\chw&\equiv -(\xi+4\tau) k \not\equiv 0 \pmod{d_0},
	\end{align*}
	implying from Theorem \ref{th:UMH}~(i) that
	\begin{align*}
		H_M\big(q^{\sigma}\sA_{I}\sB_{I}\big) =-(-1)^{\kappa+\lambda}H_M\big(q^{\sigma}\sA_{I\!I}\sB_{I\!I}\big)=0.
	\end{align*}
	Below, we assume that $d^* \mid 2(\xi +4\tau )k$. Then the first assumption in \eqref{eq:assump-2-d*} is satisfied. We may also assume that $w$ is a multiple of $d^*=d_0$ according to Theorem \ref{th:cancel-2}~(i). With the choice of $K=1$, we compute that
	\begin{align*}
		\frac{B(A-2A')u^2+A(B-2B')uv+2ABuw\cdot K}{A(Av^2+Bu^2)}&=-1,\\
		\frac{B(A-2A')uv+A(B-2B')v^2+2ABvw\cdot K}{B(Av^2+Bu^2)}&=-2.
	\end{align*}
	So the remaining assumptions in \eqref{eq:assump-2-d*} are satisfied. Finally, we compute that
	\begin{align*}
		\varepsilon=(2\ell+1)\kappa+(4m+2)\lambda.
	\end{align*}
	Thus, by Theorem \ref{th:cancel-2}~(ii),
	\begin{align*}
		H_M\big(\sH(q)\big) = (-1)^{\kappa}H_M\big(\check{\sH}(q)\big).
	\end{align*}
	It follows from $(\kappa,\lambda)\in\{(0,1),(1,1)\}$ that
	\begin{align*}
		H_M\big(q^{\sigma}\sA_{I}\sB_{I}\big)+(-1)^{\kappa+\lambda}H_M\big(q^{\sigma}\sA_{I\!I}\sB_{I\!I}\big)&=H_M\big(\sH(q)\big)+(-1)^{\kappa+\lambda}H_M\big(\check{\sH}(q)\big)\\
		&=0.
	\end{align*}
	Thus, \eqref{eq:AB-Type-I.3-2-Eq2} is established.
	
\end{itemize}

\subsubsection{Equation \eqref{eq:I.3-3}}

Recall that $(\ell,m)\mapsto (2\ell+1,2m+1)$, $a=2$ and $b=1$. Also, $M=16\ell+2m+9$ and $\sigma=-(10\ell+2m+6)k$. Further, $(\kappa,\lambda)\in\{(1,0),(1,1)\}$ and $k$ is such that $\gcd(k,M)=1$.

In Corollary \ref{coro:UMH=0-J} and Theorems \ref{th:cancel-1} and \ref{th:cancel-2}, we set
\begin{align*}
	\begin{array}{rclp{3cm}rcl}
		\kappa & \mapsto & (2\ell+1)\kappa && \lambda & \mapsto & (2m+1)\lambda\\
		u & \mapsto & 2(2\ell+1)k && v & \mapsto & (2m+1)k\\
		A & \mapsto & (2\ell+1)\mu && B & \mapsto & 2(2m+1)\mu\\
		M & \mapsto & \multicolumn{5}{l}{16\ell+2m+9}\\
	\end{array}
\end{align*}
Let
\begin{align*}
	d_0=\gcd\big(2(2\ell+1),2m+1\big)=\gcd\big(2\ell+1,2m+1\big).
\end{align*}
Then
\begin{align*}
	d=\gcd(u,v)=d_0 k.
\end{align*}
Noticing that $M=8(2\ell+1)+(2m+1)$, and that $k$ is coprime to $M$, we have
\begin{align*}
	d^*&=\gcd(u,v,M)=d_0,\\
	d_u&=\gcd(u,M)=d_0,\\
	d_v&=\gcd(v,M)=d_0.
\end{align*}
We compute that
\begin{align*}
	\frac{Av\cdot dM}{d_u(Av^2+Bu^2)}&=1,\\
	\frac{Bu\cdot dM}{d_v(Av^2+Bu^2)}&=4.
\end{align*}
Thus, the second and third assumptions in \eqref{eq:div-assump-coro}, \eqref{eq:assump-d*} and \eqref{eq:assump-2-d*} are satisfied.

\begin{itemize}[leftmargin=*,align=left,itemsep=5pt]
	\renewcommand{\labelitemi}{\scriptsize$\blacktriangleright$}
	
	\item 
	\textit{Examining \eqref{eq:A0B0-Type-I.3}}. Let us keep in mind that we have made the substitutions in \eqref{eq:sA0-Type-I.3} and \eqref{eq:sB0-Type-I.3}: $(\ell,m)\mapsto (2\ell+1,2m+1)$, $a=2$ and $b=1$. Also, $\sigma=-(10\ell+2m+6)k$. We want to show \eqref{eq:A0B0-Type-I.3},
	\begin{align}\label{eq:A0B0-Type-I.3-3-Eq1}
		H_M\big(q^{\sigma}\sA_0\sB_{0}\big) =0.
	\end{align}
	In Corollary \ref{coro:UMH=0-J}, we further set
	\begin{align*}
		\begin{array}{rclp{3cm}rcl}
			A' & \mapsto & 0 && B' & \mapsto & 0\\
			w & \mapsto & \multicolumn{5}{l}{-(10\ell+2m+6)k}\\
		\end{array}
	\end{align*}
	So,
	\begin{align*}
		\sH(q)=q^{\sigma}\sA_0\sB_0.
	\end{align*}
	We may further assume that $w$ is a multiple of $d^*$. Otherwise, we know from Theorem \ref{th:UMH}~(i) that
	\begin{align*}
		H_M\big(\sH(q)\big) = 0,
	\end{align*}
	which gives
	\begin{align*}
		H_M\big(q^{\sigma}\sA_0\sB_0\big) =0.
	\end{align*} 
	We further compute that
	\begin{align*}
		\frac{-v\kappa+u\lambda}{d}=-\frac{(2\ell+1)(2m+1)}{d_0}(\kappa-2\lambda).
	\end{align*}
	The fact that $d_0=\gcd(2\ell+1,2m+1)$ then implies that $(2\ell+1)(2m+1)/d_0$ is an odd integer. Thus, \eqref{eq:div-assump-coro} is satisfied since $(\kappa,\lambda)\in\{(1,0),(1,1)\}$. Finally, we choose $J=0$ in \eqref{eq:div-assump-coro-J}. Then
	\begin{equation*}
		\scalebox{0.85}{%
			$
			\begin{aligned}
			\dfrac{2dMAv\cdot J-2dAvw-dAuv+dBu^2+2dA'uv-2dB'u^2-Auv^2-Bu^3}{2d(Av^2+Bu^2)}&=\frac{d_0-(2\ell+1)}{d_0},\\
			\dfrac{2dMBu\cdot J-2dBuw+dAv^2-dBuv-2dA'v^2+2dB'uv+Av^3+Bu^2v}{2d(Av^2+Bu^2)}&=\frac{5d_0+(2m+1)}{2d_0}.
			\end{aligned}
			$}
	\end{equation*}
	Recall that $2\ell+1\equiv 2m+1\equiv d_0 \pmod{2d_0}$ since $d_0=\gcd(2\ell+1,2m+1)$ is odd. So \eqref{eq:div-assump-coro-J} is also satisfied. We conclude by Corollary \ref{coro:UMH=0-J} that
	\begin{align*}
		H_M\big(\sH(q)\big) = 0.
	\end{align*}
	Thus, \eqref{eq:A0B0-Type-I.3-3-Eq1} is established.
	
	\item
	\textit{Examining \eqref{eq:A0B-Type-I.3}}. Let us keep in mind that we have made the substitutions in \eqref{eq:sA0-Type-I.3}, \eqref{eq:sBI-Type-I.3} and \eqref{eq:sBII-Type-I.3}: $(\ell,m)\mapsto (2\ell+1,2m+1)$, $a=2$ and $b=1$. Also, $\sigma=-(10\ell+2m+6)k$. Then \eqref{eq:A0B-Type-I.3} becomes
	\begin{align}\label{eq:A0B-Type-I.3-3-Eq2}
		H_M\big(q^{\sigma}\sA_0\sB_{I}\big) =-(-1)^{\lambda}H_M\big(q^{\sigma}\sA_0\sB_{I\!I}\big).
	\end{align}
	In Theorem \ref{th:cancel-1}, we further set
	\begin{align*}
		\begin{array}{rclp{3.5cm}rcl}
			A' & \mapsto & 0 && B' & \mapsto & 2\tau\mu\\
			w & \mapsto & \multicolumn{5}{l}{-(10\ell+2m+6)k+\tau k}\\
		\end{array}
	\end{align*}
	Then
	\begin{align*}
		\hw = -(10\ell+2m+6)k +(2m+1-\tau) k.
	\end{align*}
	So,
	\begin{align*}
		\sH(q)&=q^{\sigma}\sA_0\sB_{I},\\
		\hat{\sH}(q)&=q^{\sigma}\sA_0\sB_{I\!I}.
	\end{align*}
	Now we compute that
	\begin{align*}
		\frac{2B'v}{d^* B} = \frac{2\tau k}{d^*}.
	\end{align*}
	If $d^* \nmid 2\tau k$, then $\tau k \not\equiv 0 \pmod{d_0}$, and thus,
	\begin{align*}
		w&=-5 (2\ell+1)k-(2m+1)k+\tau k\equiv \tau k \not\equiv 0 \pmod{d_0},\\
		\hw&=-5 (2\ell+1)k-\tau k\equiv -\tau k \not\equiv 0 \pmod{d_0}.
	\end{align*}
	So $w$ and $\hw$ are nonmultiples of $d_0=d^*$. In this case we know from Theorem \ref{th:UMH}~(i) that
	\begin{align*}
		H_M\big(\sH(q)\big) = H_M\big(\hat{\sH}(q)\big) = 0,
	\end{align*}
	which gives
	\begin{align*}
		H_M\big(q^{\sigma}\sA_0\sB_{I}\big) =-(-1)^{\lambda}H_M\big(q^{\sigma}\sA_0\sB_{I\!I}\big)=0.
	\end{align*}
	Below, we assume that $d^* \mid 2\tau k$. Then the first assumption in \eqref{eq:assump-d*} is satisfied. We may also assume that $w$ is a multiple of $d^*=d_0$ according to Theorem \ref{th:cancel-1}~(i). It is easily seen that $d^*k$ is a divisor of $\gcd(u,v,w)$. Thus, we solve the following stronger system
	\begin{align*}
		\left\{
		\begin{aligned}
			K&\equiv 1 \pmod{M/d^*},\\
			K&\equiv 0 \pmod{d/(d^*k)},
		\end{aligned}
		\right.
	\end{align*}
	and choose $K=1$. Finally, we compute that
	\begin{align*}
		\frac{(A-2A')v^2-(B-2B')uv-2Buw\cdot K}{Av^2+Bu^2}&=5,\\
		\frac{B(A-2A')uv+A(B-2B')v^2+2ABvw\cdot K}{B(Av^2+Bu^2)}&=-1.
	\end{align*}
	So the remaining assumptions in \eqref{eq:assump-d*} are satisfied. Finally, we compute that
	\begin{align*}
		\epsilon=(10\ell+5)\kappa+(2m+1)\lambda.
	\end{align*}
	Thus, by Theorem \ref{th:cancel-1}~(ii),
	\begin{align*}
		H_M\big(\sH(q)\big) = (-1)^{\kappa+\lambda}H_M\big(\hat{\sH}(q)\big).
	\end{align*}
	It follows from $(\kappa,\lambda)\in\{(1,0),(1,1)\}$ that
	\begin{align*}
		H_M\big(q^{\sigma}\sA_0\sB_{I}\big)+(-1)^{\lambda}H_M\big(q^{\sigma}\sA_0\sB_{I\!I}\big)&=H_M\big(\sH(q)\big)+(-1)^{\lambda}H_M\big(\hat{\sH}(q)\big)\\
		&=0.
	\end{align*}
	Thus, \eqref{eq:A0B-Type-I.3-3-Eq2} is established.
	
	\item 
	\textit{Examining \eqref{eq:AB0-Type-I.3}}.
	Let us keep in mind that we have made the substitutions in \eqref{eq:sAI-Type-I.3}, \eqref{eq:sAII-Type-I.3} and \eqref{eq:sB0-Type-I.3}: $(\ell,m)\mapsto (2\ell+1,2m+1)$, $a=2$ and $b=1$. Also, $\sigma=-(10\ell+2m+6)k$. Then \eqref{eq:AB0-Type-I.3} becomes
	\begin{align}\label{eq:AB0-Type-I.3-3-Eq2}
		H_M\big(q^{\sigma}\sB_0\sA_{I}\big) =-(-1)^{\kappa}H_M\big(q^{\sigma}\sB_0\sA_{I\!I}\big).
	\end{align}
	To make use of Theorem \ref{th:cancel-1}, we need to swap the choice of $(\kappa,u,A)$ and $(\lambda,v,B)$ in our initial setting. In other words, in Theorem \ref{th:cancel-1}, we set
	\begin{align*}
		\begin{array}{rclp{3cm}rcl}
			\kappa & \mapsto & (2m+1)\lambda && \lambda & \mapsto & (2\ell+1)\kappa\\
			u & \mapsto & (2m+1)k && v & \mapsto & 2(2\ell+1)k\\
			A & \mapsto & 2(2m+1)\mu && B & \mapsto & (2\ell+1)\mu\\
			A' & \mapsto & 0 && B' & \mapsto & \xi \mu\\
			M & \mapsto & \multicolumn{5}{l}{16\ell+2m+9}\\
			w & \mapsto & \multicolumn{5}{l}{-(10\ell+2m+6)k+2\xi k}\\
		\end{array}
	\end{align*}
	Then
	\begin{align*}
		\hw = -(10\ell+2m+6)k +2(2\ell+1-\xi) k.
	\end{align*}
	So,
	\begin{align*}
		\sH(q)&=q^{\sigma}\sB_0\sA_{I},\\
		\hat{\sH}(q)&=q^{\sigma}\sB_0\sA_{I\!I}.
	\end{align*}
	Now we compute that
	\begin{align*}
		\frac{2B'v}{d^* B} = \frac{4\xi k}{d^*}.
	\end{align*}
	If $d^* \nmid 4\xi k$, then $2\xi k \not\equiv 0 \pmod{d_0}$, and thus,
	\begin{align*}
		w&\equiv 2\xi k \not\equiv 0 \pmod{d_0},\\
		\hw&\equiv -2\xi k \not\equiv 0 \pmod{d_0},
	\end{align*}
	implying from Theorem \ref{th:UMH}~(i) that
	\begin{align*}
		H_M\big(q^{\sigma}\sB_0\sA_{I}\big) =-(-1)^{\kappa}H_M\big(q^{\sigma}\sB_0\sA_{I\!I}\big)=0.
	\end{align*}
	Below, we assume that $d^* \mid 4\xi k$. Then the first assumption in \eqref{eq:assump-d*} is satisfied. We may also assume that $w$ is a multiple of $d^*=d_0$ according to Theorem \ref{th:cancel-1}~(i). With the choice of $K=1$, we compute that
	\begin{align*}
		\frac{(A-2A')v^2-(B-2B')uv-2Buw\cdot K}{Av^2+Bu^2}&=2,\\
		\frac{B(A-2A')uv+A(B-2B')v^2+2ABvw\cdot K}{B(Av^2+Bu^2)}&=-4.
	\end{align*}
	So the remaining assumptions in \eqref{eq:assump-d*} are satisfied. Finally, we compute that
	\begin{align*}
		\epsilon=(8\ell+4)\kappa+(4m+2)\lambda.
	\end{align*}
	Thus, by Theorem \ref{th:cancel-1}~(ii),
	\begin{align*}
		H_M\big(\sH(q)\big) = H_M\big(\hat{\sH}(q)\big).
	\end{align*}
	It follows from $(\kappa,\lambda)\in\{(1,0),(1,1)\}$ that
	\begin{align*}
		H_M\big(q^{\sigma}\sB_0\sA_{I}\big)+(-1)^{\kappa}H_M\big(q^{\sigma}\sB_0\sA_{I\!I}\big)&=H_M\big(\sH(q)\big)+(-1)^{\kappa}H_M\big(\hat{\sH}(q)\big)\\
		&=0.
	\end{align*}
	Thus, \eqref{eq:AB0-Type-I.3-3-Eq2} is established.
	
	\item
	\textit{Examining \eqref{eq:AB-Type-I.3}}. Let us keep in mind that we have made the substitutions in \eqref{eq:sAI-Type-I.3}, \eqref{eq:sAII-Type-I.3}, \eqref{eq:sBI-Type-I.3} and \eqref{eq:sBII-Type-I.3}: $(\ell,m)\mapsto (2\ell+1,2m+1)$, $a=2$ and $b=1$. Also, $\sigma=-(10\ell+2m+6)k$. Then \eqref{eq:AB-Type-I.3} becomes
	\begin{align}\label{eq:AB-Type-I.3-3-Eq2}
		H_M\big(q^{\sigma}\sA_{I}\sB_{I}\big) =-(-1)^{\kappa+\lambda}H_M\big(q^{\sigma}\sA_{I\!I}\sB_{I\!I}\big).
	\end{align}
	In Theorem \ref{th:cancel-2}, we further set
	\begin{align*}
		\begin{array}{rclp{4.5cm}rcl}
			A' & \mapsto & \xi \mu && B' & \mapsto & 2\tau\mu\\
			w & \mapsto & \multicolumn{5}{l}{-(10\ell+2m+6)k+2\xi k+\tau k}\\
		\end{array}
	\end{align*}
	Then
	\begin{align*}
		\chw = -(10\ell+2m+6)k +2(2\ell+1-\xi) k+(2m+1-\tau) k.
	\end{align*}
	So,
	\begin{align*}
		\sH(q)&=q^{\sigma}\sA_{I}\sB_{I},\\
		\check{\sH}(q)&=q^{\sigma}\sA_{I\!I}\sB_{I\!I}.
	\end{align*}
	Now we compute that
	\begin{align*}
		\frac{2(A'Bu+AB'v)}{d^*AB} = \frac{2(2\xi +\tau )k}{d^*}.
	\end{align*}
	Similarly, if $d^* \nmid 2(2\xi +\tau )k$, then $(2\xi+\tau) k \not\equiv 0 \pmod{d_0}$, and thus,
	\begin{align*}
		w&\equiv (2\xi+\tau) k \not\equiv 0 \pmod{d_0},\\
		\chw&\equiv -(2\xi+\tau) k \not\equiv 0 \pmod{d_0},
	\end{align*}
	implying from Theorem \ref{th:UMH}~(i) that
	\begin{align*}
		H_M\big(q^{\sigma}\sA_{I}\sB_{I}\big) =-(-1)^{\kappa+\lambda}H_M\big(q^{\sigma}\sA_{I\!I}\sB_{I\!I}\big)=0.
	\end{align*}
	Below, we assume that $d^* \mid 2(2\xi +\tau )k$. Then the first assumption in \eqref{eq:assump-2-d*} is satisfied. We may also assume that $w$ is a multiple of $d^*=d_0$ according to Theorem \ref{th:cancel-2}~(i). With the choice of $K=1$, we compute that
	\begin{align*}
		\frac{B(A-2A')u^2+A(B-2B')uv+2ABuw\cdot K}{A(Av^2+Bu^2)}&=-4,\\
		\frac{B(A-2A')uv+A(B-2B')v^2+2ABvw\cdot K}{B(Av^2+Bu^2)}&=-1.
	\end{align*}
	So the remaining assumptions in \eqref{eq:assump-2-d*} are satisfied. Finally, we compute that
	\begin{align*}
		\varepsilon=(8\ell+4)\kappa+(2m+1)\lambda.
	\end{align*}
	Thus, by Theorem \ref{th:cancel-2}~(ii),
	\begin{align*}
		H_M\big(\sH(q)\big) = (-1)^{\lambda}H_M\big(\check{\sH}(q)\big).
	\end{align*}
	It follows from $(\kappa,\lambda)\in\{(1,0),(1,1)\}$ that
	\begin{align*}
		H_M\big(q^{\sigma}\sA_{I}\sB_{I}\big)+(-1)^{\kappa+\lambda}H_M\big(q^{\sigma}\sA_{I\!I}\sB_{I\!I}\big)&=H_M\big(\sH(q)\big)+(-1)^{\kappa+\lambda}H_M\big(\check{\sH}(q)\big)\\
		&=0.
	\end{align*}
	Thus, \eqref{eq:AB-Type-I.3-3-Eq2} is established.
	
\end{itemize}

\subsection{Type II --- Theorem \ref{th:Type-II}}

This section is devoted to the three coefficient-vanishing results in Theorem \ref{th:Type-II}. Here we will use a different strategy in comparison to how we treat Type I. We summarize the basic idea as follows.

We deduce from Theorem \ref{th:theta-product} with certain common factors extracted and powers of $(-1)$ modified that, for any $\ell\ge 1$,
\begin{align*}
	f\big((-1)^{\kappa}q^{ak},(-1)^{\kappa}q^{-ak+M\mu}\big)^{\ell}=\sum_{s=0}^{\ell-1}\sA_s \sF_{s},
\end{align*}
where each $\sF_{\star}$ is a series in $q^M$, and for $0\le s\le \ell-1$,
\begin{align*}
	\sA_s=q^{aks} f\big((-1)^{\kappa \ell}q^{ak\ell+M\mu s},(-1)^{\kappa \ell}q^{-ak\ell+M\mu (\ell-s)}\big).
\end{align*}
More generally, we consider
\begin{align}
	\sA:=q^{a\xi k} f\big((-1)^{\kappa \ell}q^{ak\ell+M\mu \xi},(-1)^{\kappa \ell}q^{-ak\ell+M\mu (\ell-\xi)}\big),\label{eq:sA-Type-II}
\end{align}
for generic $\xi\in\mathbb{Z}$.

Similarly, by Theorem \ref{th:theta-product}, we write for any $m\ge 1$,
\begin{align*}
	f\big((-1)^{\lambda}q^{bk},(-1)^{\lambda}q^{-bk+2M\mu}\big)^{m}=\sum_{t=0}^{m-1}\sB_t \sG_{t},
\end{align*}
where each $\sG_{\star}$ is a series in $q^M$, and for $0\le t\le m-1$,
\begin{align*}
	\sB_t=q^{bkt} f\big((-1)^{\lambda m}q^{bkm+2M\mu t},(-1)^{\lambda m}q^{-bkm+2M\mu (m-t)}\big).
\end{align*}
More generally, we consider
\begin{align}
	\sB:=q^{b\tau k} f\big((-1)^{\lambda m}q^{bkm+2M\mu \tau},(-1)^{\lambda m}q^{-bkm+2M\mu (m-\tau)}\big),\label{eq:sB-Type-II}
\end{align}
for generic $\tau\in\mathbb{Z}$.

With all parameters chosen in each case below, our generic target is to show the following by Corollary \ref{coro:UMH=0-J}:
\begin{align}\label{eq:AB-Type-II}
	\mathbf{H}_M\big(q^{\sigma}\sA\sB\big) =0.
\end{align}
Once the above relation is established, it is safe to conclude that
\begin{align}
	\mathbf{H}_{M}\Big(q^{\sigma}f\big((-1)^{\kappa}q^{ak},(-1)^{\kappa}q^{-ak+ M\mu}\big)^{\ell}f\big((-1)^{\lambda}q^{bk},(-1)^{\lambda}q^{-bk+2M\mu}\big)^{m}\Big)=0.
\end{align}

\subsubsection{Equation \eqref{eq:II-1}}

Recall that $(\ell,m)\mapsto (2\ell+1,2m+1)$ with $\gcd(2\ell+1,2m+1)=1$, $a=2m+1$ and $b=2\ell+1$. Also, $M=2\ell+4m+3$ and $\sigma=2(2m+1)^2 k$. Further, $(\kappa,\lambda)\in\{(0,1),(1,0)\}$ and $k$ is such that $\gcd(k,M)=1$.

In Corollary \ref{coro:UMH=0-J}, we set
\begin{align*}
	\begin{array}{rclp{3cm}rcl}
		\kappa & \mapsto & (2\ell+1)\kappa && \lambda & \mapsto & (2m+1)\lambda\\
		u & \mapsto & (2m+1)(2\ell+1)k && v & \mapsto & (2\ell+1)(2m+1)k\\
		A & \mapsto & (2\ell+1)\mu && B & \mapsto & 2(2m+1)\mu\\
		A' & \mapsto & \xi \mu && B' & \mapsto & 2\tau\mu\\
		M & \mapsto & \multicolumn{5}{l}{2\ell+4m+3}\\
		w & \mapsto & \multicolumn{5}{l}{2(2m+1)^2 k+(2m+1)\xi k+(2\ell+1)\tau k}\\
	\end{array}
\end{align*}
So,
\begin{align*}
	\sH(q)=q^{\sigma}\sA\sB.
\end{align*}
We find that
\begin{align*}
	d=\gcd(u,v)=(2\ell+1)(2m+1) k.
\end{align*}
Also, noticing that $M=(2\ell+1)+2(2m+1)$ with $\gcd(2\ell+1,2m+1)=1$, and that $k$ is coprime to $M$, we have
\begin{align*}
	d^*&=\gcd(u,v,M)=1,\\
	d_u&=\gcd(u,M)=1,\\
	d_v&=\gcd(v,M)=1.
\end{align*}
Then $d^*\mid w$. We further compute that
\begin{align*}
	\frac{Av\cdot dM}{d_u(Av^2+Bu^2)}&=2\ell+1,\\
	\frac{Bu\cdot dM}{d_v(Av^2+Bu^2)}&=4m+2,\\
	\frac{-v\kappa+u\lambda}{d}&=-(2\ell+1)\kappa+(2m+1)\lambda.
\end{align*}
Thus, \eqref{eq:div-assump-coro} and \eqref{eq:-1-assump} are satisfied by recalling that $(\kappa,\lambda)\in\{(0,1),(1,0)\}$. Finally, we choose $J=J_0 k$ in \eqref{eq:div-assump-coro-J}. Then
\begin{equation*}
	\scalebox{0.8}{%
		$
		\begin{aligned}
			\dfrac{2dMAv\cdot J-2dAvw-dAuv+dBu^2+2dA'uv-2dB'u^2-Auv^2-Bu^3}{2d(Av^2+Bu^2)}&=\dfrac{J_0-2m-1-\tau}{2m+1},\\
			\dfrac{2dMBu\cdot J-2dBuw+dAv^2-dBuv-2dA'v^2+2dB'uv+Av^3+Bu^2v}{2d(Av^2+Bu^2)}&=\dfrac{2J_0+2\ell-4m-1-\xi}{2\ell+1}.
		\end{aligned}
		$}
\end{equation*}
We may further choose $J_0$ so that
\begin{align*}
	\left\{
	\begin{aligned}
		J_0-2m-1-\tau&\equiv 0 \pmod{2m+1},\\
		2J_0+2\ell-4m-1-\xi&\equiv 0 \pmod{2\ell+1}.
	\end{aligned}
	\right.
\end{align*}
Since $\gcd(2\ell+1,2m+1)=1$, we know that the above system is solvable. So \eqref{eq:div-assump-coro-J} is also satisfied. We conclude by Corollary \ref{coro:UMH=0-J} that
\begin{align*}
	\mathbf{H}_M\big(\sH(q)\big) = 0,
\end{align*}
and therefore confirm \eqref{eq:AB-Type-II}.

\subsubsection{Equation \eqref{eq:II-2}}

Recall that $(\ell,m)\mapsto (2\ell+1,2m+2)$ with $\gcd(2\ell+1,2m+2)=1$, $a=2m+2$ and $b=2\ell+1$. Also, $M=2\ell+4m+5$ and $\sigma=2(2m+2)^2 k$. Further, $(\kappa,\lambda)\in\{(1,0),(1,1)\}$ and $k$ is such that $\gcd(k,M)=1$.

In Corollary \ref{coro:UMH=0-J}, we set
\begin{align*}
\begin{array}{rclp{3cm}rcl}
\kappa & \mapsto & (2\ell+1)\kappa && \lambda & \mapsto & (2m+2)\lambda\\
u & \mapsto & (2m+2)(2\ell+1)k && v & \mapsto & (2\ell+1)(2m+2)k\\
A & \mapsto & (2\ell+1)\mu && B & \mapsto & 2(2m+2)\mu\\
A' & \mapsto & \xi \mu && B' & \mapsto & 2\tau\mu\\
M & \mapsto & \multicolumn{5}{l}{2\ell+4m+5}\\
w & \mapsto & \multicolumn{5}{l}{2(2m+2)^2 k+(2m+2)\xi k+(2\ell+1)\tau k}\\
\end{array}
\end{align*}
So,
\begin{align*}
\sH(q)=q^{\sigma}\sA\sB.
\end{align*}
We find that
\begin{align*}
d=\gcd(u,v)=(2\ell+1)(2m+2) k.
\end{align*}
Also, noticing that $M=(2\ell+1)+2(2m+2)$ with $\gcd(2\ell+1,2m+2)=1$, and that $k$ is coprime to $M$, we have
\begin{align*}
d^*&=\gcd(u,v,M)=1,\\
d_u&=\gcd(u,M)=1,\\
d_v&=\gcd(v,M)=1.
\end{align*}
Then $d^*\mid w$. We further compute that
\begin{align*}
\frac{Av\cdot dM}{d_u(Av^2+Bu^2)}&=2\ell+1,\\
\frac{Bu\cdot dM}{d_v(Av^2+Bu^2)}&=4m+4,\\
\frac{-v\kappa+u\lambda}{d}&=-(2\ell+1)\kappa+(2m+2)\lambda.
\end{align*}
Thus, \eqref{eq:div-assump-coro} and \eqref{eq:-1-assump} are satisfied by recalling that $(\kappa,\lambda)\in\{(1,0),(1,1)\}$. Finally, we choose $J=J_0 k$ in \eqref{eq:div-assump-coro-J}. Then
\begin{equation*}
\scalebox{0.8}{%
	$
	\begin{aligned}
	\dfrac{2dMAv\cdot J-2dAvw-dAuv+dBu^2+2dA'uv-2dB'u^2-Auv^2-Bu^3}{2d(Av^2+Bu^2)}&=\dfrac{J_0-2m-2-\tau}{2m+2},\\
	\dfrac{2dMBu\cdot J-2dBuw+dAv^2-dBuv-2dA'v^2+2dB'uv+Av^3+Bu^2v}{2d(Av^2+Bu^2)}&=\dfrac{2J_0+2\ell-4m-3-\xi}{2\ell+1}.
	\end{aligned}
	$}
\end{equation*}
We may further choose $J_0$ so that
\begin{align*}
\left\{
\begin{aligned}
J_0-2m-2-\tau&\equiv 0 \pmod{2m+2},\\
2J_0+2\ell-4m-3-\xi&\equiv 0 \pmod{2\ell+1}.
\end{aligned}
\right.
\end{align*}
Since $\gcd(2\ell+1,2m+2)=1$, we know that the above system is solvable. So \eqref{eq:div-assump-coro-J} is also satisfied. We conclude by Corollary \ref{coro:UMH=0-J} that
\begin{align*}
H_M\big(\sH(q)\big) = 0,
\end{align*}
and therefore confirm \eqref{eq:AB-Type-II}.

\subsubsection{Equation \eqref{eq:II-3}}

Recall that $(\ell,m)\mapsto (2\ell+2,2m+1)$ with $\gcd(2\ell+2,2m+1)=1$, $a=2m+1$ and $b=4\ell+4$. Also, $M=4\ell+2m+5$ and $\sigma=3(2\ell+2)^2 k$. Further, $(\kappa,\lambda)\in\{(0,1),(1,1)\}$ and $k$ is such that $\gcd(k,M)=1$.

In Corollary \ref{coro:UMH=0-J}, we set
\begin{align*}
\begin{array}{rclp{3cm}rcl}
\kappa & \mapsto & (2\ell+2)\kappa && \lambda & \mapsto & (2m+1)\lambda\\
u & \mapsto & (2m+1)(2\ell+2)k && v & \mapsto & (4\ell+4)(2m+1)k\\
A & \mapsto & (2\ell+2)\mu && B & \mapsto & 2(2m+1)\mu\\
A' & \mapsto & \xi \mu && B' & \mapsto & 2\tau\mu\\
M & \mapsto & \multicolumn{5}{l}{4\ell+2m+5}\\
w & \mapsto & \multicolumn{5}{l}{3(2\ell+2)^2 k+(2m+1)\xi k+(4\ell+4)\tau k}\\
\end{array}
\end{align*}
So,
\begin{align*}
\sH(q)=q^{\sigma}\sA\sB.
\end{align*}
We find that
\begin{align*}
d=\gcd(u,v)=(2\ell+2)(2m+1) k.
\end{align*}
Also, noticing that $M=2(2\ell+2)+(2m+1)$ with $\gcd(2\ell+2,2m+1)=1$, and that $k$ is coprime to $M$, we have
\begin{align*}
d^*&=\gcd(u,v,M)=1,\\
d_u&=\gcd(u,M)=1,\\
d_v&=\gcd(v,M)=1.
\end{align*}
Then $d^*\mid w$. We further compute that
\begin{align*}
\frac{Av\cdot dM}{d_u(Av^2+Bu^2)}&=2\ell+2,\\
\frac{Bu\cdot dM}{d_v(Av^2+Bu^2)}&=2m+1,\\
\frac{-v\kappa+u\lambda}{d}&=-(4\ell+4)\kappa+(2m+1)\lambda.
\end{align*}
Thus, \eqref{eq:div-assump-coro} and \eqref{eq:-1-assump} are satisfied by recalling that $(\kappa,\lambda)\in\{(0,1),(1,1)\}$. Finally, we choose $J=J_0 k$ in \eqref{eq:div-assump-coro-J}. Then
\begin{equation*}
\scalebox{0.8}{%
	$
	\begin{aligned}
	\dfrac{2dMAv\cdot J-2dAvw-dAuv+dBu^2+2dA'uv-2dB'u^2-Auv^2-Bu^3}{2d(Av^2+Bu^2)}&=\dfrac{J_0-3\ell-3-\tau}{2m+1},\\
	\dfrac{2dMBu\cdot J-2dBuw+dAv^2-dBuv-2dA'v^2+2dB'uv+Av^3+Bu^2v}{2d(Av^2+Bu^2)}&=\dfrac{J_0-\xi}{2\ell+2}.
	\end{aligned}
	$}
\end{equation*}
We may further choose $J_0$ so that
\begin{align*}
\left\{
\begin{aligned}
J_0-3\ell-3-\tau&\equiv 0 \pmod{2m+1},\\
J_0-\xi&\equiv 0 \pmod{2\ell+2}.
\end{aligned}
\right.
\end{align*}
Since $\gcd(2\ell+2,2m+1)=1$, we know that the above system is solvable. So \eqref{eq:div-assump-coro-J} is also satisfied. We conclude by Corollary \ref{coro:UMH=0-J} that
\begin{align*}
H_M\big(\sH(q)\big) = 0,
\end{align*}
and therefore confirm \eqref{eq:AB-Type-II}.

\subsection{Seven families of coefficient functions}

All results in Corollary \ref{coro:explicit-vanishing} are direct consequences of Theorems \ref{th:Type-I}--\ref{th:Type-II} as long as we notice that $H_M\big(G(q)\big)=0$ provided that $H_M\big(G(q)\cdot F(q^M)\big)=0$ for any given series $F$ in $q^M$.

\section{Conclusion}

The entire project began in an attempt to understand the underlying patterns in the coefficient-vanishing phenomena associated with theta series. It turns out that most of the known results on this topic can be covered by several general relations.

Also of significance is the unified strategy presented in Section \ref{sec:outline}, which contains two aspects:
\begin{enumerate}[label={(\thesection .\arabic*)},leftmargin=*,labelsep=0.2cm,align=left]
	
	\item It provides an automatic way to verify if a coefficient-vanishing result, with or without free parameters, holds true. In fact, if such a result is discovered experimentally, one may try to fit it into the four models discussed in Section \ref{sec:proof}, or other models of a similar nature. The remaining task then becomes verifying some relations in connection with products of two theta series under the action of the $\mathbf{H}$-operator. Finally, such a verification relies on checking certain divisibility criteria presented in Corollary \ref{coro:UMH=0-J} and Theorems \ref{th:cancel-1} and \ref{th:cancel-2}.
	
	\item It allows us to manually construct specific or generic coefficient-vanishing results. For instance, the discovery of Theorem \ref{th:Type-II} relies only on one example or two presented in \cite{Tang2022c}. Roughly speaking, one may fix some of the free parameters in a generic model, and then use the divisibility conditions in Section \ref{sec:dissecting} to constrain the choice of other parameters so that these conditions remain valid.
	
\end{enumerate}

The divisibility criteria in Corollary \ref{coro:UMH=0-J} and Theorems \ref{th:cancel-1} and \ref{th:cancel-2} look very sharp, but they work effectively in practice, at least for most coefficient-vanishing results we had encountered. So we expect that our approach may light up a general theory. However, it should be pointed out that two ``clouds'' need to be taken into serious consideration.

First, what if one or more of the divisibility conditions fail to hold? For instance, we observe that \eqref{eq:I-2-o-e} is still valid even if the latter power $4m+2$ is replaced by a multiple of $4$. Namely, for $M=4\ell+6m+8$, $\sigma=-(2\ell+2m+3)k$, $\kappa\in\{1\}$, $\lambda\in\{0,1\}$ and any $k$ such that $\gcd(k,M)=1$, 
\begin{align}
	\mathbf{H}_{M}\Bigg(q^{\sigma}\cdot f\big((-1)^{\kappa}q^{2k},(-1)^{\kappa}q^{\mu M-2k}\big)^{2\ell+1} \bigg(\frac{f\big({-q^{2k}},-q^{\mu M-2k}\big)}{f\big((-1)^{\lambda}q^{k},(-1)^{\lambda}q^{\mu M-k}\big)}\bigg)^{4m+4}\Bigg)\overset{?}{=}0.
\end{align}
For this relation, we find that the corresponding criteria in \eqref{eq:div-assump} are no longer true. Thus, a subtle refinement of the results in Sections \ref{sec:linear-cong} and \ref{sec:dissecting} is necessary.

Second, it is easily seen that the original result of Richmond and Szekeres \cite{RS1978} related to \eqref{eq:RS} cannot be fit into our framework. Although one may transform \eqref{eq:RS} and its generalization to the summation form by Ramanujan's ${}_{1}\psi_{1}$ formula as Andrews and Bressoud \cite{AB1979} had done, an obstacle occurs due to the lack of an expansion formula for the reciprocal of a generic theta power such as those in Section \ref{sec:pairing}. Also, the coefficient-vanishing phenomenon appears in series associated with other classical summations such as the Appell--Lerch sum \cite{DX2022}. These examples suggest one extend the glimmer of our theory to the endless Galaxy.

\subsection*{Acknowledgements}

The authors would like to acknowledge the anonymous referee for the inspiring comments. In particular, a couple of typos in an earlier draft of this work were detected by the referee. Shane Chern was partially supported by a Killam Postdoctoral Fellowship from the Killam Trusts. Dazhao Tang was partially supported by the National Natural Science Foundation of China (No.~12201093), the Natural Science Foundation Project of Chongqing CSTB (No.~CSTB2022NSCQ--MSX0387), the Science and Technology Research Program of Chongqing Municipal Education Commission (No.~KJQN202200509), and the Doctoral start-up research Foundation (No.~21XLB038) of Chongqing Normal University.

\bibliographystyle{amsplain}

\end{document}